\documentclass{amsart}
\usepackage{graphicx}
\usepackage{amssymb, comment, color}
\usepackage{tikz-cd}
\usetikzlibrary{arrows}
\tikzset{>=latex}

\usepackage{subcaption}
 \usepackage{hyperref}

\usetikzlibrary{intersections,positioning}

\title[Genus zero transverse foliations for Reeb flows]{Genus zero transverse foliations for weakly convex Reeb flows on the tight $3$-sphere}

\newcommand{\C}{\mathbb{C}}
\newcommand{\D}{\mathbb{D}}
\newcommand{\R}{\mathbb{R}}
\newcommand{\Z}{\mathbb{Z}}

\newcommand{\N}{\mathbb{N}}
\newcommand{\M}{\mathcal{M}}

\newcommand{\J}{\mathcal{J}}
\newcommand{\B}{\mathcal{B}}

\renewcommand{\P}{\mathcal{P}}
\newcommand{\F}{\mathcal{F}}
\newcommand{\V}{\mathcal{V}}
\newcommand{\W}{\mathcal{W}}

\renewcommand{\sl}{\text{sl}}
\newcommand{\link}{\text{link}}
\newcommand{\wind}{\text{wind}}
\newcommand{\cz}{{\text{CZ}}}
\newcommand{\U}{\mathcal U}
\renewcommand{\S}{\mathcal S}

\newcounter{newcounter}[section]

\numberwithin{equation}{section}
\numberwithin{newcounter}{section}
\numberwithin{figure}{section}
\numberwithin{footnote}{section}

\newtheorem{dfn}[newcounter]{Definition}

\newtheorem{lem}[newcounter]{Lemma}
\newtheorem{prop}[newcounter]{Proposition}
\newtheorem{rem}[newcounter]{Remark}
\newtheorem{thm}[newcounter]{Theorem}

\begin{document}

 \author[N. V. de Paulo]{Naiara V. de Paulo}
 \address{Naiara V. de Paulo, Universidade Federal de Santa Catarina, Departamento de Matem\'atica, Rua Engenheiro Udo Deeke, 485, CEP 89065-100, Bairro Salto do Norte, Blumenau SC, Brazil}
 \email {naiara.vergian@ufsc.br} 
 
 \author[U. Hryniewicz ]{Umberto Hryniewicz}
\address{Umberto L. Hryniewicz \\ RWTH Aachen University, Jakobstrasse 2, Aachen 52064, Germany}
\email{hryniewicz@mathga.rwth-aachen.de}

\author[S. Kim]{Seongchan Kim}
  \address{Seongchan Kim, Department of Mathematics Education, Kongju National University, Kongju 32588, Republic of Korea}
  \email {seongchankim@kongju.ac.kr}

\author[P. A. S. Salom\~ao]{Pedro A. S. Salom\~ao}
\address{Pedro A. S. Salom\~ao, International Center for Mathematics, SUSTech, Shenzhen, China}
\email{psalomao@sustech.edu.cn}


\begin{abstract}
A contact form on the tight $3$-sphere $(S^3,\xi_0)$ is called weakly convex if the Conley-Zehnder index of every Reeb orbit is at least $2$. In this article, we study  Reeb flows of weakly convex contact forms on $(S^3,\xi_0)$ admitting a prescribed finite set of  index-$2$ Reeb orbits, which are all hyperbolic and mutually unlinked. We present conditions so that these index-$2$ orbits are binding orbits of a genus zero transverse foliation whose additional binding orbits have index $3$.  In addition, we show in the real-analytic case that the topological entropy of the Reeb flow is positive if the branches of the stable/unstable manifolds of the index-$2$ orbits are mutually non-coincident.   
\end{abstract}

\maketitle

\tableofcontents

\section{Introduction and main results}\label{sec_introduction}

The goal of this paper is to construct transverse foliations for Reeb flows on the tight $3$-sphere with prescribed binding orbits.
Our existence theorem uses specific assumptions that hold for energy levels of Hamiltonians with two degrees of freedom that appear in concrete problems.
Hence, we are able to study certain classes of degenerate Reeb flows on the tight $3$-sphere and attack questions in celestial mechanics.

The main motivation comes from the quest for homoclinics to the Lyapunov orbits in energy levels of the planar circular restricted $3$-body problem.  
We are particularly interested in the case where the energy is slightly above the first critical value; in this case, a Lyapunov orbit appears between the massive bodies.
The notion of chaos in the planar circular restricted $3$-body problem is historically connected to such homoclinics.
However, the existence of homoclinics to the Lyapunov orbit in this specific regime has only been proved for small mass ratios by Llibre, Mart\'inez, and Sim\'o in~\cite{LMSimo} and studied numerically in the same work.
The existence of such homoclinic orbits for arbitrary mass ratio remains open.
This problem is a point of entrance for symplectic methods since the desired homoclinic connection would follow if one can find a transverse foliation with the Lyapunov orbit as part of the binding.

\subsection{Transverse foliations and weakly convex Reeb flows} Let $\psi_t$, $t\in \R$, be a smooth flow on a smooth closed oriented $3$-manifold $M$. A {\bf transverse foliation} adapted to $(M,\psi_t)$ is a singular foliation $\F$ of $M$ satisfying the following:

\begin{itemize}

\item[(i)] The singular set $\P$ of  $\F$ consists of finitely many simple periodic orbits $P_1,\ldots,P_m\subset M$, called binding orbits. The set $\P$ is called the binding of~$\F$.

\item[(ii)]  The complement $M\setminus \cup_{P\in \P} P$ is smoothly foliated by surfaces, called the regular leaves of $\F$. Every regular leaf of $\F$ is a properly embedded surface 
$\dot \Sigma \hookrightarrow M\setminus \cup_{P\in \P} P.$ The closure $\Sigma = {\rm cl}(\dot \Sigma)$ is an 
 immersed compact surface whose boundary is formed by binding orbits in $\mathcal{P}$. The components of $\partial \Sigma$ are called the asymptotic limits of $\dot \Sigma$, and the ends of $\dot \Sigma$ are called the punctures of $\dot \Sigma$. Each $\dot \Sigma$ is transverse to the flow, and each puncture has an associated asymptotic limit.
\item[(iii)] The orientation of $M$ and the flow provide each regular leaf $\dot \Sigma$ with an orientation in such a way that trajectories intersect leaves positively. A puncture of $\dot \Sigma$ is called positive if the orientation of its asymptotic limit as the boundary of $\Sigma$ coincides with the orientation of the flow. Otherwise, the puncture is called negative.
\end{itemize}

A transverse foliation, all of whose regular leaves have genus zero, is called a genus zero transverse foliation.
This definition was introduced in \cite{HS_ICM} and is based on the finite energy foliations introduced by Hofer, Wysocki, and Zehnder in~\cite{fols}, where they prove the following remarkable result for Reeb flows on the tight three-sphere.

\begin{thm}[Hofer-Wysocky-Zehnder \cite{fols}]\label{thm_HWZ_a}
 Let $\lambda=f\lambda_0$ be a nondegenerate contact form on the tight three-sphere $(S^3,\xi_0)$. 
 Then the Reeb flow of $\lambda$ admits a genus zero transverse foliation whose binding orbits have Conley-Zehnder index equal to $1$, $2$, or $3$.  
\end{thm}

We want to construct genus zero transverse foliations adapted to Reeb flows on the tight $3$-sphere $(S^3,\xi_0)$ without assuming nondegeneracy on $\lambda$. We are particularly interested in transverse foliations whose binding contains a prescribed set of index-$2$ Reeb orbits.  

Here, $S^3=\{(x_1,x_2,y_1,y_2)\in \R^4 \mid x_1^2+x_2^2+y_1^2+y_2^2=1\}$, where $(x_1,x_2,y_1,y_2)$ are canonical coordinates on~$\R^4$. 
The (tight) contact structure $\xi_0$ on $S^3$  is the one induced by the Liouville form on $\R^4$, that is $\xi_0=\ker \lambda_0$, where $\lambda_0$ is the restriction of $\frac{1}{2} \sum_{i=1}^2 (x_i dy_i - y_i dx_i)$ to $S^3$.

Let $\lambda$ be a contact form on $(S^3,\xi_0)$, that is $\lambda = f\lambda_0$ for some  smooth function $f\colon S^3 \to (0,+\infty).$  Its Reeb vector field $R=R_\lambda$ is uniquely determined by 
$d\lambda(R,\cdot) = 0$ and $\lambda(R) = 1$, and its flow $\phi_t$  is called the Reeb flow of~$\lambda$.

A periodic orbit of $\lambda$, also called a Reeb orbit, is a pair $P=(x,T),$ where $x\colon \R \to M$ is a periodic trajectory of the Reeb flow of $\lambda$ and $T>0$ is a period of $x$. If $T$ is the least positive period of $x$,   $P$ is called simple. The Reeb orbits $(x,T)$ and $(y,T')$ satisfying $x(\R) =y(\R)$ and $T=T'$ are identified and we denote by $\P(\lambda)$ the set of equivalence classes of periodic trajectories of $\lambda$. We may also denote by $P$ the set $x(\R)\subset S^3$. The action $\int_{[0,T]}x^* \lambda$ of $P=(x,T)$  coincides with its period $T$ since $\lambda(R)=1$. A periodic orbit $P=(x,T) \in \P(\lambda)$ is said to be nondegenerate if $1$ is not an eigenvalue of the linearized map $d\phi_T:\xi_{x(0)}\to \xi_{x(0)}.$
The contact form $\lambda$ is called nondegenerate if every periodic orbit of $\lambda$ is nondegenerate.

The symplectic vector bundle $(\xi_0,d\lambda|_{\xi_0}) \to S^3$  admits a unique symplectic trivialization up to homotopy.  
Hence, the Conley-Zehnder index ${\rm CZ}(P)$ of every periodic orbit $P\in \P(\lambda)$ is uniquely determined by any such 
trivialization; see section~\ref{sec_def} for a definition of the Conley-Zehnder index. We will sometimes refer to the Conley-Zehnder index of a periodic orbit $P$ simply as the index of $P$.
Given $j\in \Z$, denote by $\P_j(\lambda)\subset \P(\lambda)$ the set of Reeb orbits with Conley-Zehnder index $j$, and denote by $\P_j^{u,-1}(\lambda)\subset \P_j(\lambda)$   the subset of index-$j$ Reeb orbits that are unknotted and have self-linking number $-1$; see section \ref{sec_def} for the definition of self-linking number.

\begin{dfn}
A contact form $\lambda$ on $(S^3,\xi_0)$ is called weakly convex if ${\rm{CZ}}(P)\geq 2$ for every $P\in \P(\lambda)$, that is $\P_j(\lambda)=\emptyset \ \forall j \leq 1$.
\end{dfn}

The following theorem is a particular case of Theorem \ref{thm_HWZ_a}.   

\begin{thm}[Hofer-Wysocki-Zehnder \cite{fols}, Siefring \cite{Si3}] \label{thm_HWZ_b} If $\lambda=f\lambda_0$ is a nondegenerate weakly convex contact form on the tight three-sphere $(S^3,\xi_0)$, then its Reeb flow admits a genus zero transverse foliation satisfying the following conditions:
\begin{itemize}
    \item[(i)] Every binding orbit has index $2$ or $3$.
    \item[(ii)] Every regular leaf is either a plane asymptotic to a binding orbit or a cylinder connecting an index-$3$ binding orbit at a positive end to an index-$2$ binding orbit at a negative end.
    \item[(iii)] If there exists only one binding orbit, then its index is $3$, and the transverse foliation determines an open book decomposition of $(S^3,\xi_0)$ whose pages are disk-like global surfaces of section.
    \item[(iv)] Every index-$2$ binding orbit bounds a pair of rigid planes whose closures form a $C^1$ two-sphere. The complement of the union of these two-spheres in $S^3$ contains an index-$3$ binding orbit in each component. The index-$3$ binding orbit admits rigid cylinders connecting it to the index-$2$ orbits at the boundary of the component. The complement of the rigid cylinders in each component is foliated by finitely many families of planes asymptotic to the index-3 binding orbit.
\end{itemize}
\end{thm}

\begin{rem}Theorem \ref{thm_HWZ_b} is discussed in \cite[section 7.2]{fols} without further details. The existence of precisely two distinct rigid planes asymptotic to a given index-$2$ binding orbit follows from standard regularity, compactness, and intersection arguments. See the discussion in  \cite[Proposition 4.1]{lemos}. If fact, the orbits that obstruct the existence of these rigid planes are the index-$1$ orbits, which do not exist by assumption. Siefring's intersection theory \cite{Si3}, see also \cite{props2}, implies that these two planes approach the binding orbit through opposite directions forming a $C^1$-embedded two-sphere.  This is explained in \cite[Appendix C]{dPS1}.
\end{rem}

The following definition was first introduced in \cite{dPS_SPJ}, motivated both by Theorem \ref{thm_HWZ_b} and the $3-2-3$ foliations studied in \cite{dPS1,dPS2}.

\begin{dfn}
 A (genus zero) transverse foliation satisfying the conditions of Theorem \ref{thm_HWZ_b} is called a {\bf weakly convex foliation} if it has at least one index-$2$ binding orbit. In particular, all regular leaves are planes and cylinders.   
\end{dfn}

 It follows from Theorem \ref{thm_HWZ_b} that every nondegenerate weakly convex Reeb flow on the tight three-sphere either admits an open book decomposition whose pages are disk-like global surfaces of section, or a weakly convex foliation whose regular leaves are planes and cylinders transverse to the flow.
It is our goal to construct weakly convex foliations for a possibly degenerate Reeb flow on the tight three-sphere with a prescribed set of index-$2$ binding orbits. 

\subsection{Main results} In the following we assume that the contact form $\lambda$ on the tight 3-sphere is weakly convex, that  $\P_2(\lambda)$ is non-empty and finite, and that every orbit in $\P_2(\lambda)$ is hyperbolic, unknotted, and has self-linking number $-1$. Moreover, we assume that the orbits in $\P_2(\lambda)$ are mutually unlinked and that their actions are small when compared to the actions of the simple orbits with higher indices. We also assume that each orbit in $\P_2(\lambda)$ does not link with any orbit in $\P_3^{u,-1 }(\lambda)$. Recall that a Reeb orbit $P'=(x',T')$, which is geometrically distinct from an unknotted periodic orbit $P$, is said to be linked with $P$ if $0 \neq [x'] \in H_1( S^3 \setminus P; \Z) \cong \Z.$ Otherwise, we say that $P'$ is not linked with $P$. The linking number between $P$ and $P'$ is denoted $\link(P,P')$.

Our main result states that under the above-mentioned hypotheses,  the Reeb flow of $\lambda$ admits a weakly convex foliation so that every orbit in $\P_2(\lambda)$ is a binding orbit. Moreover, all other binding orbits have index $3$. See also Theorem \ref{mainfef}.

\begin{thm}\label{main1}
Let $\lambda=f\lambda_0$ be a weakly convex contact form on  $(S^3,\xi_0)$. Assume that the following conditions are satisfied:
\begin{itemize}
    \item[I.] $\P_2(\lambda)=\P_2^{u,-1}(\lambda)$ is a non-empty finite set $\{P_{2,1},\ldots,P_{2,l}\}$ of mutually unlinked hyperbolic periodic orbits.

\item[II.] If $P\in \P(\lambda)$ is geometrically distinct from any orbit in $\P_2^{u,-1}(\lambda)$, then  its action is greater than the action of every orbit in $\P_2^{u,-1}(\lambda)$.



\item[III.] If $P\in \P_3^{u,-1}(\lambda),$ then $\link(P,P_{2,i}) = 0, \forall i=1,\ldots,l.$

\end{itemize}
Then the Reeb flow of $\lambda$ admits a weakly convex foliation $\F$ so that:
\begin{itemize}
\item[(i)] Every orbit in $\P_2(\lambda)$ is a binding orbit.

\item[(ii)] There are $l+1$ remaining binding orbits $P_{3,1},\ldots,P_{3,l+1}\in \P_3^{u,-1}(\lambda)$. 
    
\item[(iii)] For each $i\in \{1,\ldots,l\}$, there exists a pair of rigid planes $U_{i,1},U_{i,2}\in \F,$  so that each one of them is asymptotic to $P_{2,i}$ at its positive puncture. 
Moreover, $\S_i = U_{i,1}\cup P_{2,i}\cup U_{i,2}\hookrightarrow S^3$ is a $C^1$-embedded $2$-sphere. 

\item[(iv)] The open set $S^3 \setminus \bigcup_{i=1}^l \S_i$ has $l+1$ components $\U_1,\ldots, \U_{l+1}$. For each $j \in \{1,\ldots,l+1\}$ it holds $P_{3,j} \subset \U_j$.

\item[(v)] If $\S_i \subset \partial \U_j,$ then $\F$ contains  a  rigid cylinder $V_{j,i}\subset \U_j$ asymptotic to $P_{3,j}$ at its positive puncture and to $P_{2,i}$ at its negative puncture.

\item[(vi)] For each $j\in \{1,\ldots,l+1\}$, there exist $\tilde k_j\geq 1$ 
families of planes parametrized by the interval $(0,1)$, so that each such plane is asymptotic to $P_{3,j}$ at its positive puncture. 
Here, $\tilde k_j$ is the number of components of $\partial \U_j$. At each end, every such family of planes breaks onto a rigid cylinder $V_{j,i}$ connecting $P_{3,j}$ to some $P_{2,i}$ and a rigid plane asymptotic to $P_{2,i}$.
    
\end{itemize}

\end{thm}

\begin{rem}
    A similar result holds for general contact forms on $(S^3,\xi_0)$, not necessarily weakly convex. Indeed, let $\lambda=f\lambda_0$ be a contact form on $(S^3,\xi_0)$. Let $C=C(\lambda)>0$ be the constant given in Proposition \ref{prop_unif_C}. Suppose that there exist $l$ orbits  $P_{2,1},\ldots,P_{2,l}\in \P_2^{u,-1}(\lambda)$ as in I, and that every other Reeb orbit with index  $-1,0,1$ or $2$ has action $>C$. Under the additional conditions II and III, the conclusions of Theorem \ref{main1} still hold.
\end{rem}

Theorem \ref{main1} is inspired by Theorems \ref{thm_HWZ_a} and \ref{thm_HWZ_b}.  Our results only make non-degeneracy hypotheses on the periodic orbits in $\P_2(\lambda)$ and are intended to apply to classical problems emerging in Celestial Mechanics.



\begin{figure}[ht!!]
  \centering
  \includegraphics[width=0.6\textwidth]{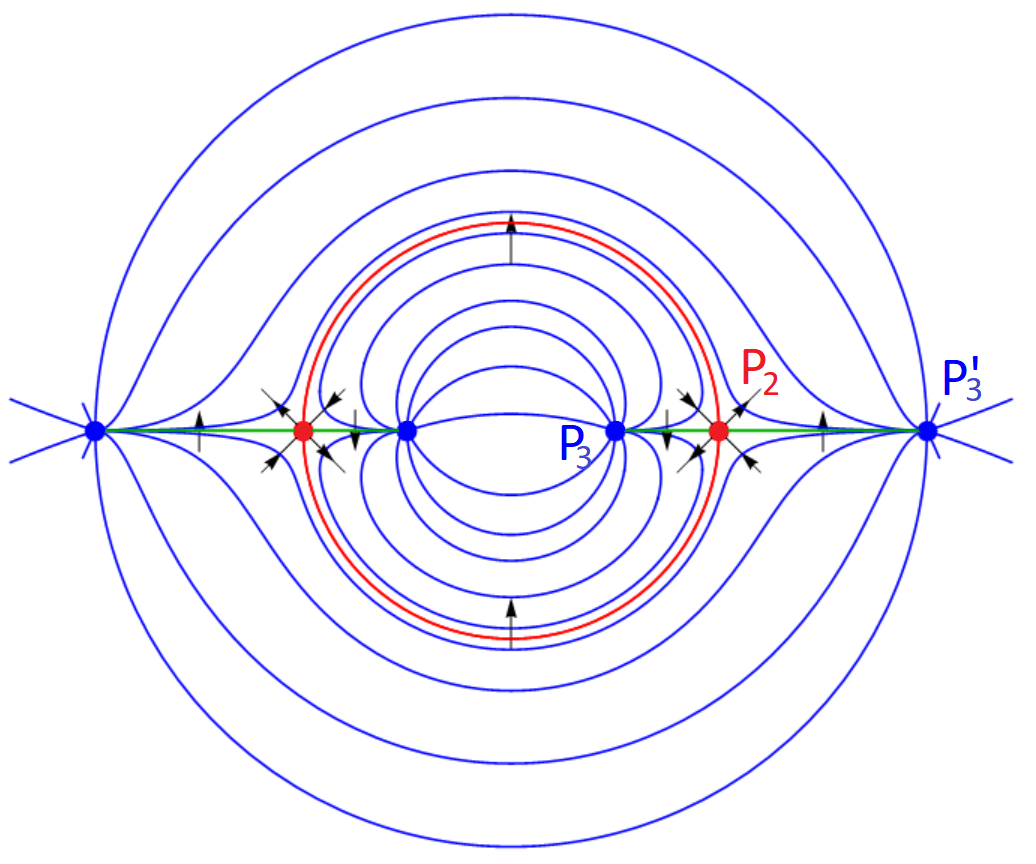}
  \caption{A weakly convex transverse foliation on $(S^3,\xi_0)$, called a $3-2-3$ foliation. The binding is formed by precisely one index-$2$ orbit  $P_2$ and two index-$3$ orbits $P_3,P_3'$. }
  \label{fif_foliation323b}
\hfill
\end{figure}


Our next result uses Theorem~\ref{main1} to study certain Reeb flows of real-analytic contact forms that appear in concrete problems.
To prepare for the statement we need to introduce some notation.

Consider a contact form $\lambda$ on $(S^3,\xi_0)$.
Let $\F$ be a weakly convex foliation adapted to the Reeb flow of $\lambda$.
Let $\U$ be a connected component of the complement of the union of the rigid spheres, and $P$ be a binding orbit in $\partial\U$.
The stable manifold $W^s(P)$ is an immersed cylinder transverse to the rigid sphere containing~$P$.
Hence, there are two well-defined local branches of $W^s(P)$ through $P$, and only one of them is contained in $\U$.
We denote by $W^s_\U(P)$ the branch of $W^s(P)$ that contains the local branch in $\U$.
The branch $W^u_\U(P)$ of the unstable manifold $W^u(P)$ is defined analogously.


\begin{thm}\label{main2}Let $\lambda=f\lambda_0$ be a real-analytic contact form on $(S^3,\xi_0)$ satisfying the conclusions of Theorem~\ref{main1}, i.e., admitting a weakly convex foliation as described in Theorem \ref{main1}. Suppose  that the actions of
the orbits in $\mathcal{P}_2(\lambda)$ coincide. 
Let $\U$ be a connected component of the complement of the rigid spheres.
Then the following statements hold:


\begin{itemize}
    
\item [(i)] 
Suppose that for every binding orbit $P$ in  $  \partial\U$ there exist binding orbits $P',P'' $ in $  \partial\U$ such that $W^s_\U(P) = W^u_\U(P')$ and $W^u_\U(P) = W^s_\U(P'')$.
Then there exists an invariant set $A\subset \U$ which admits a cross-section. This cross-section is a punctured disk bounded by the index-3 binding orbit in~$\U$, and the first return map has an infinite twist near each puncture. In particular, $A $ contains infinitely many periodic orbits. 

\item[(ii)] 
Suppose that $W^s_\U(P) \neq W^u_\U(P')$ and that $W^u_\U(P) \neq W^s_\U(P')$ holds for all pairs of binding orbits $P,P'$ in $ \partial\U$. 
Then there exists an invariant set $\Lambda\subset \U$ such that the Reeb flow restricted to $\Lambda$ has positive topological entropy.
 

\end{itemize}
\end{thm}

The assumption that the actions of the Reeb orbits in $\P_2(\lambda)$ coincide and are small when compared to the actions of the other Reeb orbits are usually verified for Hamiltonian dynamics near certain critical energy surfaces as we later explain.

The foliations obtained and used in this paper, as well as the techniques involved in the proof of Theorem~\ref{main1}, were introduced by Hofer, Wysocki and Zehnder in~\cite{fols}. In the degenerate case, such techniques were further developed in \cite{convex}.
The argument to get intersections of stable and unstable manifolds of binding orbits with Conley-Zehnder index equal to $2$ is originally found in~\cite{fols}. In the real-analytic case, we use Conley's ideas from \cite{C1}. As for the construction of a Bernoulli shift subsystem, we follow Moser's book \cite{Moser}. Genus zero transverse foliations with a single index-$2$  binding orbit were treated in \cite{dPS1,dPS2,Sa1}, see also \cite{lemos}. 
The reader finds more on transverse foliations in the surveys \cite{dPS_SPJ} and \cite{HS_ICM} and references therein.

\subsection{Sketch of proof of   Theorem \ref{main1}} Consider  sequences $\lambda_n \to \lambda$ of non-degenerate contact forms and $J_n \to J$ of generic compatible complex structures on $\ker \lambda_n = \ker \lambda$ so that for every $n$ the $\R$-invariant almost complex structure induced by $(\lambda_n,J_n)$ admits a finite energy foliation $\mathcal{F}_n$ as in Theorem \ref{thm:HWZFEF} below.  We show in Proposition \ref{prop_unif_C}  that the actions of the binding orbits of $\mathcal{F}_n$ are uniformly bounded and thus assumptions I and III imply that they consist of continuations of the $l$ orbits in $\P_2(\lambda)$, and of $l+1$ index-$3$ orbits $P^n_{3,1},\ldots P^n_{3,l+1}$. Moreover, $\mathcal{F}_n$ projects to a genus zero transverse foliation whose regular leaves are planes and cylinders asymptotic to the binding orbits, see Proposition \ref{prop_weak_conv}. Now we want to push $\mathcal{F}_n$ to a limiting finite energy foliation $\mathcal{F}$ adapted to $(\lambda,J)$. At this point, some difficulties show up in the compactness argument. It is crucial that we start with an almost complex structure $J$ which is generic enough so that some particular low energy pseudo-holomorphic curves asymptotic to the orbits in $\P_2(\lambda)$ do not exist. Here we use assumption II, see Lemma \ref{lem_genJset}. For such generic $J$'s we are able to control  the rigid planes asymptotic to the index-$2$ orbits proving that they converge to corresponding rigid planes associated with $(\lambda,J)$, see Proposition \ref{prop_rigid_planes}. The action boundedness  of the binding orbits $P_{3,1}^n,\ldots, P^n_{3,l+1}$ allows us to find distinct limiting Reeb orbits $P_{3,1}^0,\ldots,P_{3,l+1}^0\in \P_3^{u,-1}(\lambda)$ in the complement of the rigid planes. Now some undesired low action Reeb orbits of $\lambda$ that are unlinked with any of those index-$3$ orbits may obstruct the existence of a transverse foliation with binding orbits $P_{3,1}^0,\ldots, P_{3,l+1}^0$ and the orbits in $\P_2(\lambda)$.  To overcome this difficulty we may need to re-start the procedure of taking sequences $\lambda_n$,~$J_n$ as above so that the previously found index-$3$ orbits and an undesired unlinked orbit are also Reeb orbits of $\lambda_n$ for every $n$. Then we obtain new limiting index-$3$ Reeb orbits $P^1_{3,1},\ldots,P^1_{3,l+1}$ that necessarily link with the previous unlinked orbit. Repeating the aforementioned procedure of taking new sequences $\lambda_n$,~$J_n$ that freeze not only an eventual new unlinked Reeb orbit of $\lambda$, but also the previously frozen Reeb orbits, we show that the process must terminate after finitely many steps. Hence we eventually find special index-$3$ orbits $P_{3,1},\ldots, P_{3,l+1}\in \P_3^{u,-1}(\lambda)$ that are necessarily linked with all Reeb orbits that are not covers of the orbits in $\P_2(\lambda)$, see Proposition~\ref{prop_specialP3}. These special orbits and the orbits in $\P_2(\lambda)$  are candidates for binding orbits of the desired foliation. We then take limits of the planes asymptotic to the index-$3$ orbits and of the rigid cylinders connecting the index-$3$ to the index-$2$ orbits, to obtain a genus zero transverse foliation. This foliation, however, may include leaves with more than one negative puncture asymptotic to distinct orbits in $\P_2(\lambda)$. Using uniqueness of such pseudo-holomorphic curves, we show in Proposition \ref{prop_family_planes1} the existence of at least one plane asymptotic to each index-$3$ orbit. After slightly changing the almost complex structure near the rigid planes to rule out such non-generic curves with multiple negative ends at index-$2$ orbits, we obtain the desired foliation from the compactness properties of the planes asymptotic to the index-$3$ orbits and an application of the gluing theorem.

\subsection{Potential applications}


Let $H \colon \R^4 \to \R$ be a smooth Hamiltonian. Assume that $0\in \R$ is a critical value of $H$ and that the critical set $H^{-1}(0)$ contains finitely many saddle-center equilibrium points $p_1,\ldots, p_m$.
Assume that as the energy changes from negative to positive, a sphere-like component $C_E\subset H^{-1}(E), E<0$,  gets connected to other components of the energy surface, precisely at $p_1,\ldots, p_m$. In particular, $C_E$ becomes a singular sphere-like subset $C_0 \subset H^{-1}(0)$ with singularities at $p_1,\ldots,p_m$. For every $E>0$ small, $H^{-1}(E)$ contains an index-$2$ hyperbolic orbit $\gamma_i$ in the neck-region about $p_i, i=1, \ldots, m.$  The orbit $\gamma_i$, called the Lyapunoff orbit,  bounds a pair of planes in the energy surface, which are transverse to the flow and form with $\gamma_i$ a $2$-sphere $\mathcal{S}_i \subset H^{-1}(E)$,~$\forall i$.  We may wish to study the dynamics on the subset $S_E \subset H^{-1}(E)$ near $C_0$, which is bounded by the union of the $2$-spheres $\mathcal{S}_i$. The subset $S_E$ is diffeomorphic to a $3$-sphere with $m$ disjoint $3$-balls removed.  In many situations, $S_E$ is the region in $H^{-1}(E)$ where the interesting dynamics takes place. We thus may assume, possibly after changing $H$ away from $S_E$, that  the components which get connected to $C_E$ are formed by $m$ suitable sphere-like hypersurfaces. In particular, $S_E, E>0$ small, is a subset of a non-convex sphere-like subset $\hat S_E \subset H^{-1}(E)$. The Hamiltonian flow on $\hat S_E$ is equivalent to a Reeb flow on the tight $3$-sphere. If $H$ satisfies some mild convexity conditions on $C_0$, then, for energies $E>0$ sufficiently small, $\hat S_E$ satisfies the hypotheses of Theorem \ref{main1}. Indeed, $\hat S_E$ admits no periodic orbits with index $\leq 1$, and index-$3$ orbits do not intersect the $2$-spheres $\mathcal{S}_i$. The transverse foliation given in Theorem \ref{main1} restricts to a transverse foliation on $S_E$ that contains the Lyapunoff orbits and an index-$3$ orbit as binding orbits. Dynamical properties such as multiplicity of periodic orbits and the existence of homoclinics/heteroclinics to the Lyapunoff orbits follow from the transverse foliation. If $H$ satisfies a particular $\Z_m$-symmetry, then Theorem \ref{main2} applies,  deriving more information about the dynamics. As an example, the H\'enon-Heiles system is $\Z_3$-symmetric and thus the three index-$2$ Lyapunoff orbits exist and have the same arbitrarily small action for energies slightly above its  critical value $\frac{1}{6}$.   We will explore this example further in \cite{mechanical}.  The circular planar restricted three-body problem is  $\Z_2$-symmetric and fits a similar setting after regularizing collisions with one of the primaries. We expect to find applications of Theorems \ref{main1} and \ref{main2} for certain mass ratios and energies slightly above the first Lagrange value. 
\bigskip

\noindent
{\bf Acknowledgements.} 
UH is partially supported by the DFG SFB/TRR 191 ‘Symplectic Structures in Geometry, Algebra and Dynamics’, Projektnummer 281071066-TRR 191.
SK is supported by the National Research Foundation of Korea (NRF) grant funded by the Korea government (MSIT) (No.\ NRF-2022R1F1A1074066) and the research grant of Kongju National University in 2022 (Grant number: 2022-0174-01). 
PS acknowledges the support of NYU-ECNU Institute of Mathematical Sciences at NYU Shanghai and the 2022 National Foreign Experts Program. 
PS is partially supported by FAPESP 2016/25053-8 and CNPq 306106/2016-7.

\section{Basic Definitions}\label{sec_def}

Let $M$ be a closed three-manifold equipped with a contact form $\lambda$.
 Let $P=(x,T)\in \P(\lambda),$ $ x_T:=x(T\cdot)\colon \R / \Z \to M$ and  $J\in \mathcal{J}(\lambda)$, where  $\mathcal{J}(\lambda)$ denotes the set of $d\lambda$-compatible almost complex structures on the contact structure  $\xi=\ker \lambda$. The asymptotic operator $A_P=A_{P,J}$ is the  unbounded self-adjoint operator acting on  sections of $\xi$ along $P$ 
$$
A_{P}(\eta) := -J|_{x_T} \cdot  \mathcal{L}_{\dot x_T}\eta\in L^2(x_T^* \xi),$$ for every $\eta\in W^{1,2}(x_T^*\xi),$  
where $\mathcal{L}_{\dot x_T} \eta$ is the Lie derivative of $\eta$ in the direction of $\dot x_T$
$$
\begin{aligned}
(\mathcal{L}_{\dot x_T} \eta)(t)  :=  \frac{d}{ds}\bigg|_{s=0} \left\{ D\varphi^{-1}_{sT}(x_T(t+s)) \cdot \eta(t+s) \right\}.
\end{aligned}
$$
The eigenvalues of $A_P$  are real and accumulate only at $\pm \infty$. A non-trivial eigenvector never vanishes and thus, for a  fixed  frame $\Psi$ of the contact structure along $P$,  determines  a winding number $\wind_\Psi (\mu)\in \Z$ that depends only on the eigenvalue $\mu$. The function $\mu \mapsto \wind_\Psi(\mu)$ is monotonically  increasing and surjective and satisfies $\#\wind_\Psi^{-1}(k)=2, \forall k\in \Z,$ where the multiplicities are counted.  See \cite[Section 3]{props2}.

The periodic orbit $P$ is degenerate if and only if $0$ is an eigenvalue of $A_{P}.$ Denoting by $\wind_\Psi^{< 0}(A_P)$ the winding number of the largest negative eigenvalue and by  $\wind_\Psi^{\geq 0}(A_P)$ the winding number of the smallest non-negative eigenvalue, the (generalized) Conley-Zehnder index of $P$ with respect to the frame $\Psi$ is defined as
$
\cz_\Psi(P) := \wind_\Psi^{< 0}(A_P)+\wind_\Psi^{\geq 0}(A_P).
$
It depends on the frame $\Psi$ but not on $J$. If $(M,\xi) = (S^3,\xi_0)$, we fix a frame induced by a global trivialization of $\xi$ and omit $\Psi$ in the notation.

Now let  $x\colon  \R / \Z \to (S^3,\xi_0)$ be an unknot transverse to the standard contact structure $\xi_0$,  and let $u \colon \D \to S^3$ be an embedded disk, where $\D=\{z\in \C:|z|\leq 1\}$.    Assume that $u$ is a spanning disk for $x$, that is $x(t) = u( e^{ 2 \pi i t} ), \forall t \in \R / \Z$. Choose a non-vanishing section $X$ of the pullback bundle $u^* \xi _0 \to \D$.  Fix a Riemannian metric $g$ on $S^3$ and denote by exp the associated exponential map. If $X$ is sufficiently small, then $x_X(\R / \Z) \cap x(\R / \Z)= \emptyset,$ 
where  $x_X(t) := \exp_{x(t )}X(t),\forall t \in \R / \Z$. We may assume that $x_X$ is transverse to $u$. The \emph{self-linking number} of $x$, denoted by $\sl(x)$, is defined as the algebraic intersection number   $x_X \cdot u$. Here, $S^3$ is oriented in such a way that  $\lambda \wedge d \lambda >0$. The orientation of $x_X$ is the one induced by $x$, and $x$ is oriented as the boundary of $u$.   The self-linking number is independent of the involved choices.

\subsection{Pseudo-holomorphic curves}

In this section, $\lambda$ is a contact form on a smooth closed three-manifold $M$ and $\xi=\ker\lambda$ is the induced contact structure. The Reeb vector field $R=R_\lambda$ is defined as before. 
The symplectization of $M$ is  the symplectic manifold $( \R \times M, d(e^r \lambda))$, where $r$ is the $\R$-coordinate.  
For each $J \in \mathcal{J}(\lambda)$, 
the pair $(\lambda,J)$  determines an $\R$-invariant  almost complex structure $\widetilde J$ on $\R \times M$ so that
$
\widetilde J \cdot \partial_r = R$ and  $\widetilde J|_\xi = J.
$
Let $(S, j)$ be a closed connected Riemann surface, and let $\Gamma \subset S$ be a finite set of punctures. We consider finite energy $\widetilde J$-holomorphic curves in $\R \times M$, i.e.\ smooth maps $\tilde u=(a,u) \colon S \setminus \Gamma \to \R \times M$ satisfying
$ d\tilde u \circ j = \widetilde J(\tilde u) \circ d\tilde u,
$
and having finite Hofer's energy 
$0<E(\tilde{u}) = \sup_{\psi \in \Lambda}  \int_{S \setminus \Gamma} \tilde{u}^*d ( \psi(a) \lambda) <+\infty,
$
  where $\Lambda$ is the set of non-decreasing smooth functions $\psi \colon \R \to [0,1]$. If $S=S^2$ and $\# \Gamma=1$, then  $\tilde{u}$ is called a finite energy plane.


\begin{dfn}
A $\widetilde{J}$-holomorphic map $\tilde{u} \colon S \setminus \Gamma \to \R \times M$ is called somewhere injective if there exists $z_0 \in S \setminus \Gamma$ such that $d \tilde{u}(z_0) \neq 0$ and $\tilde{u}^{-1}( \tilde{u}(z_0)) = \{z_0 \}.$
\end{dfn}

 Given a puncture $z_0 \in \Gamma$, choose a holomorphic chart $\phi \colon (\D\setminus \partial \D, 0) \to (\phi(\D\setminus \partial \D), z_0)$ centered at $z_0$. The map $\tilde{u} \circ \phi( e^{- 2 \pi (s+ it)}), (s,t) \in [0,+\infty) \times \R / \Z,$ will be still  denoted by $\tilde u=(a,u)$. The finite energy condition implies that $\tilde u$ is non-constant and the limit
$
m= m_{z_0}:= \lim_{s \to +\infty} \int_{\{s\}\times S^1} u^*\lambda
$
exists. The puncture $z_0$ is said to be removable if $m=0$. In this case, $\tilde{u}$ can be smoothly extended over $z_0$ by Gromov's removable singularity theorem. The puncture $z_0$ is called positive or negative if $m>0$ and $m<0$, respectively. In the following, we tacitly assume that all punctures are non-removable. Stokes' theorem tells us that the set $\Gamma$ is non-empty.  We assign the sign $\varepsilon=\pm 1$ to each puncture, depending on whether it is positive or negative, respectively.  This induces a decomposition $\Gamma = \Gamma_+ \cup \Gamma_-$.

\begin{thm}[Hofer {\cite[Theorem 31]{Hofer93}}]\label{thm_asymp_limit}
Let $z_0 \in \Gamma$ be a non-removable puncture of a finite energy pseudo-holomorphic curve $\tilde{u}=(a, u) \colon S \setminus \Gamma \to \R \times M$. Let $(s,t) \in [0, +\infty) \times \R / \Z$ be holomorphic polar coordinates centered at $z_0$  as above, and set $\tilde u(s,t)=(a(s,t),u(s,t))$. Let $\varepsilon \in \{-1,1\}$ be the sign of $z_0$. Then, for every sequence $s_n \to +\infty$, there exist a subsequence $s_{n_k}$ of $s_n$ and a periodic   orbit $P=(x,T)\in \P(\lambda)$   so that $u(s_{n_k}, \cdot) \to x(\varepsilon T\cdot)$ in $C^{\infty}(\R / \Z, M)$ as $k \to +\infty$. 
\end{thm}

The periodic orbit in the previous statement is referred to as an asymptotic limit of $\tilde{u}$  at $z_0 \in \Gamma$. Denote the set of asymptotic limits of $\tilde{u}$ at  $z_0$ by $\Omega = \Omega(z_0)\subset \P(\lambda)$. This set is non-empty, compact and connected. See for instance   \cite[Lemma 13.3.1]{FvKbook}. Explicit examples of finite energy curves with the image of $\Omega$ being diffeomorphic to the two-torus is provided by Siefring in  \cite{Si3}.

The following theorem due to Hofer, Wysocki  and Zehnder tells us that if an asymptotic limit of $\tilde{u}$ at $z_0 \in \Gamma$ is non-degenerate, then  $\Omega$ consists of a single Reeb orbit. Moreover, $\tilde{u}$ has exponential convergence to the  asymptotic limit.

\begin{thm} [Hofer-Wysocki-Zehnder {\cite{props1}}]  \label{thm:asymptoticbehaviour}
Let $z_0 \in \Gamma$ be a non-removable puncture of a finite energy pseudo-holomorphic curve $\tilde{u}=(a, u) \colon S \setminus \Gamma \to \R \times M$. Choose holomorphic polar coordinates $(s,t) \in [0,+\infty) \times \R / \Z$ near $z_0$, and set $\tilde{u}(s,t) = (a(s,t), u(s,t))$ as before. 
Assume that $P=(x,T)\in \P(\lambda)$ is a non-degenerate asymptotic limit of $\tilde{u}$ at $z_0$. 
Then there exist  $c,d \in \R$ such that
\begin{enumerate}
\item[(i)] $\sup_{t \in S^1} \lvert a(s,t) - \varepsilon Ts -d \rvert  \to 0$ as $s \to+\infty$.
\item[(ii)] $u(s,\cdot) \to x(\varepsilon T\cdot+c)$ in $C^\infty(\R/\Z, M)$ as $s \to+\infty$.
\item[(iii)]  let $\pi \colon TM \to \xi$ be the projection along the Reeb direction. If $\pi \circ du$ does not vanish identically near $z_0$, then $\pi \circ du(s,t) \neq 0$ for every $s \gg0.$ 
\item[(iv)] define $\zeta(t,s)$ by $u(s,t) = \exp_{x(\varepsilon Tt+c)}\zeta(\varepsilon Tt+c,s), \forall t \in \R /\Z$. Then there exist an eigenvalue $\mu$ of $A_{P}$, with $\varepsilon \mu<0$, and a $\mu$-eigenfunction $e(t)\in \xi|_{x(Tt)},  t\in \R/ \Z,$ so that
$
\zeta(t,s) = e^{\varepsilon \mu s}\left( e(t) + R_1(s,t) \right),  s\gg 0,
$
where the remainder $R_1$ and its derivatives converge to $0$, uniformly in $t$,  as $s \to +\infty$.
\end{enumerate}
\end{thm}
If $\tilde u$ admits a single asymptotic limit $P$ at $z_0\in \Gamma$ and the asymptotic behavior of $\tilde u$ near $z_0$ is as in Theorem \ref{thm:asymptoticbehaviour} for a non-vanishing eigenvalue $\mu$ of $A_P$, then we say that $z_0$ is a non-degenerate puncture of $\tilde u$ and that $\tilde u$  exponentially converges to $P$ at $z_0$.
In the case  every asymptotic limit of a finite energy pseudo-holomorphic curve $\tilde{u} = (a,u)$ is non-degenerate, its Conley-Zehnder index and Fredholm index are defined as
$
\mathrm{CZ}(\tilde{u}) := \sum_{z \in \Gamma_+} \mathrm{CZ}(P_z) - \sum_{z \in \Gamma_-} \mathrm{CZ}(P_z),
$
where $P_z$ is the asymptotic limit of $\tilde{u}$ at $z \in \Gamma$, and
$
\mathrm{Ind}(\tilde{u}) := \mathrm{CZ}(\tilde{u} )- \chi(S) + \# \Gamma,
$
respectively. Here, $\cz$ is computed in a frame along the asymptotic limits induced by a trivialization of $u^*\xi$. The integer $\cz(\tilde u)$ does not depend on this trivialization. Moreover, one can define the following algebraic invariants: suppose that $\pi \circ du$ does not vanish identically. Then  
 $
 \int_{S \setminus \Gamma} u^* d \lambda >0,
 $
 and Carleman's similarity principle tells us that the zeros of $\pi \circ du$ are isolated. The asymptotic behavior of $\tilde{u}$ described in Theorem \ref{thm:asymptoticbehaviour} implies that  $ \pi \circ du$ does not vanish near the punctures. It follows that the zeros of $\pi \circ du$ are finite, and  each zero of $\pi \circ du$ has a positive local degree. We define
 $
 \mathrm{wind}_{\pi}(\tilde{u}) := \# \{ \text{zeros of } \pi \circ du \} \geq 0, 
 $
 where the zeros are counted with multiplicity. 
 
 Let $z \in \Gamma$ be a puncture, and let $P=(x,T)\in \P(\lambda)$ be the asymptotic limit of $\tilde{u}$ at $z$. Let $e\in \Gamma(x_T^*\xi)$ be the associated eigenfunction as in Theorem \ref{thm:asymptoticbehaviour}-(iv).  Define the winding number $\mathrm{wind}_\infty (\tilde u; z)$ of $\tilde{u}$ at $z$ to be the winding number of $t\mapsto e(t)$ in a frame of $x_T^*\xi$ induced by a trivialization  of $u^*\xi$. The winding number of $\tilde{u}$ is then defined as
$
\mathrm{wind}_\infty(\tilde{u}) = \sum_{z \in \Gamma_+} \mathrm{wind}_\infty (\tilde u; z)-  \sum_{z \in \Gamma_-} \mathrm{wind}_\infty (\tilde u; z).
$
This integer does not depend on the choice of trivialization. The two winding numbers are related by
$
\mathrm{wind}_\pi (\tilde{u}) = \mathrm{wind}_\infty(\tilde{u})  -\chi(S) + \# \Gamma,
$
see \cite[Proposition 5.6]{props2}.

\begin{dfn}[{Siefring \cite{Si2}}]\label{throughopposite}
Let $\tilde{u},\tilde{v}$ be finite energy planes which are exponentially asymptotic to the same periodic orbit $P \in \mathcal{P}(\lambda).$ Let $e_+, e_-$ be the respective eigensections of $A_P$ that describe $\tilde{u}, \tilde{v}$ near $\infty$. We say that $\tilde{u}$ and $\tilde{v}$ are asymptotic to $P$ through opposite directions $($resp. through the same direction$)$ if $e_+ = c e_-$ for some $c<0$ $($resp. for some $c>0)$.
\end{dfn}

\subsection{Finite Energy Foliations}

Suppose that the contact form $\lambda$ on $(S^3,\xi_0)$ is non-degenerate, and choose $J \in \mathcal{J}  (\lambda)$.
 A stable finite energy foliation for $(S^3, \lambda, J)$ is a two-dimensional smooth foliation $\tilde{\mathcal{F}}$ of $ \R \times S^3$ having the following properties:
\begin{enumerate}

\item  Every leaf $\tilde{F} \in \tilde{\mathcal{F}}$ is the image of an embedded finite energy $\tilde{J}$-holomorphic sphere $ \tilde{u}_{\tilde{F}} = (a_{\tilde{F}}, u_{\tilde{F}})$ that has precisely one positive puncture but an arbitrary number of  negative punctures.  The energies of such finite energy spheres are uniformly bounded.

\item The asymptotic limits of every $\tilde{F} \in \tilde{\mathcal{F}}$, defined as the asymptotic limits of $\tilde{u}_{\tilde{F}}$,   have Conley-Zehnder indices belonging to the set $\{1,2,3\}$ and self-linking number $-1$.

\item There exists an $\R$-action $ T \colon \R \times \tilde{\mathcal{F}} \to\tilde{\mathcal{F}}$, defined by  
$
T_c(\tilde{F}):=T(c, \tilde{F})= \{ ( c+r, z ) \mid (r,z ) \in \tilde{F} \}\in \tilde \F,  \forall c, 
$
so that if $\tilde{F} \in \tilde{\mathcal{F}}$ is not a fixed point of $T$, then   $ \mathrm{Ind}(\tilde{u}_{\tilde{F}})  \in \{1,2\}$, and $u_{\tilde{F}}$ is an embedding, transverse to the Reeb vector field. If $\tilde{F}$ is a fixed point of $T$, i.e.\ $T_c(\tilde{F}) = \tilde{F}$, $\forall c \in \R$, then $\mathrm{Ind}(\tilde{u}_{\tilde{F}}) =0$, and $\tilde{u}_{\tilde{F}}$ is a cylinder over a periodic orbit.  
\end{enumerate}

The following statement on the existence of a stable finite energy foliation is due to Hofer, Wysocki and Zehnder.

\begin{thm}[Hofer-Wysocki-Zehnder \cite{fols}]\label{thm:HWZFEF} Let $\lambda$ be a non-degenerate contact form on the tight three-sphere $(S^3,\xi_0)$. There exists a residual subset $\mathcal{J}_{\mathrm{reg}}(  \lambda) \subset \mathcal{J} (  \lambda)$, in the $C^{\infty}$-topology,    such that every $J \in  \mathcal{J}_{\mathrm{reg}}(  \lambda)$ admits a stable finite energy foliation $\tilde{\mathcal{F}}.$  Moreover, the projected foliation $\mathcal{F}:=p(\tilde{\mathcal{F}})$, where $p \colon \R \times S^3 \to S^3$ denotes the projection onto the second factor, is a genus zero transverse foliation adapted to the flow of $\lambda$.    If the binding $\P$ consists of a single periodic orbit $P,$   then $\mathrm{CZ}(P)=3$ and the foliation $\F$ provides an open book decomposition of $S^3$ whose pages are disk-like global surfaces of section. If $\# \P \geq 2$ and the binding has no periodic orbit with $\cz=1$, then the foliation $\mathcal{F}$ is a weakly convex foliation satisfying the following properties:

\begin{enumerate}

\item[(i)] The binding $\P$ consists of $\ell$ periodic orbits $P_{2,1}, \ldots, P_{2,\ell}$ with Conley-Zehnder index  $2$  and $\ell+1$ periodic orbits $P_{3,1}, \ldots, P_{3,\ell+1}$ with Conley-Zehnder index $3$.   The binding orbits are unknotted and  mutually unlinked  and have self-linking number $-1$.  The   orbits $P_{2,1}, \ldots, P_{2,\ell}$ are hyperbolic.

\item[(ii)] For every $P_{2,i}$, there exists a pair of rigid planes $U_{i,1}, U_{i,2} \in \F$ both asymptotic to $P_{2,i}$ through opposite directions. The set $\mathcal{S}_i:=U_{i,1}\cup P_{2,i}\cup U_{i,2}$ is a $C^1$-embedded $2$-sphere which separates $S^3$ into two components. In particular, $S^3 \setminus \cup_i \mathcal{S}_i$ has $l+1$ components, denoted by $\mathcal{U}_j,j=1,\ldots,l+1$. 

\item[(iii)] Each $P_{3,j}$ is contained in $\mathcal{U}_j$ and spans $\tilde k_j$  one-parameter families of planes parametrized by the interval $(0,1)$. The integer  $\tilde k_j$ coincides with the number of components of $\partial \U_j$. At their ends, the families breaks onto a cylinder connecting $P_{3,j}$ to some $P_{2,i}\subset \partial \U_j$ and a plane asymptotic to $P_{2,i}$. 

 \end{enumerate}
 \end{thm}

Let $\lambda=f\lambda_0$ be a possibly degenerate contact form on   $(S^3, \xi_0)$,  and let $J \in \mathcal{J}(  \lambda)$. As proved in \cite[Proposition 6.1]{convex}, there exists a sequence of non-degenerate  contact forms $\lambda_n=f_n \lambda_0$ on $(S^3,\xi_0)$ such that $\lambda_n \to \lambda$ in $C^{\infty}$ as $n\to +\infty$. Since $\mathcal{J}_{\rm reg}(\lambda_n),$ given in  Theorem \ref{thm:HWZFEF}, is dense in $\mathcal{J}(\lambda_n)=\mathcal{J}(\lambda)$, for each large $n$ one can find an almost complex structure $J_n\in\mathcal{J}_{\rm reg}(\lambda_n)$ such that $J_n \to J$ in $C^{\infty}$ as $n \to +\infty$ and, moreover, each $(\lambda_n,J_n)$ admits a genus zero transverse foliation $\F_n$ which is the projection to $S^3$ of a stable finite energy foliation on $\R \times S^3$. The following proposition will be useful in our argument later on.

\begin{prop}\label{prop_unif_C}
Let $\lambda_n \to \lambda$ and $\ J_n \to J$ as $n\to \infty$,  and let $\F_n$ be a genus zero transverse foliation associated with $(\lambda_n,J_n)$ as in Theorem \ref{thm:HWZFEF}. Then there exists a universal constant $C>0$, depending only on $\lambda$, such that the binding orbits of $\F_n$ have action less than $C$ for every large $n$. 
\end{prop}
\begin{proof} The assertion follows from the construction of the foliations $\F_n$ due to Hofer, Wysocki  and Zehnder \cite{fols}. In fact, the action of the binding orbits depend only on the $C^0$-norm of $f$. The uniform upper bound for every large $n$ then follows. \end{proof} 

\subsection{Topological entropy}\label{sec:entro}

Let $X$ be a nowhere vanishing vector field on a closed manifold $M$. Abbreviate by $\psi_t$ the flow of $X$.  We fix a metric $d$ that generates the topology of $M$. For every $T>0$, we define
$
d_T(x,y) := \max_{ t \in [0,T]} d( \psi_t(x), \psi_t(y) ),$ $\forall x, y \in M.
$
Fix $\varepsilon >0$. A subset $U$ is said to be $(T,\varepsilon)$-separated if $d_T(x,y) \geq \varepsilon$ for every  $x\neq y \in U$. 
Let $N(T,\varepsilon)$ denote the maximal cardinality of a $(T,\varepsilon)$-separated set. The  topological entropy $h_{\mathrm{top}}(\psi_t)$ of the flow $\psi_t$ is defined to be the growth rate of $N(T,\varepsilon)$:  
\[
h_{\mathrm{top}}(\psi^t):= \lim_{\varepsilon \to 0^+} \limsup _{T \to + \infty} \frac{1}{T} \log N(T, \varepsilon).
\]
It is well-known that the topological entropy of a smooth flow is finite.
For more details on   topological entropy, we refer to \cite{HKentropy, Paternainentropy}.

\begin{rem}

 The topological entropy $ {h}_{\rm top}(f)$ of a continuous map $f$ on a compact  Hausdorff metric space $(M,d)$ is defined  in the same way as above, with 
 \[
 d_n(x,y) := \max \{d (f^k(x), f^k(y)) \mid k=0, 1, \ldots, n-1\}, \quad  \forall x, y \in M, \; n \in \N.
 \]
 If $M$ is a smooth manifold whose topology is determined by the metric $d$ and if the flow $\psi_t$ of a nowhere vanishing vector field $X$ is smooth, then  $h_{\rm top}(\psi_t) =  {h}_{\rm top}(\psi_1)$. See \cite[Proposition 3.1.8]{HKentropy}.
\end{rem}

The following theorem due to Katok \cite{Katokentropy} relates topological entropy to periodic orbits.
\begin{thm}[{Katok}]\label{thm:Katok}
Let $\psi_t$ be the flow of a nowhere vanishing vector field $X$ on a  closed three-manifold and let $P_T(\psi_t)$ denote the number of periodic orbits of $\psi_t$ with period smaller than $T>0$. 
If $h_{\mathrm{top}} (\psi_t)>0$, then
$
\limsup_{T \to + \infty} \frac{ \log P_T(\psi_t)}{T} >0.
$
\end{thm}

 A standard way to detect chaotic behavior of a flow  is to build a  Bernoulli shift.  In order to recall its definition, let $A= \{1, \ldots, N \}$ be a finite alphabet. The set $\Sigma_N$ consisting of all doubly infinite sequences  
 $a = (a_j)_{j \in \Z}, a_j \in A,$ 
 is equipped with the metric
 \begin{equation*}\label{eq:metrix:bernoulli}
 d(a,b) = \sum_{j \in \Z} \frac{1}{2^{\lvert j \rvert}} \frac{ \lvert a_j - b_j \rvert }{1 + \lvert a_j - b_j \rvert}, \quad \forall a, b \in \Sigma_N,
 \end{equation*}
 which makes   $(\Sigma_N, d)$ a compact Hausdorff metric space. 
The  Bernoulli shift on $\Sigma_N$ is the homeomorphism $\sigma \colon \Sigma_N \to \Sigma_N,$ defined as
$\sigma(a) = ( \sigma(a)_j)_{j \in \Z},$ where $\sigma(a)_j :=a_{j+1}.$

 A homeomorphism $\phi$ on a compact Hausdorff metric space $\Lambda$ is said to be {semi-conjugate} to a Bernoulli shift  if there exists a continuous surjective map $\tau \colon \Lambda \to \Sigma_N $   for some $N \geq 2$   such that $\tau \circ \phi = \sigma \circ \tau$. 
 %


   Let $Q\subset \R^2$ be the unit square $[0,1]\times [0,1]$. We denote its right, left,  upper and lower edges by $V_0, V_\infty, H_0$ and $H_\infty$, respectively.  The compact region bounded by two disjoint and vertically monotone curves connecting $H_0$ to $H_\infty$ is called a vertical strip in $Q$.  Similarly, one defines a horizontal strip in $Q$. 
  
 We refer the reader to     \cite[Chapter III]{Moser} for the proof of the following statement.

 \begin{prop}\label{prop:embedbernoullishift}
 Let $\phi\colon Q \to \R^2$ be a mapping that satisfies the following:  
  \begin{enumerate}
 \item[(N1)]  In the square $Q$, there exist disjoint vertical strips $V_1, \ldots, V_N$ and disjoint horizontal strips $H_1, \ldots, H_N$ such that $\phi(H_i) = V_i$ for every $i=1,\ldots, N.$ The vertical strips and horizontal strips are ordered from  right to left and from  top to bottom, respectively. 
  \item[(N2)]  If $V\subset Q$ is a vertical strip, then for each $i$, the set $\phi (V) \cap V_i$ contains a vertical strip. Similarly, if $H \subset Q$ is a horizontal strip, then $\phi^{-1}(H) \cap H_i$ contains a horizontal strip for every $i$. 
 \end{enumerate} 
  Then there exists a compact invariant set $\Lambda \subset Q$ such that   $\phi|_{\Lambda}$ is semi-conjugates to a Bernoulli shift with $N$ symbols. 
Consequently, $h_{\rm{top}}(\phi)>0$.
 \end{prop}
Proposition \ref{prop:embedbernoullishift} generalizes to the case of  countably many disjoint vertical strips $V_1,V_2,\ldots,$  considering an alphabet with countably many symbols.

\section{Proof of Theorem \ref{main1}}\label{sec:3}

Let $\lambda$ be a weakly convex contact form on $(S^3,\xi_0)$ and let $\P_2(\lambda)=\{P_{2,1},\ldots,P_{2,l}\}$ be a finite set of index $2$ periodic orbits satisfying the hypotheses of Theorem \ref{main1}. 

Given $C>0$, we denote by $\P ^{\leq C}(\lambda)\subset \P(\lambda)$ the set of   periodic orbits with  action $\leq C$. For every $j \in \Z$ we  define $\P_j^{  \leq C}(\lambda) := \P _j (\lambda) \cap \P ^{\leq C}(\lambda) $ and $\P_j^{u, -1, \leq C}(\lambda) := \P_j^{u, -1}(\lambda) \cap \P ^{\leq C}(\lambda)$, where  $\P _j (\lambda)$ and $\P_j^{u, -1, \leq C}(\lambda)$ were established in the introduction.

Take any sequence  $\lambda_n$ of contact forms converging to
$\lambda$ as $n\to \infty$, and satisfying
\begin{itemize}
    \item[(a)] $\lambda_n$ is non-degenerate, $\forall n \in \N$. 
    \item[(b)] $P_{2,i}\in \P_2(\lambda_n)$, and $P_{2,i}$ is hyperbolic, $\forall i \in \{1, \ldots, l\}, \forall n \in \N$.
\end{itemize} 
As mentioned before, the non-degeneracy in condition (a) is achieved as in \cite[Proposition 6.1]{convex}. To achieve condition (b), we restrict to the space of contact forms $\lambda_n=f_n\lambda$ satisfying
$f_n|_{P_{2,i}}\equiv 1$ and $df_n|_{P_{2,i}}\equiv 0, \forall i,n,$ see \cite[Lemma 6.8]{convex}.

Since $P_{2,1},\ldots,P_{2,l}$ are hyperbolic, we can assume, moreover, that for any fixed $C>0$ sufficiently large, the following assertion holds:
\begin{itemize}
    \item[(c)] $\P_2^{\leq C}(\lambda_n)= \{P_{2,1}, \ldots, P_{2,l}\}, \forall n.$ 
\end{itemize}
Indeed, let $C>0$ be large enough so that $P_{2,i}\in \P_2^{\leq C}(\lambda_n),\, \forall i, \forall n$. If for each $n$ we can find an index-$2$ Reeb orbit $Q_n$ of $\lambda_n$, which is geometrically distinct from $P_{2,i}, \forall i=1,\ldots,l$, and whose action is $\leq C$,  then the Arzel\`a-Ascoli Theorem provides us with $Q \in \P^{\leq C}(\lambda)$ so that $Q_n \to Q$ in $C^\infty$ as $n\to +\infty$, up to the extraction of a subsequence. The lower semi-continuity of the Conley-Zehnder index and the weak convexity of $\lambda$ imply that $\cz(Q)=2$. Since $\P_2(\lambda)=\{P_{2,1},\ldots,P_{2,l}\}$, $Q$ must coincide with $P_{2,i}$ for some $i=1,\ldots,l$. However, this contradicts the hyperbolicity of the orbits in $\P_2(\lambda)$ and   condition (b).



The present goal is to show that $\lambda_n$ admits a weakly convex foliation $\F_n$ for every large $n$, so that every $P_{2,i}, i=1,\ldots,l,$ is a binding orbit. 

\begin{prop}\label{prop_weak_conv}
Let $J_n \in \J_{\rm reg}(\lambda_n)$ be a sequence of almost complex structures satisfying $J_n \to  J\in \J(\lambda)$ in $C^\infty$ as $n\to +\infty$, where $\J_{\rm reg}(\lambda_n)$ is given in Theorem \ref{thm:HWZFEF}. Let $\widetilde J_n$ be the almost complex structure on $\R \times S^3$  induced by $\lambda_n$ and $J_n$. Then, for every $n$ sufficiently large, the following holds.
\begin{itemize}
    \item[(i)]  The Reeb flow of $\lambda_n$ admits a weakly convex foliation $\F_n$, whose leaves are projections to $S^3$ of embedded finite energy $\widetilde J_n$-holomorphic planes and cylinders. 
    \item[(ii)] The binding of $\F_n$ consists of the orbits $P_{2,1},\ldots,P_{2,l} \in \P_2(\lambda_n)$ and $l+1$ orbits $P_{3,1}^n,\ldots, P_{3,l+1}^n\in \P_3^{u,-1,\leq C}(\lambda_n),$ where $C>0$ is a fixed large number that does not depend on $n$.
    \end{itemize}
\end{prop}

Fix any $C>0$ sufficiently large so that property (c) above holds. Before proving Proposition \ref{prop_weak_conv}, we show that for all large $n$ the orbits in  $  \P_3^{u,-1,\leq C}(\lambda_n)$  do not link with any orbit in $\P_2^{\leq C}(\lambda_n)$ and, moreover, $\lambda_n$ does not admit orbits with $\cz=1$ up to the action~$C$.

\begin{lem}\label{lem_convergeP3}
Let $Q^n_1, Q^n_2 \in \P^{u,-1,\leq C}_3(\lambda_n)$ such that $Q^n_1\neq Q^n_2, \forall n$. Then 
\begin{itemize}
    \item[(i)] there exist  $Q^\infty_1,Q^\infty_2\in \P_3^{u,-1,\leq C}(\lambda)$ so that, up to a subsequence, $Q^n_1 \to Q^\infty_1$ and  $Q^n_2 \to Q^\infty_2$ as $n \to +\infty.$
In particular, $\link(Q^n_j,P_{2,i})=0,  \forall j,i,$ and large  $n$.   
\item[(ii)] if  $\link(Q^n_1,Q^n_2)=0, \forall n,$
then $Q^\infty_1 \neq Q^\infty_2$. In particular,
$
\link(Q^\infty_1,Q^\infty_2)=0.
$
\item[(iii)]  $\P_1^{\leq C}(\lambda_n) = \emptyset, \forall n$ large.
\end{itemize}
\end{lem}

\begin{proof}

Because of the uniform upper bound on the actions of $Q^n_j, j=1,2,$ we can apply the Arzel\`a-Ascoli Theorem to extract a subsequence, still denoted by $Q^n_j$, so that $Q^n_j \to Q^\infty_j\in \P^{\leq C}(\lambda)$ in $C^\infty$ as $n\to +\infty$. The lower semi-continuity of the Conley-Zehnder index implies that $\cz(Q^\infty_j)\leq 3$. Since $\lambda$ is weakly convex, we conclude that $Q^\infty_j$ is simple. Indeed, if $Q^\infty_j = Q^d$ for some $Q\in\P(\lambda)$ and an integer $d>1$, then $\cz(Q^d)\geq 4,$ a contradiction. It turns out that, as a $C^\infty$-limit of $Q^n_j\in \P^{u,-1,\leq C}_3(\lambda_n)$, $Q^\infty_j$  is unknotted, has self-linking number  $-1$ and $\cz(Q^\infty_j) \in \{2,3\},$ $j=1,2$. For large $n$, the orbits  $P_{2,1},\ldots,P_{2,l}$ are the only orbits of $\lambda_n$ with $\cz < 3$ and action $\leq C$. Since these orbits 
are hyperbolic, we conclude that the limit $Q^\infty_j$ is not an orbit in $\P_2(\lambda)$. Hence  $\cz(Q^\infty_j)=3$ which implies  $Q^\infty_j\in \P_{3}^{u,-1,\leq C}(\lambda),$ $j=1,2$. 
In view of  hypothesis III in Theorem \ref{main1},  $Q^\infty_j$ is not linked with the orbits in $\P_2(\lambda)$. Hence, for every large $n$, $\link(Q^n_j,P_{2,i})=0, \forall j,i.$ This proves (i).

Assume now that $\link(Q^n_1,Q^n_2) =0,   \forall n$. Arguing indirectly, suppose that $Q^\infty_1 = Q^\infty_2$. Then $Q_1^n$ and $Q_2^n$ are arbitrarily close to each other as $n\to \infty$. Lemma 5.2 in \cite{convex} provides a lower bound on the winding of non-vanishing solutions of the transverse linearized flow along orbits with index $3$. 
In our case, since $\cz(Q^n_1) = \cz(Q^n_2)=3,   \forall n$ and both sequences converge to the same limit, which has also index $3$, we may apply \cite[Lemma 5.2]{convex} for every large $n$ to conclude that $\link(Q^n_1,Q^n_2)$ is necessarily positive, which is absurd.
Thus   $Q^\infty_1 \neq Q^\infty_2$ and as $C^\infty$-limits of $Q^n_j,j=1,2,$ we conclude that $\link(Q^\infty_1,Q^\infty_2) =0$. This proves (ii).

Suppose, by contradiction, that  $\P_1^{\leq C}(\lambda_n) \neq \emptyset$ for $n$ arbitrarily large. Then, after taking a subsequence, we may assume from the Arzel\`a-Ascoli Theorem that $P_1^n \to P_1^\infty$ in $C^\infty$ as $n\to +\infty,$ where $P_1^n \in \P_1^{\leq C}(\lambda_n), \forall n,$ and $P_1^\infty \in \P^{\leq C}(\lambda)$. The lower semi-continuity of the Conley-Zehnder index implies that $\cz(P_1^\infty)\leq 1,$ which contradicts  the weak convexity of  $\lambda$. Item (iii) is proved.
\end{proof}

\begin{proof}[Proof of Proposition \ref{prop_weak_conv}]Since $\lambda_n$ is non-degenerate and $J_n\in \J_{\rm reg}(\lambda_n)$, it follows from Theorem \ref{thm:HWZFEF} that the Reeb flow of $\lambda_n$ admits a genus zero transverse foliation $\F_n$ whose regular leaves are projections of embedded finite energy $\widetilde J_n$-holomorphic curves, where $\widetilde J_n$ is the $\R$-invariant almost complex structure in $\R \times S^3$ induced by $\lambda_n$ and $J_n$. By Proposition \ref{prop_unif_C}, there exists $C>0$ so that the actions of the binding orbits of $\F_n$ are bounded by $C$ for every  $n$. By Lemma \ref{lem_convergeP3}-(iii), $\P_1^{\leq C}(\lambda_n)=\emptyset,$ for every large $n$, and hence we conclude that the binding orbits of $\F_n$ have index $2$ or $3$ and that regular leaves of $\F_n$ are embedded planes and cylinders. Since, for large $n$, each $P_{2,i},i=1,\ldots, l,$ is not linked with any orbit in $\P^{u,-1,\leq C}_3(\lambda_n)$, see Lemma \ref{lem_convergeP3}-(i), it follows that each $P_{2,i},i=1,\ldots,l,$ is necessarily a binding orbit of $\F_n$. Each $P_{2,i}$ bounds a pair of planes, which are regular leaves of $\F_n$. Together with $P_{2,i}$, these planes  form a $2$-sphere which separates $S^3$ into two components. In particular, the complement of the union of these $2$-spheres has $l+1$ components. Each such a component $\U_j$ has a unique index $3$ binding orbit $P_{3,j}$. We conclude that the binding of $\F_n$ is formed by the orbits in $\P_2(\lambda)$ and $l+1$ binding orbits  $P_{3,1},\ldots,P_{3,l+1} \in \P_3^{u,-1,\leq C}(\lambda_n).$ 
\end{proof}

In order to construct the desired genus zero transverse foliation $\mathcal F$ adapted to the Reeb flow of $\lambda$  as in Theorem \ref{main1}, we shall study the compactness properties of the finite energy curves in the foliations $\tilde{\mathcal F}_n$, which are adapted to $(S^3,\lambda_n, J_n)$ and project to  genus zero transverse foliations $\mathcal{F}_n$ as in Proposition \ref{prop_weak_conv}. Recall that the almost complex structures $\widetilde J_n$ were taken in the residual set $\mathcal{J}_{\rm reg}(\lambda_n) \subset \mathcal{J}(\lambda_n)=\mathcal{J}(\lambda)$ in such a way that $J_n \to J$ in $C^\infty$ as $n \to +\infty$ for a fixed $J\in\mathcal{J}(\lambda)$. From now on we  choose $J$ in a generic set in order to prevent some unsuitable curves that may arise as limits in the compactness argument. More precisely, we need to rule out certain somewhere injective holomorphic curves whose asymptotic limits are contained in $\P_2(\lambda)$.
 The space of $J$'s for which such curves do not exist is residual in the $C^\infty$-topology. 

\begin{lem}\label{lem_genJset}
There exists a residual subset   $\J_{\rm reg}^*(\lambda)\subset \J(\lambda)$ in the $C^\infty$-topology so that for every $J\in \J^*_{\rm reg}(\lambda),$  the following assertion holds: let $\widetilde J$ be the almost complex structure on $\R \times S^3$ induced by $\lambda$ and $J.$ Let $\tilde u \colon \C \setminus \Gamma\to \R \times S^3$, $\#\Gamma<+\infty,$ be a somewhere injective finite energy $\widetilde J$-holomorphic curve having positive $d\lambda$-area and a unique positive puncture whose asymptotic limit is an orbit in $\P_2(\lambda).$ Then   $\Gamma =\emptyset.$ 
\end{lem}
\begin{proof} An application of \cite[Corollary 1.10]{Drag} provides us with a residual subset  $\J_{\rm reg}^*(\lambda)\subset \J(\lambda)$ such that for every $J \in\J_{\rm reg}^*(\lambda) $ the following holds: if $\tilde u \colon \C \setminus \Gamma\to \R \times S^3$  is a somewhere injective finite energy $\widetilde J$-holomorphic curve as in the statement, then  
${\rm Ind}(\tilde u) = {\rm CZ}(\tilde{u}) - 2 + 1+ \# \Gamma \geq 1,
$ provided $\pi \circ du \not\equiv 0.$ 
 Our standing assumptions on the actions of the orbits in $\P_2(\lambda)$ imply that the asymptotic limits of $\tilde u$ at the negative punctures in $\Gamma$ are covers of orbits in $\P_2(\lambda)$.
 
  Set $\Gamma = \{ z_1, \ldots, z_{\#\Gamma}\}$ and denote by $N_i\geq 1$   the covering number of the asymptotic limit corresponding to $z_i, i=1, \ldots, {\#\Gamma}.$  Then   its Fredholm index satisfies 
$
1\leq \mathrm{Ind}(\tilde{u})  =  2- \sum_{i=1}^{\#\Gamma} 2 N_i - 1 + {\#\Gamma}  \leq 1 -\#\Gamma,
$
from which we obtain $\Gamma = \emptyset$.  
\end{proof}

\subsection{Rigid planes asymptotic to the index-2 orbits} 

In this section we prove that for a generic choice of $J\in \J(\lambda)$, each  $P_{2,i} \in \P_2(\lambda)$ is the asymptotic limit of a pair of $\widetilde{J}$-holomorphic planes so that the closures of their projections to $S^3$ form a $C^1$-embedded $2$-sphere.

\begin{prop}\label{prop_rigid_plane0} Fix $ {J} \in \mathcal{J}^*_{\rm reg}(\lambda)$ as in Lemma \ref{lem_genJset}. Then  for each $i=1,\ldots,l,$
there exist embedded  finite energy $\widetilde J$-holomorphic planes $\tilde u_{i,1}=(a_{i,1},u_{i,1}),$ $\tilde u_{i,2}=(a_{i,2},u_{i,2})\colon \C \to \R \times S^3$ which are asymptotic to $P_{2,i}$ through opposite directions $($see Definition \ref{throughopposite}$)$.   In addition, the union $\mathcal{S}_i=u_{i,1}(\C) \cup P_{2,i} \cup u_{i,2}(\C)$ is a $C^1$-embedded $2$-sphere and $\mathcal{S}_i \cap \mathcal{S}_j =\emptyset, \forall i \neq j$. 
\end{prop}

To prove Proposition \ref{prop_rigid_plane0}, we choose a sequence of non-degenerate contact forms $\lambda_n$ converging to $\lambda$ and satisfying conditions (a), (b) and (c)  at the beginning of section \ref{sec:3}, and a sequence of almost complex structures $J_n \in \mathcal{J}_{\rm reg}(\lambda_n)$ converging to $J\in \mathcal{J}^*_{\rm reg}(\lambda)$ in $C^{\infty}$ so that the almost complex structure $\widetilde J_n$ induced by $(\lambda_n, J_n)$ admits a finite energy foliation $\tilde{\mathcal{F}}_n$ of $\R \times S^3$ whose projection to $S^3$ is a genus zero transverse foliation $\mathcal{F}_n$ adapted to the flow. Moreover, the binding of $\mathcal{F}_n$ consists of the orbits $P_{2,1},\ldots,P_{2,l}\in \P_2(\lambda_n)$ and  $P_{3,1},\ldots,P_{3,l+1}\in \P_3^{-1,u,\leq C}(\lambda_n)$, see Proposition \ref{prop_weak_conv}.

For every large $n$, $P_{2,i}\in \P_2^{\leq C}(\lambda_n)$ is the boundary of a pair of rigid planes $U_{i,1}^n,U_{i,2}^n\in \F_n$, both transverse to the Reeb vector field $R_{\lambda_n}$, so that the $2$-spheres
$\mathcal{S}_i^n = U_{i,1}^n \cup P_{2,i} \cup U_{i,2}^n,  i=1,\ldots, l,$ are mutually disjoint and  do not intersect any $P_{3,j}^n, j=1,\ldots, l+1$. The open set $S^3 \setminus \cup_{i=1}^l \S_i^n$ contains $l+1$ components $\U_j^n$ such that
$P_{3,j}^n \subset \U_j^n, \forall j,n.$ For each $i=1,\ldots,l,$ 
there exists a pair of embedded $\widetilde J_n$-holomorphic planes 
$\tilde u^n_{i,k}=(a^n_{i,k},u^n_{i,k})\colon \C \to \R \times S^3, k=1,2, $ 
asymptotic to $P_{2,i}$ through opposite directions so that
$U_{i,k}^n=u^n_{i,k}(\C)$ for every large $n$.


The following proposition implies Proposition \ref{prop_rigid_plane0}.


\begin{prop}\label{prop_rigid_planes}For each $i=1,\ldots,l,$ the embedded $\widetilde J_n$-holomorphic rigid planes $\tilde u^n_{i,1}, \tilde u^n_{i,2}\colon \C \to \R \times S^3$ converge in $C^\infty_{\rm loc}$ as $n\to \infty$, up to reparametrizations and $\R$-translations, to embedded $\widetilde J$-holomorphic rigid planes $\tilde u_{i,1}=(a_{i,1},u_{i,1}),\tilde u_{i,2}=(a_{i,2},u_{i,2})\colon \C \to \R \times S^3$ asymptotic to $P_{2,i}$ through opposite directions. The $2$-sphere $\S_i=u_{i,1}(\C) \cup P_{2,i}\cup u_{2,i}(\C)$ is $C^1$-embedded. Moreover, $\mathcal{S}_i \cap \mathcal{S}_j = \emptyset, \forall i\neq j,$ and given  neighborhoods $\V_i \subset S^3$ of $
\S_i = u_{i,1}(\C) \cup P_{2,i} \cup u_{i,2}(\C),i=1,\ldots,l,$ we have 
$
\S_i^n \subset \V_i,  \forall i=1,\ldots,l,  \forall \mbox{large } n. 
$
\end{prop}

\begin{proof}
Fix $i\in\{1,\ldots,l\}$. For simplicity,  denote $\tilde u^n_{i,1}$ by $\tilde u_n=(a_n,u_n), \forall n.$ The case of $\tilde u^n_{i,2}$ is treated similarly. Let $\U \subset S^3$ be a small compact tubular neighborhoods 
$
\U \equiv \R / T_{2,i}\Z \times B_\delta(0)
$
of $P_{2,i},$ where    $ \delta>0$ is small and   $T_{2,i}>0$ denotes the action of $P_{2,i}$ and  $B_\delta(0)\subset \R^2$ is the closed ball of radius $\delta$ centered at the origin.  Since $P_{2,i}$ is hyperbolic, we can take $\U$ sufficiently small so that
\begin{itemize}
    \item $\U$ contains no  periodic orbits that are contractible in $\U$.
    \item if $P\subset \U$ is a periodic orbit that is homotopic to $P_{2,i}$ in $\U$, then $P=P_{2,i}$.
\end{itemize}

Choose a parametrization of $\tilde u_n$ so that 
\begin{equation} \label{paramun}
\begin{cases}
  u_n(\C \setminus \D)\subset  \U, \\
 u_n(1)\in \partial \U,\\
  u_n(z^*_n) \in \partial \U  \  \mbox{ for some}  \ z^*_n \in \partial \D \  \mbox{ satisfying } \  {\rm Re}(z^*_n) \leq 0,\\
  a_n(2) = 0. 
\end{cases}
\end{equation}
The existence of such a parametrization is guaranteed as follows. For a fixed parametrization of  $\tilde u_n$,  the closure $K$ of the set $u_n^{-1}(S^3 \setminus \U)$ is  compact with non-empty interior. Take the closed disk $D\subset \C$ containing $K$ which has the smallest radius among all closed disks containing $K$. Then there exist  $w_1\neq w_2\in K \cap \partial D$ so that $u_n(w_1),u_n(w_2) \in \partial \U$.  Reparametrizing $\tilde u_n$ under a map of the form $z\mapsto a z + b, \ a,b\in \C$ we may assume that $D=\D=\{z\in \C:|z|\leq 1\}$ and that $w_1=1\in K\subset \D$. If $K\cap \partial \D$ does not contain a point $w_2\in \partial \D$ with non-positive real part, then shifting $K$ slightly to the left, it is possible to find a disk containing $K$ of radius smaller than $1$. This contradicts the minimizing property of $D$, thus the existence of $z_n^*$ as in \eqref{paramun} follows. The last condition in \eqref{paramun} may be achieved by considering a suitable $\R$-translation  of $\tilde u_n$ for each $n$.

\begin{lem}\label{lem_p21} Let $\tilde u_n=(a_n,u_n)\colon\C \to \R \times S^3$ satisfy the normalizations in \eqref{paramun}. Assume that there exist a subsequence of $\tilde u_n$, still denoted by $\tilde u_n$, and a finite energy $\widetilde J$-holomorphic map $\tilde v=(a,v)\colon\C \setminus \D \to \R \times S^3$ so that $\tilde u_n|_{\C\setminus \D}$ converges in $C^\infty_{\rm loc}(\C \setminus \D)$ to $\tilde v$ as $n \to +\infty$. Then the following assertions hold:
\begin{itemize}
    \item[(i)] $\tilde v$ is non-constant;
    \item[(ii)] $\tilde v$ is asymptotic to $P_{2,i}$ at $\infty$.
\end{itemize}
\end{lem}

\begin{proof}
For every $R> 1$ the image of the loop 
$
t \mapsto \gamma_R(t) := v(Re^{it}), \forall t\in \R / 2\pi \Z,
$ is contained in $\U$ since it is the $C^\infty$-limit of the loops  
$
t\mapsto \gamma_R^n(t):=u_n(Re^{it}),  \forall t\in \R / 2\pi \Z,  \forall n,
$ 
which are contained in $\U$. Hence $\gamma_R$ is homotopic to $\gamma_R^n$ in $\U$ for $n$ sufficiently large. Since $u_n(\C \setminus \D) \subset \U$ for every $n$, and $\gamma_R^n$ converges  to $P_{2,i}$ in $\U$ as $R \to +\infty$, we conclude that $\gamma_R$ is homotopic to $P_{2,i}$ in $\U$ for every $R>1$. This implies, in particular, that $\gamma_R$ is non-contractible in $\U$ and thus non-constant. As a result, $\tilde v$ is non-constant. Moreover, any asymptotic limit $P\subset \U$ of $\tilde v$ at $\infty$ must be homotopic to $P_{2,i}$ in $\U$ since each $\gamma_R$ has this property for every $R>1$.
Thus our choice of $\U$ implies that the unique asymptotic limit of $\tilde v$ at $\infty$ is $P_{2,i}$.
\end{proof}


We aim at showing that under the normalizations in \eqref{paramun} a bubbling-off phenomenon cannot occur for the sequence $\tilde u_n$, 
i.e. there is no subsequence of $\tilde u_n$, still denoted by $\tilde u_n$, satisfying $|\nabla \tilde u_{n}(z_n)| \to +\infty$ as $n \to \infty$ for a sequence $z_n \in \C$. Here, $|\nabla \tilde u_{n}(z_n)|$ is induced by the inner product on $\R\times S^3$ associated with the pair $(\lambda_n,J_n)$. In the absence of bubbling-off, the sequence $\tilde u_n$ has gradient bounds which, in this setup and under the normalizations in \eqref{paramun}, imply $C^\infty_{\rm loc}$-bounds for $\tilde u_n$ from an elliptic bootstrapping argument, see \cite{Hofer93}. As a result we will be able to conclude that, up to extraction of a subsequence, $\tilde u_n$ converges in $C^\infty_{\rm loc}$ to a $\widetilde J$-holomorphic plane $\tilde u_{i,1}\colon\C \to \R \times S^3$ asymptotic to $P_{2,i}$ as $n\to \infty$.  

An important tool in the bubbling-off analysis is the topological result known as  Hofer's Lemma, see \cite[Lemma 26]{Hofer93}. More specifically, assume $\tilde u_n$ admits a subsequence, still denoted by $\tilde u_n$, such that $|\nabla \tilde u_n(z_n)| \to +\infty$ as $n \to \infty$ for a sequence $z_n \in \C$. Hofer's Lemma allows us to perturb $z_n$ (the new points are still denoted by $z_n$) and find a sequence of positive numbers $\delta_n \to 0$, satisfying  $r_n:=\delta_n|\nabla \tilde u_n(z_n)| \to +\infty$, and so that an appropriate rescale $\tilde v_n \colon B_{r_n}(0) \to \R \times S^3$ of $\tilde u_n|_{B_{\delta_n}(z_n)}$  has $C^0_{\rm loc}$-  and $C^1_{\rm loc}$-bounds and satisfies $|\nabla \tilde v_n(0)|=1 $. To be precise, $\tilde v_n$ is defined by
$$\tilde v_n(z)=\left(a_n\left(z_n + \frac{\delta_n}{r_n} z\right)-a_n(z_n),u_n\left(z_n + \frac{\delta_n}{r_n} z\right)\right), \ \ \forall z\in B_{r_n}(0).$$
From an elliptic bootstrapping argument, we obtain $C_{\rm loc}^\infty$-bounds and then, up to extraction of a subsequence, $\tilde v_n$ converges in $C^{\infty}_{\rm loc}$  to $\tilde v \colon \C \to \R\times S^3,$  where $\tilde v$ is non-constant and has bounded energy by Fatou's Lemma.

If $|z_n| \to +\infty$ or $z_n$ converges to a point in $\C \setminus \D$, then 
in view of the normalizations in \eqref{paramun}, the image $v(\C)$ is contained in $\U$ and thus any of its non-trivial asymptotic limits  is a contractible periodic orbit in $\U$, a contradiction to the choice of $\U$. With this contradiction we conclude that the sequence $z_n$ must be bounded and, up to a subsequence, converges to some point $z_*\in\D$. Such a point is called a bubbling-off point for $\tilde u_n$.

Each bubbling-off point in $\D$ takes away at least $\gamma_0>0$ of the $d\lambda_n$-area of $\tilde u_n$. Here, $\gamma_0 >0 $ is any positive number smaller than the period of the shortest periodic orbit of $\lambda,$  which exists because of  the  assumptions on $\lambda.$ 
Hence, after passing to a subsequence, we may assume that the set of bubbling-off points $\Gamma \subset \D$ is finite. In particular,  $|\nabla \tilde u_n|$ is locally bounded on $\C \setminus \Gamma$. 

The normalizations in \eqref{paramun} provide $C^0_{\rm loc} $-bounds for $\tilde u_n$ in $\C \setminus \Gamma$. Hence, up to extraction of a subsequence, $\tilde u_n$ converges in $C^\infty_{\rm loc}(\C \setminus \Gamma)$ to a $\widetilde J$-holomorphic curve $
\tilde v=(b,v)\colon\C \setminus \Gamma \to \R \times S^3.
$
By Lemma \ref{lem_p21}, $\tilde v$ is asymptotic to $P_{2,i}$ at $\infty$.  Since $P_{2,i}$ is simple, it is somewhere injective. 

Let $z^*\in \Gamma.$ We claim that 
\begin{equation}\label{dlambda}\int_{\partial B_\varepsilon(z^*)} v^* \lambda > \gamma_0, \ \ \ \ \forall \varepsilon>0 \mbox{ small.} 
\end{equation}
Here, $B_\varepsilon(z^*)$ is the ball centered at $z^*$ of radius $\varepsilon>0$ and $\partial B_\varepsilon(z^*)$ has the counterclockwise orientation. 
To prove \eqref{dlambda}, recall that $\tilde u_n$ can be appropriately reparametrized in a small neighborhood of $z_n \to z^*$ so that it converges in $C^\infty_{\rm loc}$  to a non-constant finite energy plane with $d\lambda$-area $>\gamma_0$. These neighborhoods of $z_n$ are strictly contained in $B_\varepsilon(z^*)$ for every $n$ sufficiently large. Stokes' theorem then gives the desired estimate \eqref{dlambda}. The positivity of the integral in \eqref{dlambda}  implies that every puncture in $\Gamma$ is negative and therefore $\infty$ is the only positive puncture of $\tilde v$.

\begin{lem}\label{lem_gamma_vazio}
 $\Gamma = \emptyset$.
 \end{lem}
 
 \begin{proof}  
The first step  is to show that the asymptotic limit of $\tilde v$ at each  $z^* \in \Gamma$  is a cover of an orbit in $\P_2(\lambda)$. Indeed, the hypothesis II in Theorem \ref{main1} implies that if there exists an asymptotic limit $P=(x,T)$ at $z^*$ which is not a cover of an orbit in $\P_2(\lambda)$, then its period $T$ is greater than $T_{2,i}$. In particular,
$
\int_{\C \setminus \Gamma} v^* d\lambda <T_{2,i} - T<0,
$ a contradiction.

We conclude that $\tilde v$ is asymptotic to covers of orbits in $\P_2(\lambda)$ at its negative punctures.
Suppose that the $d\lambda$-area of $\tilde v$ vanishes. Then $\tilde v$ is a trivial cylinder over $P_{2,i}$. In particular, $\#\Gamma = 1$. If $\Gamma \neq \{1\} $, then $v(1) \in \partial \mathcal{U}$ since $u_n(1) \in \partial \mathcal{U}$ for every $n$. This is a contradiction. If $\Gamma = \{1\}$,  we know from our normalizations in \eqref{paramun} that $u_n(z_n^*)\in \partial \U$, and we can assume that $z_n^*\to z^*_\infty\in \partial \D,$ where ${\rm Re}(z^*_\infty) \leq 0,$ and hence  $z^*_\infty  \neq 1$ and $v(z^*_\infty) \in \partial \U,$ again a contradiction. 
It follows that the $d\lambda$-area of $\tilde v$ is positive. Since $J\in \J^*_{\rm reg}(\lambda)$, we conclude that $\Gamma =\emptyset$, see Lemma \ref{lem_genJset}. 
\end{proof}

We have proved that, under the normalizations  \eqref{paramun}, we can extract a subsequence of  $\tilde u_n$, still denoted  by  $\tilde u_n$, so that it converges in $C^\infty_{\rm loc}$ to a  finite energy $\widetilde J$-holomorphic plane 
$\tilde v\colon \C \to \R \times S^3$ asymptotic to $P_{2,i}$ at $\infty$. We denote this plane by
$
\tilde u_{i,1}=(a_{i,1},u_{i,1})\colon \C \to \R \times S^3.$
Note that it is embedded. Indeed, since $u_{i,1}$ does not intersect its asymptotic limit, an application of Siefring's result \cite[Theorem 5.20]{Si2}, see also  \cite[Theorem 14.5.5]{FvKbook}, shows that $\tilde u_{i,1}$ (and also $u_{i,1}$) is embedded.


Now the analysis of holomorphic cylinders with small area (see \cite[Lemma 4.9]{fols}) shows that given any $S^1$-invariant neighborhood $\W$ of $P_{2,i}$, there exist  $R_0>0$ and $n_0 \in \N$ such that the loop $t \mapsto \tilde u_n(Re^{2\pi i t/T_{2,i}})$ belongs to $\W$ for every $R>R_0$ and $n>n_0$. This implies that $u_n(\C)$ is arbitrarily $C^0$-close to $u_{i,1}(\C)\cup P_{2,i}$. For a proof of these statements in the same setting, see \cite[Lemma 7.5]{dPS1}.

Considering the sequence of   $\widetilde J_n$-holomorphic planes $\tilde u^n_{i,2}$, we proceed as before to obtain an embedded finite energy  $\widetilde J$-holomorphic plane 
$\tilde u_{i,2}=(a_{i,2},u_{i,2}) \colon \C \to \R \times S^3,$
asymptotic to $P_{2,i}$ as the $C^\infty_{\rm loc}$-limit of $\tilde u^n_{i,2}$.  Moreover,  $u^n_{i,2}(\C)$ is arbitrarily $C^0$-close to $u_{i,2}(\C) \cup P_{2,i}$ for $n$ large. 

From the uniqueness of $\widetilde J$-holomorphic planes asymptotic to $P_{2,i}$ through each direction,  see \cite[Proposition C.3]{dPS1}, we know that there exist at most two $\widetilde J$-holomorphic planes asymptotic to $P_{2,i}$, up to reparametrizations and $\R$-translations. In this case, they are asymptotic to $P_{2,i}$ through opposite directions, as shown below.

\begin{lem}\label{lem_int_ui12} 
$u_{i,m}(\C) \cap u_{j,n}(\C) = \emptyset, \forall (i,m)\neq (j,n)$.

\end{lem}
\begin{proof}
The transverse foliation $\F_n$ admits $l+1$ binding orbits $P_{3,j}^n,j=1,\ldots,l+1.$ 
By Lemma \ref{lem_convergeP3} these orbits  converge, up to a subsequence, to  mutually unlinked periodic orbits  $P_{3,1}^\infty,\ldots,P^\infty_{3,l+1} \in \P_3^{u,-1,\leq C}(\lambda),$ which are not linked with any index-2 Reeb orbits.

We first show that the planes $u_{i,1}(\C)$ and $u_{i,2}(\C)$ cannot be the same. Assume by contradiction that this is the case. Since $u^n_{i,1}(\C)$ and $u^n_{i,2}(\C)$ are $C^0$-close to $u_{i,1}(\C)\cup P_{2,i}=u_{i,2}(\C)\cup P_{2,i}$, we conclude that $u_{i,1}^n$ is homotopic to $u_{i,2}^n$, relative to $P_{2,i}$,  in a small neighborhood $\V\subset S^3$ of $u_{i,1}(\C)\cup P_{2,i}$. Since  the $2$-sphere $\S_i^n=u^n_{i,1}(\C) \cup P_{2,1} \cup u^n_{i,2}(\C)$ separates $S^3$ into two components and in each component there exists some index-3 binding orbit of $\F_n$, any homotopy from $u^n_{i,1}$ to $u^n_{i,2}$ necessarily intersects some $P_{3,j}^n$. Since the limit $P^\infty_{3,j}$ is not linked with $P_{2,i},\forall i,$ we know that for $n$ sufficiently large $P^n_{3,j} \cap \V = \emptyset$. 

We now assume by contradiction that $u_{i,1}(\C) \cap u_{i,2}(\C) \neq \emptyset,$ and hence $\tilde{u}_{i,1}(\C) \cap \tilde{u}_{i,2}(\C) \neq \emptyset.$ Then the positivity and stability of intersections of pseudo-holomorphic curves (see \cite[Appendix E]{MSbook}) tell us that $\tilde{u}_{i,1}^n(\C) \cap \tilde{u}_{i,2}^n(\C) \neq \emptyset$ for all $n$ large enough, a contradiction.
This finishes the proof.  \end{proof}

We conclude from Lemma \ref{lem_int_ui12} and the uniqueness of planes asymptotic to $P_{2,i}$, through each direction, that $\tilde u_{i,1}$ and  $\tilde u_{i,2}$ are asymptotic to $P_{2,i}$ through opposite directions.

\begin{lem}\label{lem_C1embedded}
The $2$-sphere $\mathcal{S}_i=u_{i,1}(\C) \cup P_{2,i} \cup u_{i,2}(\C)$ is $C^1$-embedded.
\end{lem}

\begin{proof}
Since $\cz(P_{2,i})=2$, the leading eigenvalues  of $A_{P_{2,i},J}$, which describe the behavior of  $\tilde u_{i,1}$ and $\tilde u_{i,2}$ near $\infty$, coincide with the unique eigenvalue $\mu<0$ with winding number $1$.  This means that we can write
$$
u_{i,j}(e^{2\pi(s+it)}) = \exp_{x_{2,i}(T_{2,i}t)}\{e^{\mu s}(e_j(t) + R_j(s,t))\}, \ \ \ \ s\gg 0, \ \ \ j=1,2,
$$
where $P_{2,i}=(x_{2,i},T_{2,i})$ and $e_j\colon \R / \Z \to x_{2,i}^*\xi,j=1,2,$ is a $\mu$-eigensection  with winding number $1$. The remainder term $R_j$ and its derivatives converge to $0$, uniformly in $t$ as $s\to +\infty$. See Theorem \ref{thm:asymptoticbehaviour}.

Since the $\mu$-eigenspace is one-dimensional, we have $e_2=c e_1$ for some $c\neq 0$. If $c>0$, then $u_{i,1}(\C) \cap u_{i,2}(\C) \neq \emptyset,$ see \cite[Proposition C.0]{dPS1}. This contradicts Lemma \ref{lem_int_ui12}, and we conclude that $c<0$. 

Defining 
$
r=e^{\mu s}  \Leftrightarrow  s = \frac{1}{\mu}\ln r,
$ we see that the maps
$$
v_j(r,t):= u_{i,j}(e^{2\pi(s+it)}) = {\rm exp}_{x_{2,i}( T_{2,i}t)}\{r(e_j(t) +\bar R_j(r,t))\},  \ \  \forall (r,t), \ \ j=1,2,
$$
extend continuously to $[0,\varepsilon]\times \R / \Z$ with $\varepsilon>0$ small. Since 
$
\lim_{r\to 0}  \bar R_j(r,t) =0   
$ uniformly in $t,$
we conclude that $v_j$ is at least $C^1$. Moreover, the tangent space of $v_j$ along $P_{2,i}$ coincides with $\R e_j \oplus TP_{2,i}$. Now, since $\wind_\pi(\tilde u)=\wind_\infty(\tilde u;\infty)-1=0,$ we conclude that $u_{i,j}$ is an immersion transverse to $R_\lambda$, $j=1,2.$ Hence $u_{i,1}(\C) \cup P_{2,i} \cup u_{i,2}(\C)$ is a $C^1$-embedded $2$-sphere. 
\end{proof}

Given a neighborhood $\V \subset S^3$ of $\S_i$, we have $\S_i^n \subset \V $ for every large $ n.$
Since the same argument holds for every $i=1,\ldots,l,$ the proof of Proposition \ref{prop_rigid_planes} is finished. \end{proof}

\subsection{Special index-$3$ orbits}
Although Lemma \ref{lem_convergeP3} provides candidates
$
P_{3,1}^\infty,\ldots,$ $P_{3,l+1}^\infty$ 
 in $ \P_3^{u,-1,\leq C}(\lambda)$ 
for the index $3$ binding orbits of the desired foliation, we cannot guarantee that every periodic orbit $U$, which is not a cover of an orbit in $\P_2(\lambda)$ nor a cover of any $P_{3,j}^\infty,$  is linked with some $P_{3,i}^\infty$. Such an orbit $U$ might a priori exist.  We shall  rule out this unpleasant scenario by  appropriately choosing new sequences $\lambda_n\to \lambda$ so that the corresponding limiting orbits $P_{3,1}^\infty,\ldots,P_{3,l+1}^\infty$ do not admit such an unlinked orbit $U$.

\begin{prop}\label{prop_specialP3}
Let $C>0$ be as in Proposition \ref{prop_unif_C}, and let $J\in \J^*_{\rm reg}(\lambda)$. Then there exist  a sequence of non-degenerate contact forms $\lambda_n=f_n \lambda$ converging in $C^\infty$ to $\lambda$  and a sequence of  almost  complex structures $J_n \in \J_{\rm reg}(\lambda_n)$ converging  in  $C^\infty$ to $J$  so that the following holds.
\begin{itemize}
\item[(i)] $P_{2,1},\ldots,P_{2,l}\in \P_2(\lambda_n), \forall n$, and the almost complex structure $\widetilde J_n$ on $\R \times S^3$ induced by $(\lambda_n,J_n)$ admits a stable finite energy foliation $\tilde{\mathcal{F}}_n$ that projects to a genus zero transverse  foliation $\F_n$, whose binding orbits are $P_{2,1},\ldots, P_{2,l}$ and  $P^n_{3,1},\ldots,P^n_{3,l+1}\in \P_3^{u,-1,\leq C}(\lambda_n).$ 
\item[(ii)] there exist $l+1$ periodic orbits 
$
\P_3^s(\lambda):=\{P_{3,1},\ldots,P_{3,l+1}\}\subset \P_3^{u,-1,\leq C}(\lambda),
$ 
so that for every $j$,
$P^n_{3,j} \to P_{3,j} \mbox{ as }n\to +\infty.$ 
Moreover, 
\begin{itemize}
    \item $\link(P_{3,i},P_{3,j})=0,\forall i\neq j.$ 
    \item $\link(P_{3,i},P_{2,j})=0, \forall i,j.$ 
    \item  every $P\in \P^{\leq C}(\lambda)$, which is not a cover of any orbit in $P_2(\lambda) \cup \P_3^s(\lambda)$,  is linked with some orbit in $\P_3^s(\lambda)$.
    \end{itemize}
   
\item[(iii)] let $\tilde u_{i,1}=(a_{i,1},u_{i,1}),\tilde u_{i,2}=(a_{i,2},u_{i,2})\colon \C \to \R \times S^3,  i=1,\ldots,l,$ be the unique $\widetilde J$-holomorphic planes asymptotic to $P_{2,i}$ which are $C^\infty_{\rm loc}$-limits of planes $\tilde u^n_{i,1},\tilde u^n_{i,2}$ in $\tilde{\F}_n$, asymptotic to $P_{2,i}$ through opposite directions, and whose existence is assured by  Proposition \ref{prop_rigid_planes}. Denote by $\U_j\subset S^3,j=1,\ldots,l+1,$ the  components of $S^3 \setminus \cup_{i=1}^l \S_i,
    $
    where
    $\S_i = u_{i,1}(\C) \cup P_{2,i} \cup u_{i,2}(\C).$
    Then
    $
    P_{3,j} \subset \U_j,  \forall j=1,\ldots,l+1.
    $
\end{itemize}
\end{prop}

\begin{proof}
Take a sequence of non-degenerate contact forms $\lambda_n=f_n\lambda ,n\in \N,$ converging in  $C^\infty$ to $\lambda$ and a sequence $J_n\in \J_{\rm reg}(\lambda_n)$ converging in $C^\infty$ to $J\in \J_{\rm reg}^*(\lambda)$  as in the previous section. We assume that every $\lambda_n$ satisfies conditions (a), (b), and (c) from the beginning of section \ref{sec:3}.   By  Lemma \ref{lem_convergeP3}  every orbit in $\P_3^{u,-1,\leq C}(\lambda_n)$ is not linked with any orbit in $\P_2^{\leq C} (\lambda_n)$ for every large $n$. 
We conclude  in view of  Proposition \ref{prop_weak_conv}  that the almost complex structure $\widetilde J_n$ induced by $(\lambda_n,J_n)$ admits a stable finite energy foliation which projects to a weakly convex   foliation $\F_n$ so that the orbits $P_{2,1},\ldots,P_{2,l}\in \P_2^{\leq C}(\lambda_n)$ are  binding orbits of $\mathcal{F}_n$. Moreover,  the remaining binding orbits of $\mathcal{F}_n$ are  periodic orbits in the set 
$$
\P_3^{n,0}(\lambda_n) := \{  P^{n,0}_{3,1},\ldots, P^{n,0}_{3,l+1}\} \subset \P_3^{u,-1,\leq C}(\lambda_n).
$$  They satisfy the additional properties:
\begin{itemize}
   
    \item  $\link(P^{n,0}_{3,i},P^{n,0}_{3,j})=0,\forall i\neq j.$ 
   
    \item $\link(P^{n,0}_{3,i},P_{2,j})=0, \forall i,j.$  
   
    \item given $P\in \P(\lambda_n)$, which is not a cover of any orbit in $\P_2(\lambda_n)\cup \P_3^{n,0}(\lambda_n),$ there exists $j\in\{1,\ldots,l+1\}$ so that $\link(P,P_{3,j}^{n,0}) \neq 0.$
\end{itemize}

We claim that as $n\to +\infty$ the orbits $P_{3,1}^{n,0}, \ldots, P_{3,l+1}^{n,0}$ converge, up to a subsequence, to mutually distinct and mutually unlinked periodic orbits $P_{3,1}^{\infty,0},\ldots,P^{\infty,0}_{3,l+1}$   in $\P_3^{u,-1,\leq C}(\lambda).$ 
First of all, the upper bound $C$ on the actions of $P_{3,1}^{n,0},\ldots,P_{3,l+1}^{n,0}$ and the Arzel\`a-Ascoli Theorem imply that the orbits $P^{n,0}_{3,j},j=1,\ldots, l+1,$ converge, up to a subsequence, to elements   $P_{3,j}^{\infty,0} \in \P^{\leq C}(\lambda),j=1,\ldots, l +1$. Their Conley-Zehnder indices are $\leq 3$ by the lower semi-continuity.    Since $\lambda$ is weakly convex and the orbits in $\P_{2}^{u,-1,\leq C}(\lambda_n)$ are the only orbits converging to the corresponding orbits in $\P_2(\lambda)$, we have $\cz(P_{3,j}^{\infty,0})=3, \forall j.$  In particular, $P^{\infty,0}_{3,j}$ is simply covered for every $j$. As limits of the Reeb orbits $P^{n,0}_{3,j}$ as $n\to +\infty,$ we conclude that $P^{\infty,0}_{3,j}$ is unknotted and has self-linking number $-1$. If $P_{3,j}^{\infty,0} = P_{3,k}^{\infty,0}$ for some $j\neq k$, then, since $\cz(P_{3,j}^{\infty,0})=3$, we conclude from \cite[Lemma 5.2]{convex} that, for every large $n$,  $\link(P^{n,0}_{3,j},P^{n,0}_{3,k})>0$, a contradiction. Hence the orbits $P^{\infty,0}_{3,1},\ldots,P^{\infty,0}_{3,l+1}$ are mutually distinct  and  mutually unlinked.  The claim is then proved.

Next we show that   the orbits $P_{3,j}^{\infty,0},j=1,\ldots, l+1,$ lie in distinct components $\U_{j}$ of $S^3 \setminus \cup_{i=1}^l \mathcal{S}_i,$ where $\mathcal{S}_i=u_{i,1}(\C) \cup P_{2,i} \cup u_{i,2}(\C)$ is the $C^1$-embedded $2$-sphere such that for every $i$, the maps $u_{i,1}, u_{i,2}$   are the projections to $S^3$ of the $\widetilde J$-holomorphic planes  $\tilde u_{i,1}=(a_{i,1},u_{i,1}),\tilde u_{i,2}=(a_{i,2},u_{i,2})\colon \C \to \R \times S^3,$   asymptotic to $P_{2,i}$ through opposite directions, and obtained as $C^\infty_{\rm loc}$-limits of the corresponding $\widetilde J_n$-holomorphic planes of $\tilde \F_n$, see Proposition \ref{prop_rigid_planes}. 

\begin{lem}\label{lem_P3infty}  After reordering $\U_j$, if necessary, we have
$
P_{3,j}^{\infty,0} \subset \U_j,  \ \forall j=1,\ldots,l+1.
$
\end{lem}

\begin{proof}
For every large $n$ and for every $i$, there exist   $\widetilde J_n$-holomorphic planes $\tilde u_{i,1}^n=(a_{i,1}^n,u_{i,1}^n),\tilde u_{i,2}^n=(a_{i,2}^n,u_{i,2}^n) \colon \C \to \R \times S^3$ which are  asymptotic to $P_{2,i}$ through opposite directions, and which converge in $C^\infty_{\rm loc}$, up to a subsequence, to $\widetilde J$-holomorphic planes  $\tilde u_{i,1}=(a_{i,1},u_{i,1}),\tilde u_{i,2}=(a_{i,2},u_{i,2}) \colon \C \to \R \times S^3,$ also asymptotic to $P_{2,i}$ through opposite directions.

Let 
$
\S_i^n = u_{i,1}^n(\C) \cup P_{2,i} \cup u_{i,2}^n(\C),   \forall i=1,\ldots,l.
$
For every $j=1,\ldots,l+1,$ denote by $\U_j^n\subset S^3$ the component of
$
S^3 \setminus \cup_{i=1}^l \S_i^n,
$
which contains $P_{3,j}^{n,0}.$ By Proposition \ref{prop_rigid_planes}, given small neighborhoods $\V_i$ of $\S_i,i=1,\ldots,l$, there exists $n_0$ so that $\S_i^n\subset \V_i$ for every $n>n_0$. Since $P_{3,j}^{\infty,0}$ is not linked with any $P_{2,i}$, the orbits $P_{3,1}^{n,0},\ldots P_{3,l+1}^{n,0}$ are contained in distinct components $\U_{1}^n,\ldots,\U_{l+1}^n$ for every large $n$. In particular, after relabelling the components $\U_1,\ldots,\U_{l+1}$, if necessary, these orbits are contained in distinct components $\U_{1},\ldots,\U_{l+1}$ for every large $n$.  Hence 
$P_{3,j}^{\infty,0} \subset \U_j, \forall j.
$
The lemma follows.
\end{proof}

Abbreviate
$
\P_3^{\infty,0}(\lambda) = \{P_{3,1}^{\infty,0},\ldots, P_{3,l+1}^{\infty,0} \}\subset \P_3^{u,-1,\leq C}(\lambda). 
$
From Lemma \ref{lem_P3infty} we know that $P_{3,j}^{\infty,0} \subset \U_j,  \forall j=1,\ldots,l+1.$
If every orbit in $\P^{\leq C}(\lambda)$, which is not a cover of an orbit in $\P_2(\lambda)\cup \P_3^{\infty,0}(\lambda)$, is linked with some orbit in $\P^{\infty,0}_3(\lambda)$, then  $\P_3^{\infty,0}(\lambda)$ is the desired set of periodic orbits and there is nothing else to be proved. 
Otherwise, if there exists a simple periodic orbit $U_0\in \P^{\leq C}(\lambda)$ which is not a cover of an orbit in $\P_2(\lambda)\cup \P_3^{\infty,0}(\lambda)$   and is not linked with any orbit in $\P^{\infty,0}_3(\lambda)$,
then we proceed as follows. Consider a new sequence of non-degenerate contact forms $\lambda_n$ converging in $C^{\infty}$ to $\lambda$ as $n \to +\infty$,
so that  every orbit in $\P_2(\lambda)$ and every orbit in
$$
\mathcal{L}_0:= \P_3^{\infty,0}(\lambda)\cup \{U_0\},
$$  is a periodic orbit of $\lambda_n, \forall n$. As before, for a suitable sequence $J_n \in \J_{\rm reg}(\lambda_n)$ converging to $J$, we obtain a sequence of weakly convex   foliations $\F_n$ whose binding orbits are the orbits in $\P_2(\lambda)$ together with other mutually unlinked  orbits $P^{n,1}_{3,1},\ldots,P^{n,1}_{3,l+1} \in \P_3^{u,-1,\leq C}(\lambda_n)$. 

Taking the limit $n\to +\infty$, we obtain a new set of periodic orbits 
$$
\P_3^{\infty,1}(\lambda)=\{P_{3,1}^{\infty,1},\ldots,P_{3,l+1}^{\infty,1} \}\subset \P_3^{u,-1,\leq C}(\lambda),
$$ 
so that each $P_{3,j}^{\infty,1}$ is the $C^\infty$-limit of $P_{3,j}^{n,1}$ as $n\to +\infty$. Arguing as before, we conclude that these orbits are mutually distinct and  mutually  unlinked and satisfy
$
P_{3,j}^{\infty,1} \subset \U_j, \forall j.
$ Some of them may coincide with the corresponding orbits in $\P_3^{\infty,0}(\lambda)$. 

We claim that no orbit in $\P_3^{\infty,1}(\lambda)$ coincides with $U_0$. Indeed, observe that each $P_{3,j}^{n,1}$ is contained in $\U_j$ for all large $n$ and either  coincides with $P_{3,j}^{\infty,0}$ or is linked with $P_{3,j}^{\infty,0}$.  Hence  $P_{3,j}^{\infty,1}$ either coincides with $P_{3,j}^{\infty,0}$ or is linked with $P_{3,j}^{\infty,0}$.  In particular, since  $U_0$ is linked with some $P^{n,1}_{3,j}$ for every large $n$, we conclude that $U_0$ is linked with some $P_{3,j}^{\infty,1}$.  The claim is proved.

Now if every orbit in $\P^{\leq C}(\lambda)$, which is not a cover of an orbit in $\P_2(\lambda)\cup \P_3^{\infty,1}(\lambda)$, is linked with some orbit in $P^{\infty,1}_3(\lambda)$, then  $\P_3^{\infty,1}(\lambda)$ is the desired set of periodic orbits and there is nothing else to be proved.  
Otherwise, if there exists a simple periodic orbit $U_1\in \P^{\leq C}(\lambda)$, which is not a cover of an orbit in $\P_2(\lambda)\cup \P_3^{\infty,1}(\lambda)$ and is not linked with any orbit in $P^{\infty,1}_3(\lambda)$,
then we construct another new sequence of non-degenerate contact forms $\lambda_n$ converging in $C^{\infty}$ to $\lambda$ as $n \to +\infty$ as before,   so that  every orbit in $\P_2(\lambda)$ and every orbit in
$$
\mathcal{L}_1 := \mathcal{L}_0 \cup \P_3^{\infty,1} \cup \{U_1\},
$$
is a periodic orbit of $\lambda_n$ for every $n.$ Choosing a suitable sequence $J_n \to J$, we find a new set $\P_3^{\infty,2}(\lambda)=\{P_{3,1}^{\infty,2},\ldots,P_{3,l+1}^{\infty,2} \}\subset \P_3^{u,-1,\leq C}(\lambda)$, with $P_{3,j}^{\infty,2} \subset \U_j, \forall j,$ as the limit of the new binding orbits $P^{n,2}_{3,j}, j=1,\ldots,l+1$.

As before, we claim that no orbit in $\P_3^{\infty,2}(\lambda)$ coincides with $U_0$ or $U_1$. To see this, observe that each $P_{3,j}^{n,2}$ is contained in $\U_j$ for all large $n$ and either it coincides with one of the orbits $P_{3,j}^{\infty,0}$ or $P_{3,j}^{\infty,1}$ or is linked with both of them.  Hence  $P_{3,j}^{\infty,2}$ either  coincides with one of the orbits $P_{3,j}^{\infty,0},P_{3,j}^{\infty,1}$, or is linked with both of them.    In particular, since  $U_0$ is linked with some $P^{n,2}_{3,j}$ for every large $n$, we conclude that $U_0$ is linked with some $P_{3,j}^{\infty,2}$. The same holds with $U_1$.   The claim follows. 


Again, if we find a simple periodic orbit $U_2\in \P^{\leq C}(\lambda)$ which is not a cover of an orbit in $\P_2(\lambda)\cup \P_3^{\infty,2}(\lambda) $  and is not linked with any   orbit in $\P_3^{\infty,2}(\lambda)$, we define a new set
$$
\mathcal{L}_2 := \mathcal{L}_1 \cup \P_3^{\infty,2}(\lambda)\cup \{U_2\},
$$
and consider again a new sequence $\lambda_n \to \lambda$ (and $J_n\to J$) as before to obtain $\P_3^{\infty,3}(\lambda)$ with similar properties and so on.

Repeating this process indefinitely, if necessary, we end up with  sequences  
$$
  \P_3^{\infty,k}(\lambda)  \subset \P_3^{u,-1,\leq C}(\lambda)   
 \ \ \  \mbox{and} \ \ \  U_k\in \P^{\leq C}(\lambda), \ \ k\in \N,
 $$ so that 
\begin{itemize}
    \item $\P_3^{\infty,k}(\lambda)\subset \P_3^{u,-1,\leq C}(\lambda)$ is formed by $l+1$ mutually distinct and mutually unlinked orbits $P_{3,j}^{\infty,k} \subset \U_j, \forall j=1, \ldots, l+1$.

    \item $U_k$ is not a cover of any orbit in $\P_2(\lambda)\cup \P_3^{\infty,k}(\lambda)$.
    
    \item $U_k$ is not linked with any orbit in $  \P_3^{\infty,k}(\lambda)$. 
    
    \item for every $p>k$ there exists an orbit in $ \P_3^{\infty,p}(\lambda)$ which is linked with $U_k$. 
\end{itemize}
We may extract a subsequence so that the orbits $P_{3,1}^{\infty,k}, \ldots, P_{3,l+1}^{\infty,k} \in \P_3^{\infty,k}(\lambda)$ are converging to 
$
Q^\infty_1,\ldots,Q^\infty_{l+1} \in \P_3^{u,-1,\leq C}(\lambda)
$ as $k\to +\infty$. 
The periodic orbits $Q_j^\infty,j=1,\ldots,l+1,$ are mutually distinct, mutually unlinked and $Q_j^\infty \subset \U_j, \forall j$. Moreover, $\cz(Q_j^\infty)=3, \forall j.$ 

Using the Arzel\`a-Ascoli theorem, we may  assume that 
$
U_k \to U_\infty$ as $k\to +\infty,
$ where the action of $U_\infty$ is $\leq C$. Since $P_{2,i}$ is hyperbolic and each $U_k$ is geometrically distinct from the orbits in $\P_2(\lambda)$, we conclude that  $U_\infty$ is not a cover of any orbit in $\P_2(\lambda)$. Observe that:
\begin{itemize}

    \item If $U_\infty$ is a cover of  $Q_j^\infty$ for some $j$, then  since $\cz (Q_j^\infty)=3$  we conclude that   $U_k$ is linked with $P_{3,j}^{\infty,k}$ for every $k$ sufficiently large,   a contradiction.
    
    \item If $U_\infty$ is not a cover of $Q_j^\infty$ for any $j$  and  is linked with some $Q_j^\infty$, then  $U_k$ is linked with $P_{3,j}^{\infty,k}$ for $k$ sufficiently large, a contradiction.
    
    \item  If $U_\infty$ is not a cover of $Q_j^\infty$ for any $j$  and is  not linked with any $Q_j^\infty$, then for $k$ sufficiently large  $U_k$ is not linked with any $P_{3,j}^{\infty,p}$ for  every $p>k$. This is also a contradiction.
    
    \end{itemize}

We have proved that the process of constructing such new sequences of non-degenerate contact forms converging to $\lambda$  must terminate after finitely many steps. Hence  we find  a sequence $\lambda_n\to \lambda$ with the desired properties whose limiting periodic orbits $P_{3,1},\ldots, P_{3,l+1}$ satisfy all  properties in the statement of Proposition \ref{prop_specialP3}.  The proof is complete.
\end{proof}

\subsection{Compactness properties of holomorphic cylinders}\label{sec:cylinder}

Let $\lambda_n$ and $J_n \in \mathcal{J}_{\rm reg}(\lambda_n)$ be sequences of non-degenerate contact forms and almost complex structures converging  to $\lambda$ and $J \in \mathcal{J}^*_{\rm reg}(\lambda)$, respectively, as obtained in Proposition \ref{prop_specialP3}. For every $n$,  the almost complex structure $\widetilde J_n$ induced by $(\lambda_n, J_n)$ admits a stable finite energy foliation $\tilde \F_n$ that projects to a genus zero transverse  foliation $\F_n$ whose binding orbits are $P_{2,1},\ldots,P_{2,l}, P_{3,1}^n, \ldots, P_{3,l+1}^n$, where, for every $j$, the orbit $P_{3,j}^n$ converges to $P_{3,j} \in \P_3^s(\lambda)$.

Fix $j\in \{1, \ldots, l+1\}$, so that $P_{3,j}^n \subset \U_j$ for every $n$. Choose an arbitrary boundary component, say $\S_i,$ of the closure of $\U_j$.   Then every $\F_n$ contains a unique  rigid cylinder connecting $P_{3,j}^n$ to $P_{2,i}.$ This cylinder is the projection of an embedded finite energy $\widetilde{J}_n$-holomorphic cylinder $\tilde{v}_n=(b_n,v_n) \colon \R \times \R / \Z \to \R \times S^3$. It is asymptotic to $P_{3,j}^n$ at its positive puncture $+\infty$ and to $P_{2,i}$ at its negative puncture $-\infty$.

We shall study the compactness properties of the sequence $\tilde{v}_n.$ 
Consider $\U\subset \U_j$ a  small compact tubular neighborhood of $P_{3,j}$. Since $\cz(P_{3,j})= 3$, we can choose $\U$  sufficiently small so that
\begin{itemize}
\item   $\U$ contains no periodic orbits that are contractible in    $\U$.
   \item there exists no Reeb orbit    $P\subset \U$ of $\lambda$ which is geometrically distinct from $P_{3,j}$, is homotopic to $P_{3,j}$ in $\U$ and satisfies $\link(P,P_{3,j})=0$.
  \end{itemize}
The first property follows from the fact that for $\U$ sufficiently small, every periodic orbit in $\U$ must be homotopic in $\U$ to a positive cover of $P_{3,j}$ and hence is non-contractible in $\U$. The second property can be achieved since $P_{3,j}\in \P_3^{u, -1}(\lambda).$ Indeed, the flow near $P_{3,j}$ twists fast enough so that any periodic orbit sufficiently close to $P_{3,j}$, if it exists, must be linked with $P_{3,j}$. See Lemma 5.2 in \cite{convex}. 


Using that $P^n_{3,j} \to P_{3,j}$ as $n\to +\infty$, we observe that $P^n_{3,j}\subset {\rm int}(\U)$ for every large $n$ and, moreover, due to the asymptotic properties of $\tilde v_n$, we can normalize $\tilde v_n$ to satisfy the following conditions 
\begin{equation}\label{paramvn3}
\begin{cases}
 v_n(s, t)\in  \U \ \mbox{for all} \ s >0, \ t \in \R / \Z.\\
  v_n(0,0)\in \partial \U.\\
  b_n(1,0) = 0. 
\end{cases}
\end{equation}

\begin{lem}\label{lem_p3j3} Let $\tilde v_n$ satisfy the normalizations  \eqref{paramvn3}. Assume that there exist a subsequence of $\tilde v_n$, still denoted by  $\tilde v_n$, and a finite energy $\widetilde J$-holomorphic map $\tilde v=(b,v)\colon (0, \infty) \times \R / \Z \to \R \times S^3$ so that $\tilde v_n|_{(0, \infty) \times S^1}$ converges in $C^\infty_{\rm loc}((0, \infty) \times \R / \Z)$ to $\tilde v$ as $n \to +\infty$. Then 
 $\tilde v$ is asymptotic to $P_{3,j}$ at $\infty$.
\end{lem}

\begin{proof}
We argue as in     Lemma \ref{lem_p21}. Let $R>0$. Since the loop  $t\mapsto \gamma_R^n(t):=v_n(R,t) ,  t\in  \R / \Z,$ lies in $\U$ for every $n$  and since $\gamma_R^n$ converges in $C^\infty(\R / \Z)$ to the loop 
$
t \mapsto \gamma_R(t): = v(R, t ), t\in \R / \Z,
$ 
as $n \to +\infty$, we have $\gamma_R(t) \in \U, \forall t$. In particular, $\gamma_R$ is homotopic to $\gamma_R^n$ in $\U $ for every large $n$. 

The fact that $v_n((0, \infty) \times \R / \Z) \subset \U$ for every $n$ and that $\gamma_R^n$ converges to $P_{3,j}^n$ in $\U$ as $R \to +\infty$ implies that $\gamma_R$ is homotopic to $P_{3,j}^n$ in $\U$ for every $R >0$. Since $P_{3,j}^n\to P_{3,j}$ as $n\to +\infty$, $\gamma_R$ is homotopic to $P_{3,j}$ in $\U$ for every $R >0$. In particular, $\gamma_R$ is non-contractible in $\U$ and thus non-constant. Hence $\tilde v$ is non-constant as well. 

Any asymptotic limit $P\subset \U$ of $\tilde v$ at $\infty$  is homotopic to $P_{3,j}$ in $\U$ since each $\gamma_R$ is homotopic to $P_{3,j}$ in $\U$ for every $R>0$.  Assume $P\neq P_{3,j}$. Then $P$ must be linked with $P_{3,j}$. This follows from the choice of $\U$. But since $\gamma_R$ is the $C^\infty$-limit of the loops $\gamma^n_R$ as $n\to +\infty$, which are all unlinked with $P_{3,j}^n$, we conclude that $\gamma_R$ is not linked with $P_{3,j}$, 
a contradiction.  Hence $P_{3,j}$ is the unique asymptotic limit of $\tilde v$ at $\infty$. This finishes the proof.
\end{proof}

 We next show that, up to a subsequence, the sequence   $\tilde{v}_n \colon \R \times \R / \Z \to \R \times S^3$  converges to a $\widetilde J$-holomorphic map   $\tilde v=(b,v) \colon (\R \times \R / \Z) \setminus \Gamma \to \R \times S^3$ which is  asymptotic to $P_{3,j}$ at its unique positive puncture $+\infty$  and to orbits in $\P_2(\lambda)$ at the negative punctures in $\Gamma \cup\{-\infty\}$.

\begin{prop}\label{prop_cylinder1}
Let $\tilde{v}_n=(b_n,v_n) :\R \times \R / \Z \to \R \times S^3$ be the sequence of embedded finite energy $\widetilde{J}_n$-holomorphic cylinders as above. 
Under the  particular choice of a small compact tubular neighborhood $\U \subset \U_j$ of $P_{3,j}$  and the normalizations \eqref{paramvn3}, there exists an embedded finite energy  $\widetilde J$-holomorphic curve $\tilde v=(b,v) \colon (\R \times \R / \Z) \setminus \Gamma \to \R \times S^3$, asymptotic to $P_{3,j}$ at its unique positive puncture $+\infty$,  to $P_{2,i}$ at $-\infty$ and  to other distinct orbits in $\P_2(\lambda)$ at the   punctures in $ \Gamma$,  so that, up to a subsequence, $\tilde v_n \to \tilde v$ in $C^\infty_{\rm loc}$ as $n \to +\infty$.  
Moreover,  $v( ( \R \times \R / \Z) \setminus \Gamma) \subset \U_j,$ and the convergence of $\tilde v$ to $P_{3,j}$ at $+\infty$ is exponential.
\end{prop}

\begin{proof}
 Arguing as in the proof of Proposition \ref{prop_rigid_planes}, we take a subsequence of $\tilde v_n$, still denoted by $\tilde v_n$, which admits  a sequence $(s_n, t_n) \in \R \times \R/ \Z  $ so that $|\nabla \tilde v_n(s_n, t_n)| \to +\infty$ as $n \to +\infty$. Then we have $\limsup_{n \to \infty}s_n \leq 0$  and, up to a subsequence, we assume that $(s_n, t_n)$ converges to a point in $(-\infty, 0] \times \R / \Z$.  
Moreover, extracting a subsequence, we can assume that the set of bubbling-off points $\Gamma\subset   (-\infty, 0] \times \R / \Z  $ is finite. 
Because of the normalizations \eqref{paramvn3}, we can find a $\widetilde J$-holomorphic map $\tilde v=(b,v) \colon  ( \R \times \R / \Z) \setminus \Gamma \to \R \times S^3 $ so that $\tilde v_n$ converges to $\tilde v $ in $C^\infty_{\rm loc} ((\R \times \R / \Z) \setminus \Gamma)$ as $n \to +\infty$. Moreover, $\tilde v$ is asymptotic to $P_{3,j}$ at $+\infty$, see Lemma \ref{lem_p3j3}, and $v( (\R \times \R / \Z) \setminus \Gamma) \subset \U_j$  by positivity and stability of intersections. Every puncture in $\Gamma$ is negative. For a proof, see  the argument just after \eqref{dlambda}. The puncture at $-\infty$ is also negative since $\int_{\{s\} \times \R / \Z} v_n^* \lambda_n > T_{2,i}$ for every $n$ and every fixed $s\ll 0$. Since $+\infty$ is the only positive puncture of $\tilde v$ and its asymptotic limit is simple, we also conclude that $\tilde v$ is somewhere injective.

Next we claim that   the asymptotic limit of $\tilde v$ at each  $(s^*, t^*) \in \Gamma$  is a cover~of an orbit in $\P_2(\lambda)$.  Assume by contradiction that   there exists an asymptotic limit $Q\subset \U_j$ at $(s^*, t^*) \in \Gamma$, which is not a cover of an orbit in $\P_2(\lambda)$.   If $Q$ is not a cover of $P_{3,j}$, then $Q$ is linked with $P_{3,j}$, see Proposition \ref{prop_specialP3}. Let  $\varepsilon>0$ be sufficiently small so that the loop $v(\lvert (s,t)-(s^*, t^*)\rvert=\varepsilon)\subset S^3$ is arbitrarily close to $Q$. Then  the loop $v_n(|(s,t)-(s^*, t^*)|=\varepsilon)$ is linked with $P_{3,j}$, if    $n$ is large enough.   In particular, $v_n(|(s,t)-(s^*, t^*)|\leq \varepsilon)$ intersects $P^n_{3,j}$ for every large $n$, a contradiction. It follows that $Q$ is a cover of $P_{3,j}.$ But this implies that $\int_{(\R \times \R / \Z) \setminus \Gamma}v^* d\lambda \leq T_{3,j}-T_{2,i} -T_{3,j}<0$, a contradiction. A similar argument shows that the asymptotic limit of $\tilde v$ at $-\infty$ is a cover of an orbit in $\P_2(\lambda)$.  We conclude that the asymptotic limits of $\tilde v$ at its negative punctures are covers of orbits in $\P_2(\lambda).$

Now we show that   the $d\lambda$-area of $\tilde v$ is positive. Otherwise, $\tilde v$ is a cylinder over some periodic orbit $P$, see \cite[Theorem 6.11]{props2}. In particular, $\Gamma=\emptyset$ and $\tilde v$ is a trivial cylinder over $P_{3,j}$. But this contradicts our normalization \eqref{paramvn3} since it implies $v(0,0)\in \partial \U$.


We have showed that $\tilde v$ is asymptotic to $P_{3,j}$ at $+\infty$, and  to covers of orbits in $\P_2(\lambda)$ at its negative punctures in $\Gamma\cup \{-\infty\}$. The usual analysis near $P_{3,j}$ (see \cite[Theorem 7.2]{convex} and also \cite[Proposition 9.3]{dPS1}) implies that the convergence of $\tilde v$ to $P_{3,j}$ is exponential with a negative leading eigenvalue of $A_{P_{3,j},J}$, whose eigenvector has winding number $1$ with respect to any global trivialization of the contact structure. The asymptotic behavior of $\tilde v$ at $+\infty$ is as in Theorem \ref{thm:asymptoticbehaviour}.

We still need to prove that $\tilde v$ is asymptotic to $P_{2,i}$ at $-\infty$ and to other distinct orbits of $\P_2(\lambda)$ at the remaining negative punctures in $\Gamma$.
Assume, by contradiction, that $\tilde v$ is asymptotic to a $p_0$-cover, $p_0>1$, of  some $P_{2,i_0} \in  \P_2 (\lambda)$ at a negative puncture   $(s_0, t_0) \in \Gamma.$ In particular, $P_{2,i_0}\subset \partial \U_j$. Since $P_{2,i_0}$ is hyperbolic and satisfies $\mu_\cz(P_{2,i_0})=2$, the asymptotic operator $A_{P_{2,i_0},J}$ associated with $P_{2,i_0}$ and $J$ admits a unique  positive eigenvalue $\mu$ with winding number $1$ (the least positive eigenvalue) and associated $\mu$-eigenfunctions $e,e'$ which point inside and outside $\U_j$, respectively. Moreover,  the asymptotic operator $A_{P_{2,i_0}^{p_0},J}$ associated with the ${p_0}$-cover $P_{2,i_0}^{p_0}$ of $P_{2,i_0}$ and $J$ admits an eigenfunction $e^{p_0}$, which equals to $e$ covered ${p_0}$ times, whose associated eigenvalue is $p_0 \mu$. Its winding number is $p_0$ with respect to a global trivialization of $\xi$. Since   $P_{2,i_0}^{p_0}$  is also hyperbolic and satisfies $\mu_{\cz}(P_{2,i_0}^{p_0})=2{p_0},$ the eigenvalue ${p_0} \mu$ is the least positive eigenvalue of  $A_{P_{2,i_0}^{p_0},J}$, and the other positive eigenvalues admit winding numbers larger than ${p_0}$. Since  the image of $v $ lies in $\U_j$, this implies that the eigenfunction $e^{p_0}$   describes the asymptotic behavior of $\tilde v$ near the  puncture $(s_0, t_0)$. 
Since ${p_0}>1$, the map $v$ admits self-intersection, see \cite[Proposition C.0-i)]{dPS1}, which is impossible since the somewhere injective curve $\tilde v$ is the $C^\infty$-limit of embedded curves whose projections to $S^3$ are embedded surfaces. We conclude that the asymptotic limit of $\tilde v$ at a puncture in $\Gamma$ is an orbit in $\P_2(\lambda)$. In a similar way, if $\tilde v$ is asymptotic to the same orbit in $\P_2(\lambda)$ at distinct punctures in $\Gamma\cup \{-\infty\}$, then   $v$ also self-intersects, see \cite[Proposition C.0-ii)]{dPS1}. Again, this is a contradiction since the somewhere injective curve $\tilde v$ is the  $C^\infty_{\rm loc}$-limit of embedded curves  whose projections to $S^3$ are embedded surfaces. 
 
We conclude that, up to extraction of a subsequence, the sequence $\tilde{v}_n$ of embedded finite energy $\widetilde{J}_n$-holomorphic planes, normalized as in \eqref{paramvn3},    converges  in $C^\infty_{\rm loc} ((\R \times \R / \Z) \setminus \Gamma)$   to a somewhere injective curve $\tilde v$ as $n \to +\infty$. Moreover, $\tilde v$ is asymptotic to $P_{3,j}$ at $+\infty$, and to  distinct orbits in $\P_2(\lambda)$ at its negative puncture  in $\Gamma\cup \{-\infty\}$.           Moreover, $\tilde v$ and $v$ are embeddings, and $v( (\R \times \R / \Z) \setminus \Gamma) \subset \U_j$. 

It remains to show that $\tilde v$ is asymptotic to $P_{2,i}$ at $-\infty$. To see this, let $P_{2,l_0}\in \P_2(\lambda)$ be the asymptotic limit of $\tilde v$ at $-\infty.$ Assume $l_0\neq i.$
For suitable $s_n,c_n\in \R$, with $s_n\to -\infty$, the shifted maps
$
\tilde w_n(s,t) =(b(s-s_n)+c_n,v(s-s_n)),  \forall (s,t) \in \R \times \R / \Z,
$
converges in $C^\infty_{\rm loc}$ to a finite energy $\widetilde J$-holomorphic curve $\tilde w \colon (\R \times \R / \Z) \setminus \Gamma' \to \R \times S^3$ with $\Gamma'$ finite, which is asymptotic to $P_{2,l_0}$ at its positive puncture at $+\infty$. Moreover,  $s_n,c_n$ can be appropriately chosen so that $\tilde w$ is not a trivial cylinder over $P_{2,l_0}$. Indeed, since all orbits with action $\leq T_{2,l_0}$ are hyperbolic, one can apply the SFT-compactness theorem (see \cite{BEHWZ03}) to obtain a genus zero building of $\widetilde J$-holomorphic curves so that the building has a positive puncture at $P_{2,l_0}$ and a negative puncture at $P_{2,i}$. The first level of this building is the non-trivial curve $\tilde w$.  Moreover, arguing as above, the punctures of $\tilde w$ in $\Gamma' \cup \{-\infty\}$ are negative and $\tilde w$ is asymptotic to distinct orbits in $\P_2(\lambda)$ at $\Gamma' \cup \{-\infty\}$. Hence $\tilde w$ is somewhere injective and has positive $d\lambda$-area. Since $J\in \J_{\rm reg}^*(\lambda)$, we can apply Lemma \ref{lem_genJset} to conclude that the set of negative punctures of $\tilde w$ is empty, a contradiction. This proves that $l_0=i$ and finishes the proof of this proposition.
\end{proof}

\subsection{Compactness properties of holomorphic planes}

In this section we study the compactness properties of the families of $\widetilde J_n$-holomorphic planes asymptotic to $P^n_{3,j}\in \P_3^{u,-1,\leq C}(\lambda_n),$ $j=1,\ldots,l+1$. Recall that 
$
P^n_{3,j} \to P_{3,j} $ as $n\to+\infty$, $\forall j=1,\ldots,l+1,
$
where $P_{3,j}\in \P_3^{u,-1,\leq C}(\lambda).$ The orbit $P_{3,j}$ lies in the component 
$
\U_j\subset S^3 \setminus \bigcup_{i=1}^l \S_i, \    j=1,\ldots, l+1,
$ where $\S_i = U_{i,1} \cup P_{2,i} \cup U_{i,2}\subset S^3$ is a $C^1$-embedded $2$-sphere. We may assume that  the sequences $\lambda_n$ and $J_n$ are given as in   Proposition \ref{prop_specialP3} so that the   orbits $P_{3,j}$ satisfy the properties stated in that proposition.

Fix $j$ and denote by $\tilde k_j$ the number of boundary components of   $\U_j$. Let 
$
i_1,\ldots i_{\tilde k_j} \in \{1,\ldots,l\}
$ 
be such that 
$
\partial \U_j = \bigcup_{k=1}^{\tilde k_j} \S_{i_k}.
$
For each $n\in \N$, $j\in \{1,\ldots,l+1\}$ and $k\in \{1,\ldots, \tilde k_j\}$, there exists a one-parameter family of  planes 
$
 \F^{j,n}_{k,\tau},  \tau\in (0,1),
$ asymptotic to $P^n_{3,j}$. 
For simplicity, we omit $j,k$ in the notation, i.e.\ $\F^n_\tau = \F^{j,n}_{k,\tau}, \forall \tau.$ Let  $\tilde u^n_\tau = (a^n_\tau,u^n_\tau)\colon \C \to \R \times S^3,  \tau\in (0,1),$ be the family of $\widetilde J_n$-holomorphic planes so that
$
u^n_\tau(\C) = \F^n_\tau, \forall n,\tau.
$

Fix $p\in \U_j \setminus P_{3,j}$ and consider a sequence $p_n\to p,$
where $p_n\in u^n_{\tau_n}(\C)\subset \U_j$ for some $\tau_n \in (0,1).$ 
Denote by
$\tilde v_n=(b_n,v_n)\colon \C \to \R \times S^3 $
the $\widetilde J$-holomorphic plane $\tilde u^n_{\tau_n}$. In particular, $p_n\in v_n(\C), \forall n.$

We consider a small compact tubular neighborhoods $\U \subset \U_j$   of $P_{3,j}$. As in Section \ref{sec:cylinder}, since $\cz(P_{3,j})= 3$,  we can choose $\U$ sufficiently small so that

\begin{itemize}
\item $p \not \in \U$.
\item   $\U$ contains no periodic orbits that are contractible in    $\U$.
   \item there exists no periodic orbit    $P\subset \U$  which is geometrically distinct from $P_{3,j}$, is homotopic to $P_{3,j}$ in $\U$ and satisfies $\link(P,P_{3,j})=0$.
  \end{itemize}

Using that $P^n_{3,j} \to P_{3,j}$ as $n\to +\infty$, we observe that $P^n_{3,j}\subset {\rm int}(\U)$ for every large $n$ and, moreover, we can normalize $\tilde v_n$ to satisfy the following conditions 
\begin{equation}\label{paramvn}
\begin{cases}
 v_n(\C \setminus \D)\subset  \U, \\
  v_n(1)\in \partial \U,\\
  v_n(z^*_n) \in \partial \U  \  \mbox{ for some}  \ z^*_n \in \partial \D \  \mbox{ satisfying } \  {\rm Re}(z^*_n) \leq 0,\\
  b_n(2) = 0. 
\end{cases}
\end{equation}
This normalization is constructed exactly as  in \eqref{paramun}. 

 The proof of the following lemma is similar to the proof of Lemma \ref{lem_p3j3}. 
 
 \begin{lem}\label{lem_p3j} 
 Let $\tilde v_n$ satisfy the normalizations in \eqref{paramvn}. Assume that there exist a subsequence of $\tilde v_n$, still denoted by $\tilde v_n$, and a $\widetilde J$-holomorphic map $\tilde v=(b,v)\colon \C \setminus \D \to \R \times S^3$ so that $\tilde u_n|_{\C\setminus \D}$ converges in $C^\infty_{\rm loc}(\C \setminus \D)$ to $\tilde v$ as $n \to +\infty$. Then 
    $\tilde v$ is  asymptotic to $P_{3,j}$ at $\infty$.
 \end{lem}




Recall that $\tilde v_n=(b_n,v_n)$ is such that $p_n\in v_n(\C), \forall n$, where $p_n \to p\in \U_j\setminus P_{3,j}$ is fixed. Next we show that there exists an open subset of $\U_j\setminus P_{3,j}$ of full measure so that if the limit point  $p$ is fixed in this subset, then the sequence $\tilde v_n$, normalized as in \eqref{paramvn} (for particular choices of the small tubular neighborhood $\U\subset \U_j$ of $P_{3,j}$), does not admit bubbling-off points and converges in $C^\infty_{\rm loc}$ as $n\to +\infty$ to a $\widetilde J$-holomorphic plane $\tilde v\colon \C \to \R \times S^3$ asymptotic to $P_{3,j}$. The cases for which this compactness property fails occur when the sequence $\tilde v_n$ admits bubbling-off points in $\D$, and the limiting curve is asymptotic to $P_{3,j}$ at $\infty$ and to distinct orbits in $\P_2(\lambda)$ at its negative punctures. As we shall prove below,   there are only finitely many such curves. Therefore, if $p$ is taken in the complement of the image of these curves, the compactness property holds, and the limiting curve is a $\widetilde J$-holomorphic plane asymptotic to $P_{3,j}$.

\begin{prop}\label{prop_family_planes1}
There exists an open subset $\U_j' \subset \U_j\setminus P_{3,j}$ of full measure in $\U_j\setminus P_{3,j}$ so that if $p_n\to p \in \U_j'$ and $\tilde v_n=(b_n,v_n)\colon\C \to \R \times S^3$ is such that $p_n\in v_n(\C), \forall n,$ then, under the  particular choice of a small compact tubular neighborhood $\U  \subset \U_j$ of $P_{3,j}$ depending on $p$ as above and the normalizations in \eqref{paramvn}, there exists a $\widetilde J$-holomorphic plane $\tilde v=(b,v)\colon\C \to \R \times S^3$, exponentially asymptotic to $P_{3,j}$,  so that $\tilde v_n \to \tilde v$ in $C^\infty_{\rm loc}$ as $n \to +\infty$ and $p\in v(\C)$.
\end{prop}

\begin{proof}
Arguing as in the proof of Proposition \ref{prop_rigid_planes}, we take a subsequence of $\tilde v_n$, still denoted by $\tilde v_n$, which admits  a sequence $z_n \in \C$ so that $|\nabla \tilde v_n(z_n)| \to +\infty$ as $n \to \infty$. Then $z_n$ is bounded and, up to extraction of a subsequence, converges to a point in $\D$.  
Moreover, extracting a subsequence, we can assume that the set of bubbling-off points $\Gamma\subset \D$ is finite. 
Because of the normalizations \eqref{paramvn}, 
we   find a $\widetilde J$-holomorphic map $\tilde v=(b,v)\colon \C \setminus \Gamma \to \R \times S^3 $ so that $\tilde v_n$ converges to $\tilde v$ in $C^\infty_{\rm loc} (\C \setminus \Gamma)$ as $n \to +\infty$. Moreover,  $\tilde v$ is asymptotic to $P_{3,j}$ at $\infty$ (see    Lemma \ref{lem_p3j}) and to  distinct orbits in $\P_2(\lambda)$ at the negative punctures in $\Gamma$. The convergence of $\tilde v$ to $P_{3,j}$ is exponential with a negative leading eigenvalue whose eigenfunction has winding number $1$ with respect to any global trivialization of the contact structure. We also have $v(\C \setminus \Gamma) \subset \U_j$.

Performing a soft-rescaling of $\tilde v_n$ near each puncture in $\Gamma$ where $\tilde v$ is asymptotic to some $P_{2,i}\in \P_2(\lambda),$ we find a new $\widetilde J$-holomorphic curve  $\tilde w=(d,w)\colon \C \setminus \Gamma'\to \R \times S^3,$ with $\Gamma'$ finite, which is 
asymptotic to $P_{2,i}$ at $\infty$ and, at its punctures in $\Gamma'\subset \D$, the curve is asymptotic to covers of orbits in $\P_2(\lambda)$ (see hypothesis II in Theorem \ref{main1}). See also \cite{fols} and \cite{dPS1} for more details on the soft-rescaling. Since $J\in \J_{\rm reg}^*(\lambda),$ we can apply Lemma \ref{lem_genJset} to conclude that  $\tilde w$ is a plane asymptotic to $P_{2,i}$.  In particular, by uniqueness of such planes, $w(\C) \subset \partial \U_j$.

Define the points $q_n\in \C$ by  $v_n(q_n) =p_n, \forall n.$ We observe that the sequence $q_n$ is bounded and stays away from any puncture of $\tilde w$. Indeed, 
 the usual analysis using cylinders of small area implies that every point sufficiently close to the punctures in $\Gamma$ are mapped under $v_n$ to a point arbitrarily close to $P_{2,i}\cup w(\C)\subset \partial \U_j$, see \cite[Theorem 6.6]{fols}.  Thus, after taking a subsequence, we may assume that $q_n \to q_\infty \in \D$, where $q_\infty \not \in \Gamma$ and $w(q_\infty)=p.$ Denote $\tilde w_1 = \widetilde v$ and $\Gamma_1=\Gamma$.

Now take $p_n'\to p'\neq p\in \U_j\setminus (w_1(\C\setminus \Gamma_1)\cup P_{3,j})$, and, as before, consider the $\widetilde J_n$-holomorphic planes, suitably normalized as in \eqref{paramvn} and again denoted by $\tilde v_n=(b_n,v_n)$, so that $p_n'\in v_n(\C)$. After taking a subsequence, we can assume that there exists $\Gamma_2\subset \D,$ and a $\tilde J$-holomorphic curve  $\tilde w_2=(c_2,w_2)\colon\C \setminus \Gamma_2 \to \R \times S^3$ so that $\tilde v_n \to \tilde w_2$ in $C^\infty_{\rm loc}(\C \setminus \Gamma_2)$ as $n \to +\infty$. The asymptotic limits of $\tilde w_2$ at the punctures in $\Gamma_2$ are distinct orbits in $\P_2(\lambda)$. A soft-rescaling near each  $z' \in \Gamma_2$ produces a $\widetilde J$-holomorphic plane whose asymptotic limit coincides with the asymptotic limit of $\tilde w_2$ at $z'$. Moreover, if $q_n' \in \D$  is such that $v_n(q_n') = p'_n\to p'$ then, up to a subsequence, $q_n' \to q'_\infty \notin \Gamma_2$ and $w_2(q'_\infty)=p'$. In particular, the image of $w_2$ differs from the image of $w_1$. 

Assume that $\tilde w_1$ and $\tilde w_2$ are asymptotic to the same orbit $P_{2,i}\in \P_2(\lambda)$ at   punctures $z_1\in \Gamma_1$ and $z_2 \in \Gamma_2$, respectively. Then Proposition C.0 in \cite{dPS1} implies that $w_1$ and $w_2$ must intersect each other (the intersections of $\bar w_1$ and $\bar w_2$ are isolated) and, as $C^\infty_{\rm loc}$-limits of curves that do not intersect each other, the positivity and stability of intersections of pseudo-holomorphic curves (see \cite[Appendix E]{MSbook}) give a contradiction. Thus $w_1$ and $w_2$ have mutually distinct asymptotic limits at their negative punctures.

We can repeat the process with a point $p'' \in \U_j \setminus ( w_1(\C \setminus \Gamma_1) \cup w_2(\C \setminus \Gamma_2)\cup P_{3,j})$ and find the limiting curve  associated with the sequence $\tilde v_n=(b_n,v_n)$, suitably normalized and satisfying $p''_n\in v_n(\C)\to p''$. As before, if the limiting curve $\tilde w_3=(c_3,w_3)$ has negative punctures, then it is asymptotic to an orbit in $\P_2(\lambda)$ at these punctures. These asymptotic limits are distinct from the asymptotic limits of $\tilde w_1$ and $\tilde w_2$ at any of their negative punctures. Hence, we conclude that when we vary the point $p$ in $\U_j \setminus P_{3,j}$ there can be at most $\tilde k_j$ of such limiting curves admitting negative punctures. A choice of $p\in \U_j\setminus P_{3,j}$ in the complement of the image of these curves implies that the limiting curve does not admit bubbling-off points and thus is a $\widetilde J$-holomorphic plane exponentially asymptotic to $P_{3,j}$.  The proof is complete. \end{proof}

\subsection{Another generic set of almost complex structures}

Let $J\in \J_{\rm reg}^*(\lambda)$ be chosen as in   Lemma \ref{lem_genJset}. In Section \ref{sec:3} we   constructed a transverse foliation $\F_J$ adapted to the Reeb flow of $\lambda$, so that the binding is formed by the orbits $P_{2,i}\in \P_2(\lambda),i=1,\ldots,l,$ and finitely many orbits $P_{3,j}\in \P_3^{u,-1,\leq C}(\lambda), j=1,\ldots, l+1,$ where $C>0$ is as in Proposition \ref{prop_unif_C}. 
Each $P_{2,i}$ is the boundary of a pair of rigid planes $U_{i,1},U_{i,2}\in \F_J$ so that $\S_i = U_{i,1}\cup P_{2,i} \cup U_{i,2}\subset S^3$ is a $C^1$-embedded $2$-sphere that separates $S^3$ into two distinct components. The union of these $2$-spheres is denoted by $\S$.
Each $P_{3,j},j=1,\ldots, l+1,$ lies in the component $\U_j$ of  $S^3\setminus \bigcup_{i=1}^l \S_i.$ The closure of $\U_j$ has $\tilde k_j$ boundary components, denoted by $\S_{n^j_k},k=1,\ldots,\tilde k_j,$ where  $n^j_k\in \{1,\ldots,l\}$.
The leaves of the transverse foliation $\F_J$ are projections to $S^3$ of embedded finite energy $\widetilde J$-holomorphic curves, where $\widetilde J$ is the almost complex structure in $\R \times S^3$ associated with the pair $(\lambda,J)$.  For every $i=1,\ldots,l,$ there exist two embedded $\widetilde J$-holomorphic planes $\tilde u_{i,j}=(a_{i,j},u_{i,j})\colon\C \to \R \times S^3,j=1,2,$ asymptotic to $P_{2,i}\in \P_2(\lambda)$ at $\infty$, and so that
$
U_{i,j} = u_{i,j}(\C)\subset \S_i,  \forall i,j.
$

Let $\varepsilon>0$ be small. For each $j=1,\ldots, l+1 $ and for every $k=1,\ldots,\tilde k_j,$ take a compact  $\varepsilon$-neighborhood $\U^\varepsilon_{j,n^j_k} \subset {\rm closure}(\U_j)$ of $\S_{n_k^j}$.  
Abbreviate
$$
\U^\varepsilon_{j} = \bigcup_{k=1}^{\tilde k_j} \U_{j,n^j_k}^\varepsilon\subset {\rm closure}(\U_j) \ \ \ \ \ \mbox{ and }\ \ \ \ \ \U^\varepsilon = \bigcup_{j=1}^{l+1} \U^\varepsilon_j \subset S^3.
$$
Denote by $\J_J^\varepsilon(\lambda) \subset \J(\lambda)$ the space of $d\lambda$-compatible almost complex structures $J'$ satisfying
$
J' = J \ \ \mbox{ in } \ (S^3 \setminus \U^\varepsilon )\cup \S.
$
The set $\J_J^\varepsilon(\lambda)$ inherits the $C^\infty$-topology from $\J(\lambda)$.  Denote by $\widetilde J'$ the almost complex structure on $\R \times S^3$ determined by $\lambda$ and $J'\in \J_J^\varepsilon(\lambda)$. In particular, 
$\tilde u_{i,j}$ is $\widetilde J'$-holomorphic for every $J'\in \J_J^\varepsilon(\lambda).$

Taking $\varepsilon>0$ sufficiently small, we can assure that for every $j=1,\ldots, l+1,$ there exists an embedded $\widetilde J$-holomorphic plane $\tilde w_j=(c_j,w_j)\colon\C \to \R \times S^3$, which is exponentially asymptotic to $P_{3,j}$ at $\infty$ and satisfies
$
w_j(\C)\subset \U_j \setminus \U^\varepsilon.  
$
In particular, $\tilde w$ is also $\widetilde J'$-holomorphic for every  almost complex structure $\widetilde J'$ associated with $\lambda$ and $J'\in \J_J^\varepsilon(\lambda)$.

The foliation $\F_J$ constructed in the previous section may contain regular leaves which are projections to $S^3$  of a $\widetilde J$-holomorphic curve $\tilde w=(b,w)\colon\C \setminus \Gamma \to \R \times S^3,$ satisfying
\begin{itemize}
    \item $\Gamma \neq \emptyset$. 
    \item $\int_{\C \setminus \Gamma} w^*d\lambda>0.$
    \item $\infty$ is a positive puncture of $\tilde w$ and every puncture in $\Gamma$ is negative.
    \item  $\exists \, j\in \{1,\ldots,l+1\}$ so that $\tilde w$ is exponentially asymptotic to $P_{3,j}$ at $\infty$ and to distinct orbits in $\P_2(\lambda)$ at the punctures in $\Gamma$ and  $w(\C\setminus \Gamma)\subset \U_j.$
\end{itemize} 
The Fredholm index of $\tilde w$ is
$
 {\rm Ind}(\tilde w)   = \cz(P_{3,j}) - \sum_{z \in \Gamma} \cz(P_{2,z}) - 1 + \#\Gamma
 = 2 - \#\Gamma.
$
Here, $P_{2,z}\in \P_2(\lambda)$ is the asymptotic limit of $\tilde w$ at $z\in \Gamma$. Therefore,
$\# \Gamma >1$ implies that  ${\rm Ind}(\tilde w) \leq 0.$

The following theorem, based on the weighted Fredholm theory developed in \cite{props3}, states that it is always possible to find $J'\in \J_J^\varepsilon(\lambda)$, which is $C^\infty$-close to $J$, so that the Fredholm index of $\tilde w$ as above is at least $1$. In particular, such curves are rigid cylinders asymptotic to some $P_{3,j}$ at the positive puncture and to an orbit in $\P_2(\lambda)$ at the negative puncture.    See Figure \ref{fig:genericJ}.

\begin{thm}[Hofer-Wysocki-Zehnder \cite{props3}, Dragnev \cite{Drag}]\label{thm_props3} 
Given $J \in \mathcal{J}_{\rm reg}^*(\lambda)$ and $\varepsilon>0$ sufficiently small, there exists a residual set $\J_{J,{\rm reg}}^\varepsilon(\lambda) \subset \J^\varepsilon_J(\lambda)$ in the $C^\infty$-topology so that the following holds: let $J'\in \J_{J,{\rm reg}}^\varepsilon(\lambda)$ and let $\tilde v=(b,v)\colon\C \setminus \Gamma \to \R \times S^3,\ \Gamma\neq \emptyset,$ be a somewhere injective finite energy $\widetilde J '$-holomorphic curve, where $\widetilde J'$ is the almost complex structure in $\R \times S^3$ induced by the pair $(\lambda,J')$. Assume that all punctures in $\Gamma$ are negative and that $\tilde v$ is exponentially asymptotic to some $P_{3,j}$ at the positive puncture $+\infty$   and to distinct orbits in $\P_2(\lambda)$ at the punctures in $\Gamma$. Then $\#\Gamma=1$. In particular, $\tilde v$ is a $\widetilde J'$-holomorphic cylinder asymptotic to $P_{3,j}$ at  $\infty$ and to an orbit in $\P_2(\lambda)$ at its negative puncture.
\end{thm}

\begin{figure}[ht!]
\centering
\includegraphics[scale=0.4]{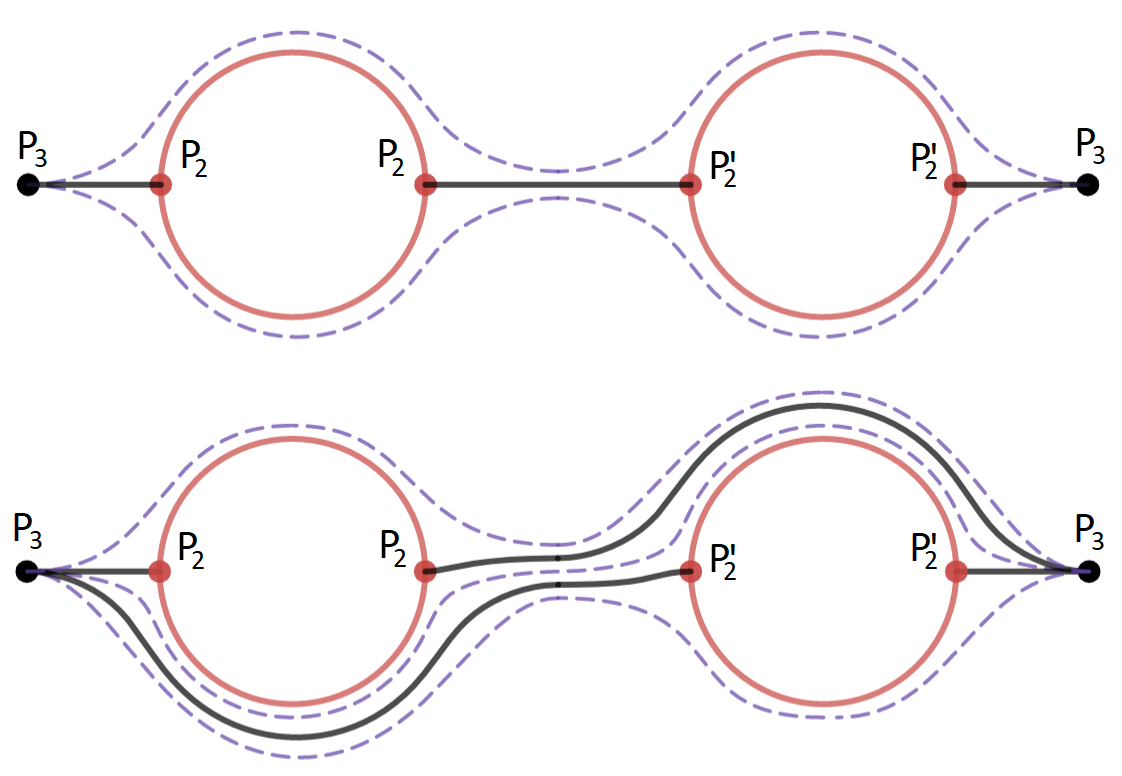} 
\caption{(top) A non-generic $\widetilde J$-holomorphic curve $\tilde v\colon \C \setminus \{z_2,z_2'\} \to \R \times S^3,$ asymptotic to $P_3$ at its positive puncture $\infty$ and to $P_2$ and $P_2'$ at the negative punctures $z_2$ and $z_2'$, respectively. (bottom) For a small generic perturbation of $J$, this curve unfolds into generic $\widetilde J$-holomorphic cylinders, one connecting $P_3$ to $P_2$ and the other connecting $P_3$ to $P_2'$.}
\label{fig:genericJ}
\end{figure}

\subsection{Finding the desired transverse foliation}

Take $\varepsilon>0$ sufficiently small and let $J'\in \J^{\varepsilon}_{J,{\rm reg}}(\lambda)$ be as in Theorem \ref{thm_props3}. Then for every $i=1,\ldots,l,$ the rigid planes $U_{i,1},U_{i,2}$ are projections of embedded $\widetilde J'$-holomorphic planes $\tilde u_{i,1},\tilde u_{i,2}\colon\C \to \R \times S^3$, which are  asymptotic to $P_{2,i}$ at $\infty$. 
For every $j=1,\ldots,l+1,$ there exists an embedded $\widetilde J'$-holomorphic plane $\tilde w_j\colon\C \to \R\times S^3$ asymptotic to $P_{3,j}$ at $\infty$.
Our goal is to construct a transverse foliation which contains only planes and cylinders asymptotic to the orbits $P_{2,i},i=1,\ldots,l,$ and $P_{3,j},j=1,\ldots,l+1$, and so that the given planes are part of the regular leaves.

Denote by $\M_{\widetilde J'}(P_{3,j})$ the space of $\widetilde J'$-holomorphic planes exponentially asymptotic to $P_{3,j}$ at $\infty$. We identify those planes which have the same image in $S^3$, that is, $\tilde w_1 \sim \tilde w_2$ if there exists $a,b\in \C$ and $c\in \R$ so that 
$\tilde w_1(z) = c+ \tilde w_2(az +b), \forall z\in \C,$
where $c+\tilde w_2$ is the  $c$-translation of $\tilde w_2$  in the $\R$-direction. The intersection theory developed in~\cite{props2, props3,Si2} implies the following proposition.

\begin{prop}\label{prop_fredholm}Let $j\in \{1,\ldots,l+1\}$. Then the following assertions hold:
\begin{itemize}
    \item[(i)]  $\M_{\widetilde J'}(P_{3,j})$ has the structure of a $1$-dimensional smooth manifold.
    \item[(ii)] if $[\tilde w=(c,w)]\in \M_{\widetilde J'}(P_{3,j})$, then 
$
w(\C) \subset \U_j. 
$
    \item[(iii)] if $[\tilde w_1=(c_1,w_1)]\neq[\tilde w_2=(c_2,w_2)]\in \M_{\widetilde J'}(P_{3,j}),$ then 
$
w_1(\C) \cap w_2(\C) = \emptyset.
$
\end{itemize}
\end{prop}

The last step toward the proof of Theorem \ref{main1} is   the following assertion.

\begin{prop}\label{prop:convergence_of_planes}
Fix $j\in \{1,\ldots,l+1\}$ and let $\tilde k_j$ be the number of components of $\partial \U_j$. Then there exist $\tilde k_j$ embedded $\widetilde J'$-holomorphic cylinders $$
\tilde v_{j,m}=(b_{j,m},v_{j,m})\colon\C \setminus \{0\} \to \R \times S^3, \ \ \ \ \forall m=1,\ldots, \tilde k_j,
$$ 
and $\tilde k_j$ families of embedded $\widetilde J'$-holomorphic planes  $\tilde w_{j,m,\tau}\in \M_{\widetilde J'}(P_{3,j}),$ $$\tilde w_{j,m,\tau}=(c_{j,m,\tau},w_{j,m,\tau})\colon\C \to \R \times S^3,\ \ \tau\in (0,1), \ \ \ \ \forall m=1,\ldots,\tilde k_j,$$ so that the following properties hold:
\begin{itemize}
    \item[(i)] $\infty$ is a positive puncture of $\tilde v_{j,m}$, where it is exponentially asymptotic to $P_{3,j}$, and $0\in \C$ is a negative puncture of $\tilde v_{j,m},$ where it is asymptotic to $P_{2,n^j_m}, \forall m=1,\ldots,\tilde k_j$. Moreover,
    $v_{j,m}(\C \setminus \{0\}) \subset \U_j,  \forall m,
    $
    and 
    $$
    v_{j,m}(\C \setminus \{0\}) \cap v_{j,n}(\C \setminus \{0\}) = \emptyset, \quad  \forall m \neq n.
    $$
    \item[(ii)] $\infty$ is a positive puncture of $\tilde w_{j,m,\tau},$ where it is exponentially asymptotic to $P_{3,j}$, and
    $
    w_{j,m,\tau}(\C) \subset \U_j,  \forall m.
    $
    \item[(iii)] $v_{j,m}(\C \setminus \{ 0 \})$ and $w_{j,m,\tau}(\C),$ $m=1,\ldots, \tilde k_j,\tau\in (0,1),$ are regular leaves of a transverse foliation of $\U_j $.
    \item[(iv)] for every $m\in \{1,\ldots,\tilde k_j\}$,  $\tilde w_{j,m,\tau}$ converges in the SFT-sense $($see \cite{BEHWZ03}$)$ to  
    $
    \tilde v_{j,m} \oplus \tilde u_{n^j_m ,1}$ as  $\tau \to 0^+ $
    and  to
    $
    \tilde v_{j,m+1}\oplus \tilde u_{n^j_{m+1},2}$ as $\tau\to 1^-.$ Here, $\tilde u_{n^j_m,1}=(a_{n^j_m,1},u_{n^j_m,1})$ and $\tilde u_{n^j_{m+1},2}=(a_{n^j_{m+1},2},u_{n^j_{m+1},2})$ are rigid planes asymptotic to orbits in $\P_2(\lambda)$. 
    Moreover, given neighborhoods $\V_{j,m,1}\subset {\rm closure}(\U_j)$ of $u_{n^j_m,1}(\C)\cup P_{2,n^j_m} \cup v_{j,m}(\C \setminus \{0\})$ and $\V_{j,m+1,2} \subset {\rm closure}(\U_j)$ of $u_{n^j_m,2}(\C) \cup P_{2,n^j_{m+1}}\cup v_{j,m+1}(\C \setminus \{0\}),$ there exists $\delta>0$ so that $w_{j,m,\tau}(\C) \subset \V_{j,m,1},  \forall 0<\tau<\delta,$ and $w_{j,m,\tau}(\C) \subset \V_{j,m+1,2},  \forall 1-\delta< \tau<1.$  By convention, $\tilde k_j+1 \equiv 1$. 
\end{itemize}
\end{prop}

\begin{proof}
The choice of $\varepsilon>0$ sufficiently small and the definition of $\widetilde J'$ guarantee  the existence of an embedded $\widetilde J'$-holomorphic plane $\tilde w=(a,w)\colon \C \to \R \times S^3$, which is exponentially asymptotic to $P_{3,j}$  and satisfies $w(\C) \subset \U_j \setminus P_{3,j}$. By Proposition \ref{prop_fredholm}, $\tilde w\in \M_{\widetilde J'}(P_{3,j})$ lies in a $1$-parameter family $\tilde w_\tau=(c_\tau,w_\tau)\colon \C \to \R \times S^3, \tau\in (-\delta, \delta),$ of embedded $\widetilde J'$-holomorphic planes which are exponentially asymptotic to $P_{3,j}$ for some $\delta$ small enough.  For each $\tau\in (-\delta,\delta),$ $w_\tau \colon \C \to S^3$ is an embedding transverse to the Reeb flow, and $w_{\tau_1}(\C) \cap w_{\tau_2}(\C) =\emptyset, \forall \tau_1 \neq \tau_2.$ Consider the maximal one-parameter family of planes containing the family $\tilde w_\tau, \tau\in (-\delta,\delta)$, i.e.\ the smooth family of embedded $\widetilde J'$-holomorphic planes $\tilde w_\tau\in \M_{\widetilde J'}(P_{3,j})$ so that the family $w_\tau(\C)$ fills the maximal volume in $S^3$. Parametrize this maximal family by 
\begin{equation}\label{familywtau}
\tilde w_\tau=(c_\tau,w_\tau)\colon\C \to \R \times S^3,\ \ \ \ \ \tau\in (0,1).
\end{equation} Such a family is not compact since otherwise the $S^1$-family of such planes in the complement of $P_{3,j}$ determines an open book decomposition of $S^3$ whose binding is $P_{3,j}$ and, as a consequence of the transversality of the pages with respect to the flow, the orbits in $\P_2(\lambda)\neq \emptyset$ are linked with $P_{3,j}$, a contradiction. 

We fix the convention that $\tau$ increases in the direction of the Reeb flow and that the Reeb vector field points inside $\U_j$ along $U_{n^j_m,1}=u_{n^j_m,1}(\C)$ and outside $\U_j$ along $U_{n^j_m,2}=u_{n^j_m,2}(\C), \forall m=1,\ldots, \tilde k_j.$

For each $\tau\in (0,1),$ we choose the following normalization of $\tilde w_\tau$. Consider a small compact tubular neighborhood $\U \subset \U_j$  of $P_{3,j}$, so that
\begin{itemize}
    \item $\U$ contains no periodic orbits that are contractible in $\U$.
     \item there exists no Reeb orbit    $P\subset \U$ of $\lambda$ which is geometrically distinct from $P_{3,j}$, is homotopic to $P_{3,j}$ in $\U$ and satisfies $\link(P,P_{3,j})=0$.
     \item $w_\tau (\C \setminus \D)\subset  \U.$
    \item $w_\tau(1)\in \partial \U.$
    \item  $w_\tau(z^*_\tau) \in \partial \U$ for some $z^*_\tau \in \partial \D$  satisfying ${\rm Re}(z^*_\tau) \leq 0.$
    \item $c_\tau(2) = 0$ 
\end{itemize}

Let us study the compactness properties of the family \eqref{familywtau} under the normalizations above.  Take a strictly increasing sequence $\tau_n \to 1^-$ and denote $\tilde w_n=\tilde w_{\tau_n} .$ Arguing as in the proof of Proposition \ref{prop_rigid_planes}, we take a subsequence of $\tilde w_n$, still denoted by $\tilde w_n$, which admits a sequence $z_n \in \C$ so that $|\nabla \tilde w_n(z_n)| \to +\infty$ as $n \to+ \infty$. Then $z_n$ is bounded and, up to a subsequence, $z_n$ converges to a point in $\D$. 
Moreover, extracting a subsequence, we can assume that the set of bubbling-off points $\Gamma\subset \D$ is finite. The normalizations of $\tilde w_n$ imply the existence of a $\widetilde J$-holomorphic curve $\tilde v=(b,v)\colon \C \setminus \Gamma \to \R \times S^3,$ so that $\tilde w_n \to \tilde v$ in $C^\infty_{\rm loc} (\C \setminus \Gamma)$ as $n \to +\infty$. By Lemma \ref{lem_p3j}, $\tilde v$ is non-constant and exponentially converges to $P_{3,j}$ at the positive puncture $\infty$. Every puncture in $\Gamma$ is negative. 

Observe that $\Gamma \neq \emptyset$. Indeed, if $\Gamma =\emptyset$, then $\tilde v$ is an embedded finite energy $\widetilde J'$-holomorphic plane exponentially asymptotic to $P_{3,j}$. In particular, $[\tilde v]\in \M_{\widetilde J'}(P_{3,j})$. By Proposition \ref{prop_fredholm}, the family \eqref{familywtau} can  be continued, contradicting its maximality and the fact that $\tau_n\to 1^-$.

The image of $v$ is contained in $\U_j$ since, otherwise, by stability and positivity of intersections of pseudo-holomorphic curves (see \cite[Appendix E]{MSbook}) in $\R \times S^3$,  for every $n$ sufficiently large, $w_n$ intersects one of the rigid planes asymptotic to  some orbit in $\P_2(\lambda)$ which implies that  $\tilde w_n$ intersects the corresponding $\tilde J'$-holomorphic   plane, absurd. Hence, we obtain   $v(\C \setminus \Gamma) \subset \U_j$.

Arguing again as in the proof of Proposition \ref{prop_rigid_planes}, we  obtain that    $\tilde v$ is asymptotic to $P_{3,j}$ at $+\infty$ and to an orbit  in $\P_2(\lambda)$ at each negative puncture  in $\Gamma\neq \emptyset$.   By Theorem \ref{thm_props3}, $\Gamma =1$ and thus $\tilde v$ is a $\widetilde J'$-holomorphic cylinder, exponentially asymptotic to $P_{3,j}$ at $+\infty$, and asymptotic to   $P_{2,n^j_m}\in \P_2(\lambda)$ for some $m\in \{1, \ldots, \tilde k_j\}$ at its unique negative puncture. We can assume that $\Gamma = \{0\}$, so that   $\tilde v$ is asymptotic to $P_{2,n^j_m}$ at $0$.

Performing a soft-rescaling of $\tilde w_n$ near the negative puncture $0$, we find a new $\widetilde J'$-holomorphic curve $\tilde u=(a,u)\colon\C \setminus \Gamma_u \to \R \times S^3$, which is asymptotic to $P_{2,n^j_m}$ at $+ \infty$ and   to covers of orbits in $\P_2(\lambda)$  at its punctures in $\Gamma_u$. As before, the generic choice of $  J$  and the soft-rescaling process imply that $\Gamma_u=\emptyset$,    and hence $\tilde u$ is a plane asymptotic to $P_{2,n^j_m}$. 

We conclude that $\tilde w_n$ converges to a $2$-level building $\mathcal{B}$ of embedded $\widetilde J'$-holomor\-phic curves. The top level contains the cylinder $\tilde v\colon\C \setminus \{0\} \to \R \times S^3$, exponentially asymptotic to $P_{3,j}$ at $\infty$, and asymptotic to $P_{2,n^j_m}$ at its negative puncture $0\in \C.$ The bottom level consists of a  plane $\tilde u\colon\C \to \R \times S^3,$  asymptotic to  $P_{2,n^j_m}$ at $\infty.$
The usual analysis of cylinders with small area, see \cite{small} and also \cite[Proposition 9.5]{dPS1}, implies that given a neighborhood $\V\subset S^3$ of $v(\C\setminus\{0\}) \cup P_{2,n^j_m} \cup u(\C)$ we have $w_n(\C) \subset \V$ for $n$ sufficiently large.

The uniqueness of $\widetilde J'$-holomorphic planes asymptotic to orbits in $\P_2(\lambda)$, see \cite[Proposition C.-3]{dPS1}, and the fact that $\tau_n\to 1^-$ implies that $\tilde u=\tilde u_{n^j_{m},2}$.

For every large $n_0$, we can patch 
$
w_{n_0}(\C)\cup P_{3,j} \cup v(\C\setminus\{0\})\cup P_{2,n^j_m} \cup u(\C)
$
to form a topological embedded $2$-sphere $S_{n_0}\subset {\rm closure}(\U_j)$. The $2$-sphere $S_{n_0}$ separates $S^3$ into two disjoints subsets, one of them, denoted by $\mathcal{A}_{n_0}$, contains $w_n(\C)$ for $n>n_0$. The volume of $\mathcal{A}_{n_0}$ tends to $0$ as $n_0\to +\infty$. It then follows that for every sequence $\tau_n \to 1^-$, the image of $w_{\tau_n}(\C)$ is contained in $\mathcal{A}_{n_0}$ for every $n$ sufficiently large. This implies that the limiting building $\B$ is the unique SFT-limit of $\widetilde w_\tau$ as $\tau \to 1^-$.

 According to C. Wendl \cite[Theorem 1]{Wendl1}, the curves $\tilde v$ and $\tilde u$ are automatically transverse. In particular,  we can glue $\tilde v=:\tilde v_{j,m}$ with $\tilde u_{n^j_m,1}$ along $P_{2,n^j_m}$ to form a new family of embedded $\widetilde J'$-holomorphic planes, all of them exponentially asymptotic to $P_{3,j}$. See \cite[Section 7]{Nel13} or \cite[Section 10]{WendlSFT}. Such planes lie in a maximal family of planes in $\M_{\widetilde J'}(P_{3,j})$ and will be denoted by
$
\tilde w_\tau'=(c_\tau',w_\tau')\colon\C \to \R \times S^3,  \tau\in(0,1),
$
so that $w'_\tau(\C) \subset \U_j, \forall \tau$. Under our parametrizations,  $\tilde w'_\tau$ converges to the holomorphic building formed by $\tilde v$ and $\tilde u_{n^j_m,1}$ as $\tau \to 0^+$. In our notation the family $\tilde w'_\tau$ now corresponds to $\tilde w_{j,m,\tau}, \tau \in (0,1).$

If $\tilde k_j = 1,$ then the family $\tilde w'_\tau,\tau\in(0,1),$ coincides with the family $\tilde w_\tau,\tau\in (0,1),$ and the compactness properties above show that the  $\cup_{\tau\in (0,1)} w_\tau(\C),$ is open and closed in $\U_j\setminus (P_{3,j}\cup v(\C \setminus \{0\}))$ and thus coincides with $\U_j \setminus (P_{3,j}\cup v(\C \setminus \{0\}))$. If $\tilde k_j>1$, then the families $\tilde w_\tau,\tau\in (0,1),$ and $\tilde w'_\tau,\tau\in (0,1),$ do not coincide, and  we consider the compactness properties of the family $\tilde w'_\tau$ as $\tau \to 1^-$. As before, this family converges to a building whose top level consists of an embedded $\widetilde{J}'$-cylinder $\tilde v'\colon\C \setminus \{0\} \to \R \times S^3,$ which is exponentially asymptotic to $P_{3,j}$ at $\infty$ and asymptotic to some other $P_{2,i}\in \P_2(\lambda)$ at $0$, and whose lower level consists of an embedded $\widetilde{J}'$-holomorphic plane $\tilde u'$, which is asymptotic to  $P_{2,i}$  at $+\infty$.

We necessarily have $P_{2,n^j_{m}} \neq P_{2,i}=:P_{2,n^j_{m+1}}$ since the families $\tilde w_\tau$ and $\tilde w_\tau'$ are distinct and hence, by the uniqueness and intersection properties of the $\widetilde J'$-holomorphic planes exponentially asymptotic to $P_{3,j}$, points of $w'_\tau(\C)$ cannot accumulate at $P_{2,n^j_{m}}$ as $\tau \to 1^-$. It follows  from the normalizations of  $\tilde w'_\tau , \tau\in (0,1),$ that  $\tilde u'=:\tilde u_{n^j_{m+1},2}$ and, as before, we glue $\tilde v'=:\tilde v_{j,m+1}$ with $\tilde u_{n^j_{m+1},1}$ to obtain a new maximal family of embedded $\widetilde J'$-holomorphic planes $\tilde w''_\tau,\tau\in (0,1),$ which are exponentially asymptotic to $P_{3,j}$. 

If $\tilde k_j=2$, then the new family $\tilde w''_\tau$ coincides with the family $\tilde w_\tau$ and $v''(\C\setminus\{0\})=v(\C\setminus \{0\})$, where $\tilde v''$ is the embedded $\widetilde{J}'$-holomorphic cylinder consisting of the top level of  a holomorphic building associated with the family $\tilde w''_\tau$, as $\tau \to 1^-$, and  
\[
\bigcup_{\tau\in (0,1)} w_\tau(\C) \cup \bigcup_{\tau\in (0,1)} w'_\tau(\C)
\]
fills $\U_j\setminus \left( P_{3,j}\cup v(\C \setminus \{0\})\cup v'(\C \setminus \{0\}) \right)$. Otherwise, we glue  $v''=:\tilde v_{j,m+2}$ with the  rigid plane $\tilde u_{n^j_{m+2},1}$ and continue in a similar manner.  It has to stop after a finite number of steps. Indeed, the number of such families of $\widetilde J'$-holomorphic planes asymptotic to $P_{3,j}$ is precisely $\tilde k_j$, the number of components in $\partial\U_j$.

We  conclude that there exist $\tilde k_j$ embedded $\widetilde J'$-holomorphic cylinders $\tilde v_{j,m}\colon\C \setminus \{0\} \to \R \times S^3,$ $m=1, \ldots, \tilde k_j,$ which are exponentially asymptotic to $P_{3,j}$ at the positive puncture $\infty$ and to   $P_{2,n^j_m}\subset \S_{n^j_m}$ at their negative puncture $0$. In the complement of such cylinders, there exist $\tilde k_j$ families of embedded $\widetilde J'$-holomorphic planes $\tilde w_{j,m,\tau},\tau\in(0,1)$, which are exponentially asymptotic to $P_{3,j}$ at $\infty$. Moreover, each  family converges to the holomorphic building formed by $\tilde v_{j,m} \oplus \tilde u_{n^j_m,1}$ as $\tau \to 0^+$ and to the building $\tilde v_{j,m+1}\oplus \tilde u_{n^j_{m+1},2}$ as $\tau \to 1^-$ for every $m=1,\ldots,\tilde k_j$. Here, $\tilde k_j+1 \equiv 1$.  This finishes the proof. \end{proof}

Summarizing the results from Propositions \ref{prop_rigid_plane0}, \ref{prop_specialP3} and  \ref{prop:convergence_of_planes}, we obtain the following statement, which implies Theorem \ref{main1}.

\begin{thm}\label{mainfef}
Given a weakly convex contact form $\lambda = f \lambda_0$ on  $(S^3,\xi_0)$ satisfying hypotheses I-III  of Theorem \ref{main1}, there exists a dense subset $ \widetilde{\J}_{\rm reg}  (\lambda) \subset \J(\lambda)$   in the $C^{\infty}$-topology, so that  for  every $J \in \widetilde{\J}_{\rm reg}  (\lambda) ,$ the pair $(\lambda, J)$ admits a stable finite energy foliation $\tilde{\mathcal{F}}$ satisfying the following properties:

\begin{enumerate}

\item[(i)] For  each $i=1,\ldots,l,$
there exists a pair of embedded  finite energy $\widetilde J$-holomorphic planes $\tilde u_{i,1}=(a_{i,1},u_{i,1}),$ $\tilde u_{i,2}=(a_{i,2},u_{i,2})\colon \C \to \R \times S^3$ which are asymptotic to $P_{2,i}.$  The union $\mathcal{S}_i=u_{i,1}(\C) \cup P_{2,i} \cup u_{i,2}(\C)$ is a $C^1$-embedded $2$-sphere separating $S^3$ into two components and $\mathcal{S}_i \cap \mathcal{S}_j =\emptyset, \forall i \neq j$.   Every component $\U_j, j=1,\ldots, l+1,$ of $S^3 \setminus \cup_{i=1}^l \S_i$ contains an index-$3$ orbit $P_{3,j} $ satisfying the linking properties given in Proposition \ref{prop_specialP3}.

\item[(ii)]  Given $j \in \{1, \ldots, l+1\},$  denote by $\tilde k_j$ the number of boundary components of   $\U_j$, denoted by  $\S_{n^j_k},k=1,\ldots,\tilde k_j,$ where  $n^j_k\in \{1,\ldots,l\}$.
\begin{enumerate}
    \item[(a)] 
Then there exist $\tilde k_j$ embedded finite energy $\widetilde J$-holomorphic cylinders 
\[ 
\tilde v_{j,m}=(b_{j,m},v_{j,m})\colon\C \setminus \{0\} \to \R \times S^3, \ \  \forall m=1,\ldots, \tilde k_j,
\]
which is exponentially asymptotic to $P_{3,j}$ at its positive puncture $+\infty$ and to $P_{2, n^j_m}$ at its negative puncture $0, \forall m = 1, \ldots, \tilde k_j. $ Moreover, 
they satisfy
    $v_{j,m}(\C \setminus \{0\}) \subset \U_j,  \forall m,
    $
    and 
    $$
    v_{j,m}(\C \setminus \{0\}) \cap v_{j,n}(\C \setminus \{0\}) = \emptyset, \quad  \forall m \neq n.
    $$

\item[(b)] The complement 
$
(\R \times \U_j) \setminus \cup_{m=1}^{\bar k_j} \tilde v_{j,m}(\C \setminus \{ 0 \}  )
$
is foliated by   $\bar k_j$ families of embedded finite energy $\widetilde J$-holomorphic planes    
\[
\tilde w_{j,m,\tau}=(c_{j,m,\tau},w_{j,m,\tau})\colon\C \to \R \times S^3,\ \tau\in (0,1), \ \  \forall m=1,\ldots,\tilde k_j,
\]
 exponentially asymptotic to $P_{3,j}$ at its   positive puncture $+\infty.$
Moreover, each plane $\tilde w_{j,m\tau}$ satisfies the compactness property described in      Proposition \ref{prop:convergence_of_planes}-(iv).
 
  \end{enumerate}

\item[(iii)] 
Every finite energy $\widetilde J$-holomorphic curve  described above  satisfies   ${\rm wind}_\pi =0,$ so that its projection  to $S^3$ is  transverse to the Reeb vector field of $\lambda.$ 
\end{enumerate}
Consequently, the projection $\mathcal{F}$ of the finite energy foliation $\tilde{\mathcal{F}}$ to $S^3$ provides the weakly convex foliation as in Theorem \ref{main1}.
\end{thm}

\section{Transition maps}\label{sec:infinitetwist}
 
 Throughout this section, we assume that the Reeb flow of $\lambda$ is real-analytic.
 Let $\F$ be a genus zero transverse foliation adapted to the Reeb flow of  $\lambda=f\lambda_0$ as in Theorem \ref{main1}. We shall set some notations to represent the elements associated with $\F$, see Figure \ref{fig:examplenotation}. First recall that all the orbits in $\P_2(\lambda)$ are binding orbits of $\F$ and, for each $i=1,\ldots,l,$ the orbit $P_{2,i}\in \P_2(\lambda)$ bounds two  rigid planes $U_{i,1},U_{i,2} \in \F$, so that the embedded $2$-sphere 
$
\S_i = U_{i,1} \cup P_{2,i} \cup U_{i,2}\subset S^3 
$
is $C^1$. The $2$-spheres $\S_i,i=1,\ldots,l,$ are  mutually disjoint and each one of them separates $S^3$ into two components. In this way, the complement of their union is formed by $l+1$ components, denoted by 
$
\U_j \subset S^3 \setminus \cup_{i=1}^l \S_i, j=1,\ldots,l+1.
$

Inside $\U_j,$ there exist a binding orbit $P_{3,j}\in \P_3^{u,-1}(\lambda)$ and $\bar k_j$ one-parameter families of planes asymptotic to $P_{3,j}$, denoted
$
\F_{j,1},\ldots,\F_{j,\tilde k_j}\subset \F, 
j=1,\ldots,l+1,$
where $\bar k_j\in \N^*$ coincides with the number of components in $\partial \U_j$. The family $\F_{j,k}=(\F_{j,k,\tau})_\tau$ is parametrized by $\tau \in (0,1).$ It breaks, as $\tau \to 0^+$, onto a rigid plane
$
U_{j,k}^{-} \in \{U_{1,1},U_{1,2}, \ldots,U_{l,1},U_{l,2}\}\subset \F,
$ asymptotic to a binding orbit
$
\P_{j,k}^{-} \in \P_2(\lambda), 
$
and a rigid cylinder
$
V_{j,k}^{-} \in \F,
$
 asymptotic to $P_{3,j}$ at its positive puncture and to $\P_{j,k}^{-}$ at its negative puncture.
In a similar manner, as $\tau\to 1^-$,  it  breaks onto a rigid plane
$
U_{j,k}^{+} \in \{U_{1,1},U_{1,2}, \ldots,U_{l,1},U_{l,2}\}\subset \F,
$
asymptotic to a binding orbit
$
\P_{j,k}^{+} \in \P_2(\lambda), 
$
and a rigid cylinder
$
V_{j,k}^{+} \in \F,
$
 asymptotic to $P_{3,j}$ at its positive puncture and to $\P_{j,k}^{+}$ at its negative puncture.

Let $
\mathcal{C} := \{(j,k)\in \N ^* \times \N^*  \mid j=1,\ldots, l+1, \; k=1,\ldots , \tilde k_j\}.
$
It parametrizes the space of families of planes asymptotic to the index $3$ binding orbits: each  $(j,k)\in \mathcal{C}$ corresponds to a family $\F_{j,k}$ of planes in $\U_j$ asymptotic to $P_{3,j}$. After relabelling the families of planes, we can assume that $\mathcal{P}^+_{j,k} = \mathcal{P}^-_{j,k+1}, \forall k \mod \bar k_j.$

For every $(j,k)\in \mathcal{C}$, there exist  a unique branch of the local unstable manifold of $\P_{j,k}^{-}$ and a unique  branch of the local stable  manifold of $\P_{j,k}^{+}$, which intersect $\F_{j,k,\tau}$ for $\tau$ sufficiently close to $0$ and $\tau$ sufficiently close to $1$, respectively. Denote these local branches by
$
\B^u_{j,k}\subset W^u_{\rm loc}(\P_{j,k}^{-})$
and $\B^s_{j,k}\subset W^s_{\rm loc}(\P_{j,k}^{+}).
$
We shall fix planes
$
\F_{j,k}^{-} := \F_{j,k,\tau_-}^{-}$ and  $ \F_{j,k}^{+} := \F_{j,k,\tau_+}^{+},$  
where $\tau_->0$ is sufficiently close to $0$ and $\tau_+<1$ is sufficiently close to $1$. In particular, the intersections of $\F_{j,k}^{-}$ and $\F_{j,k}^{+}$ with $\B^u_{j,k}$ and $\B^s_{j,k}$, respectively, are simple closed curves, denoted by 
\begin{equation}\label{eq:CC}
C^u_{j,k} := \F_{j,k}^{-} \cap \B^u_{j,k}\ \  \ \ \mbox{ and }\ \ \ \  C^s_{j,k} := \F_{j,k}^{+} \cap \B^s_{j,k}.
\end{equation}
The closed disks bounded by $C^u_{j,k}$ and $C^s_{j,k}$ will be denoted by 
\begin{equation}\label{eq:DD}
\mathcal{D}^u_{j,k} \subset \F_{j,k}^{-} \ \ \ \ \mbox{ and } \ \ \ \ \mathcal{D}^s_{j,k} \subset \F_{j,k}^{+},
\end{equation}
respectively. These disks have $d\lambda$-area equal to the actions of $\P_{j,k}^{-}$ and $\P_{j,k}^{+},$ respectively.

  \begin{figure}[t]
 \begin{center}
\begin{tikzpicture}[scale=1]
  \draw[thick] (1,1.5) arc (180:360: 1cm and  1cm);    \draw[thick]  (1,1.5 ) arc (180:0: 1cm and 1cm);   \draw[thick] (1,-1.5) arc (180:360: 1cm and  1cm);   \draw[thick]  (1 ,-1.5 ) arc (180:0: 1cm and 1cm);   \draw[thick] (-1,0) arc (180:360: 3cm and  3cm);   \draw[thick]  (-1,0) arc (180:0: 3cm and 3cm);
 \draw[thick] (-2.5,0) to (0.5,0);  \draw[thick] (3.5,0) to (6.5,0);  \draw[thick] (1,1.5) to (1.7,1.5);  \draw[thick] (3,1.5) to (2.3,1.5);  \draw[thick] (1,-1.5) to (1.7,-1.5);  \draw[thick] (3,-1.5) to (2.3,-1.5);     \draw[thick] (0.5,0) [out=80, in=255]  to (0.78,1.2);    \draw[thick] (0.78,1.2) [out=75, in=180]  to (0.97,1.5);    \begin{scope}[yscale=-1]   \draw[thick] (0.5,0) [out=80, in=255]  to (0.78,1.2);    \draw[thick] (0.78,1.2) [out=75, in=180]  to (0.97,1.5);    \end{scope}
  \begin{scope}[xscale=-1, xshift=-4cm]     \draw[thick] (0.5,0) [out=80, in=255]  to (0.78,1.2);    \draw[thick] (0.78,1.2) [out=75, in=180]  to (0.97,1.5);    \begin{scope}[yscale=-1]   \draw[thick] (0.5,0) [out=80, in=255]  to (0.78,1.2);    \draw[thick] (0.78,1.2) [out=75, in=180]  to (0.97,1.5);    \end{scope} \end{scope}
 \draw[densely dotted] (-2.5,0) to (-3.5,0);\draw[densely dotted] (-2.5,0) [out=120, in=270] to (-2.8,2);\draw[densely dotted] (-2.5,0) [out=160, in=280] to (-3.2,1.2 );\draw[densely dotted] (-2.5,0) [out=70, in=250] to (-2.1,2.2 );\draw[densely dotted] (-2.5,0) [out=60, in=180] to (2,3.7  );\draw[densely dotted] (-2.5,0) [out=43, in=180] to (2,3.5 );\draw[densely dotted] (-2.5,0) [out=5, in=245] to (-1.1,1 );\draw[densely dotted] (-1.1,1) [out=65, in=180] to (2.01, 3.2);
 \begin{scope}[yscale=-1]  \draw[densely dotted] (-2.5,0) to (-3.5,0);\draw[densely dotted] (-2.5,0) [out=120, in=270] to (-2.8,2);\draw[densely dotted] (-2.5,0) [out=160, in=280] to (-3.2,1.2 );\draw[densely dotted] (-2.5,0) [out=70, in=250] to (-2.1,2.2 );\draw[densely dotted] (-2.5,0) [out=60, in=180] to (2,3.7  );\draw[densely dotted] (-2.5,0) [out=43, in=180] to (2,3.5 );\draw[densely dotted] (-2.5,0) [out=5, in=245] to (-1.1,1 );\draw[densely dotted] (-1.1,1) [out=65, in=180] to (2.01, 3.2);
\end{scope}
\begin{scope}[xscale=-1, xshift=-4cm] \draw[densely dotted] (-2.5,0) to (-3.5,0);\draw[densely dotted] (-2.5,0) [out=120, in=270] to (-2.8,2);\draw[densely dotted] (-2.5,0) [out=160, in=280] to (-3.2,1.2 );\draw[densely dotted] (-2.5,0) [out=70, in=250] to (-2.1,2.2 );\draw[densely dotted] (-2.5,0) [out=60, in=180] to (2,3.7  );\draw[densely dotted] (-2.5,0) [out=43, in=180] to (2,3.5 );\draw[densely dotted] (-2.5,0) [out=5, in=245] to (-1.1,1 );\draw[densely dotted] (-1.1,1) [out=65, in=180] to (2.01, 3.2);
 \begin{scope}[yscale=-1] \draw[densely dotted] (-2.5,0) to (-3.5,0);\draw[densely dotted] (-2.5,0) [out=120, in=270] to (-2.8,2);\draw[densely dotted] (-2.5,0) [out=160, in=280] to (-3.2,1.2 );\draw[densely dotted] (-2.5,0) [out=70, in=250] to (-2.1,2.2 );\draw[densely dotted] (-2.5,0) [out=60, in=180] to (2,3.7  );\draw[densely dotted] (-2.5,0) [out=43, in=180] to (2,3.5 );\draw[densely dotted] (-2.5,0) [out=5, in=245] to (-1.1,1 );\draw[densely dotted] (-1.1,1) [out=65, in=180] to (2.01, 3.2);
\end{scope}\end{scope}
\draw[densely dotted] (1.7, 1.5) to (2.3, 1.5);\draw[densely dotted] (1.7, 1.5) [out=50, in=130] to (2.3, 1.5); \draw[densely dotted] (1.7, 1.5) [out=-50, in=-130] to (2.3, 1.5); \draw[densely dotted] (1.7, 1.5) [out=150, in=270] to (1.6, 1.6);\draw[densely dotted] (1.6, 1.6) [out=90, in=180] to (2, 1.8);\draw[densely dotted] (1.7, 1.5) [out=160, in=270] to (1.4, 1.6);\draw[densely dotted] (1.4, 1.6) [out=90, in=180] to (2, 2);\draw[densely dotted] (1.7, 1.5) [out=160, in=270] to (1.2, 1.65);\draw[densely dotted] (1.2, 1.65) [out=90, in=180] to (2, 2.2);
\begin{scope}[xscale=-1, xshift=-4cm]\draw[densely dotted] (1.7, 1.5) [out=150, in=270] to (1.6, 1.6);\draw[densely dotted] (1.6, 1.6) [out=90, in=180] to (2, 1.8);\draw[densely dotted] (1.7, 1.5) [out=160, in=270] to (1.4, 1.6);\draw[densely dotted] (1.4, 1.6) [out=90, in=180] to (2, 2);\draw[densely dotted] (1.7, 1.5) [out=160, in=270] to (1.2, 1.65);\draw[densely dotted] (1.2, 1.65) [out=90, in=180] to (2, 2.2);\end{scope}
\begin{scope}[yscale=-1, yshift=-3cm]\draw[densely dotted] (1.7, 1.5) [out=150, in=270] to (1.6, 1.6);\draw[densely dotted] (1.6, 1.6) [out=90, in=180] to (2, 1.8);\draw[densely dotted] (1.7, 1.5) [out=160, in=270] to (1.4, 1.6);\draw[densely dotted] (1.4, 1.6) [out=90, in=180] to (2, 2);\draw[densely dotted] (1.7, 1.5) [out=160, in=270] to (1.2, 1.65);\draw[densely dotted] (1.2, 1.65) [out=90, in=180] to (2, 2.2);
\begin{scope}[xscale=-1, xshift=-4cm]\draw[densely dotted] (1.7, 1.5) [out=150, in=270] to (1.6, 1.6);\draw[densely dotted] (1.6, 1.6) [out=90, in=180] to (2, 1.8);\draw[densely dotted] (1.7, 1.5) [out=160, in=270] to (1.4, 1.6);\draw[densely dotted] (1.4, 1.6) [out=90, in=180] to (2, 2);\draw[densely dotted] (1.7, 1.5) [out=160, in=270] to (1.2, 1.65);\draw[densely dotted] (1.2, 1.65) [out=90, in=180] to (2, 2.2);\end{scope}\end{scope}
  \begin{scope}[yscale=-1] \draw[densely dotted] (1.7, 1.5) to (2.3, 1.5);\draw[densely dotted] (1.7, 1.5) [out=50, in=130] to (2.3, 1.5);\draw[densely dotted] (1.7, 1.5) [out=-50, in=-130] to (2.3, 1.5); \draw[densely dotted] (1.7, 1.5) [out=150, in=270] to (1.6, 1.6);\draw[densely dotted] (1.6, 1.6) [out=90, in=180] to (2, 1.8);\draw[densely dotted] (1.7, 1.5) [out=160, in=270] to (1.4, 1.6);\draw[densely dotted] (1.4, 1.6) [out=90, in=180] to (2, 2);\draw[densely dotted] (1.7, 1.5) [out=160, in=270] to (1.2, 1.65);\draw[densely dotted] (1.2, 1.65) [out=90, in=180] to (2, 2.2);
\begin{scope}[xscale=-1, xshift=-4cm]\draw[densely dotted] (1.7, 1.5) [out=150, in=270] to (1.6, 1.6);\draw[densely dotted] (1.6, 1.6) [out=90, in=180] to (2, 1.8);\draw[densely dotted] (1.7, 1.5) [out=160, in=270] to (1.4, 1.6);\draw[densely dotted] (1.4, 1.6) [out=90, in=180] to (2, 2);\draw[densely dotted] (1.7, 1.5) [out=160, in=270] to (1.2, 1.65);\draw[densely dotted] (1.2, 1.65) [out=90, in=180] to (2, 2.2);\end{scope}
\begin{scope}[yscale=-1, yshift=-3cm]\draw[densely dotted] (1.7, 1.5) [out=150, in=270] to (1.6, 1.6);\draw[densely dotted] (1.6, 1.6) [out=90, in=180] to (2, 1.8);\draw[densely dotted] (1.7, 1.5) [out=160, in=270] to (1.4, 1.6);\draw[densely dotted] (1.4, 1.6) [out=90, in=180] to (2, 2);\draw[densely dotted] (1.7, 1.5) [out=160, in=270] to (1.2, 1.65);\draw[densely dotted] (1.2, 1.65) [out=90, in=180] to (2, 2.2);\begin{scope}[xscale=-1, xshift=-4cm]\draw[densely dotted] (1.7, 1.5) [out=150, in=270] to (1.6, 1.6);\draw[densely dotted] (1.6, 1.6) [out=90, in=180] to (2, 1.8);\draw[densely dotted] (1.7, 1.5) [out=160, in=270] to (1.4, 1.6);\draw[densely dotted] (1.4, 1.6) [out=90, in=180] to (2, 2);\draw[densely dotted] (1.7, 1.5) [out=160, in=270] to (1.2, 1.65);\draw[densely dotted] (1.2, 1.65) [out=90, in=180] to (2, 2.2);\end{scope}\end{scope}\end{scope}
\draw[densely dotted] (0.5,0) [out=  100 , in= 240  ] to ( 0.8 , 1.5); \draw[densely dotted] (0.8, 1.5) [out=70, in=180]to (2, 2.6);\draw[densely dotted] (0.5,0) [out=100, in= 260] to (0.1, 1);\draw[densely dotted] (0.1,1) [out=70, in=180] to  (2,2.75);\draw[densely dotted] (0.5,0) [out=170, in=260] to (-0.5, 0.8);\draw[densely dotted] (-0.5,0.8) [out=80, in=230] to  (0.45, 2.2);  \draw[densely dotted] (0.45, 2.2)[out=50, in=180] to(2,2.9);
\begin{scope}[xscale=-1, xshift=-4cm]\draw[densely dotted] (0.5,0) [out=  100 , in= 240  ] to ( 0.8 , 1.5); \draw[densely dotted] (0.8, 1.5) [out=70, in=180]to (2, 2.6);\draw[densely dotted] (0.5,0) [out=100, in= 260] to (0.1, 1);\draw[densely dotted] (0.1,1) [out=70, in=180] to  (2,2.75);\draw[densely dotted] (0.5,0) [out=170, in=260] to (-0.5, 0.8);\draw[densely dotted] (-0.5,0.8) [out=80, in=230] to  (0.45, 2.2);  \draw[densely dotted] (0.45, 2.2)[out=50, in=180] to(2,2.9);\end{scope}
\begin{scope}[yscale=-1]\draw[densely dotted] (0.5,0) [out=  100 , in= 240  ] to ( 0.8 , 1.5); \draw[densely dotted] (0.8, 1.5) [out=70, in=180]to (2, 2.6);\draw[densely dotted] (0.5,0) [out=100, in= 260] to (0.1, 1);\draw[densely dotted] (0.1,1) [out=70, in=180] to  (2,2.75);\draw[densely dotted] (0.5,0) [out=170, in=260] to (-0.5, 0.8);\draw[densely dotted] (-0.5,0.8) [out=80, in=230] to  (0.45, 2.2);  \draw[densely dotted] (0.45, 2.2)[out=50, in=180] to(2,2.9);
\begin{scope}[xscale=-1, xshift=-4cm]\draw[densely dotted] (0.5,0) [out=  100 , in= 240  ] to ( 0.8 , 1.5); \draw[densely dotted] (0.8, 1.5) [out=70, in=180]to (2, 2.6);\draw[densely dotted] (0.5,0) [out=100, in= 260] to (0.1, 1);\draw[densely dotted] (0.1,1) [out=70, in=180] to  (2,2.75);\draw[densely dotted] (0.5,0) [out=170, in=260] to (-0.5, 0.8);\draw[densely dotted] (-0.5,0.8) [out=80, in=230] to  (0.45, 2.2);  \draw[densely dotted] (0.45, 2.2)[out=50, in=180] to(2,2.9);\end{scope}\end{scope}
\draw[densely dotted] (0.5,0) to (3.5,0);\draw[densely dotted] (0.5,0) [out=60, in= 120 ] to (1, 1);\draw[densely dotted] (1,1) [out=300, in= 180 ] to  (2,0.4);\draw[densely dotted] (0.5,0) [out= 40   , in= 120  ] to (1,0.6);\draw[densely dotted] (1,0.6) [out= 300   , in= 180  ] to (2,0.3);
\draw[densely dotted] (0.5,0)  [out= 20   , in= 150  ] to (1,0.3); \draw[densely dotted] (1,0.3) [out= 330   , in= 180  ] to (2,0.15);
\begin{scope}[xscale=-1, xshift=-4cm]\draw[densely dotted] (0.5,0) [out=60, in= 120 ] to (1, 1);\draw[densely dotted] (1,1) [out=300, in= 180 ] to  (2,0.4);\draw[densely dotted] (0.5,0) [out= 40   , in= 120  ] to (1,0.6);\draw[densely dotted] (1,0.6) [out= 300   , in= 180  ] to (2,0.3);\draw[densely dotted] (0.5,0)  [out= 20   , in= 150  ] to (1,0.3); \draw[densely dotted] (1,0.3) [out= 330   , in= 180  ] to (2,0.15);\end{scope}
\begin{scope}[yscale=-1]\draw[densely dotted] (0.5,0) [out=60, in= 120 ] to (1, 1);\draw[densely dotted] (1,1) [out=300, in= 180 ] to  (2,0.4);\draw[densely dotted] (0.5,0) [out= 40   , in= 120  ] to (1,0.6);\draw[densely dotted] (1,0.6) [out= 300   , in= 180  ] to (2,0.3);\draw[densely dotted] (0.5,0)  [out= 20   , in= 150  ] to (1,0.3); \draw[densely dotted] (1,0.3) [out= 330   , in= 180  ] to (2,0.15);
\begin{scope}[xscale=-1, xshift=-4cm]\draw[densely dotted] (0.5,0) [out=60, in= 120 ] to (1, 1);\draw[densely dotted] (1,1) [out=300, in= 180 ] to  (2,0.4);\draw[densely dotted] (0.5,0) [out= 40   , in= 120  ] to (1,0.6);\draw[densely dotted] (1,0.6) [out= 300   , in= 180  ] to (2,0.3);\draw[densely dotted] (0.5,0)  [out= 20   , in= 150  ] to (1,0.3); \draw[densely dotted] (1,0.3) [out= 330   , in= 180  ] to (2,0.15);\end{scope}\end{scope}
\draw[thick, gray] (-1,0) to (-0.3,0.7);  \draw[thick, gray] (-1,0) to (-1.7,0.7); \draw[thick, gray] (-1,0) to (-0.3,-0.7);\draw[thick, gray] (-1,0) to (-1.7,-0.7);
\draw[thick, gray, -<] (-1,0) to (-0.65,0.35);    \draw[thick, gray, ->] (-1,0) to (-0.6,-0.4);  \draw[thick, gray, ->] (-1,0) to (-1.4,0.4);   \draw[thick, gray, -<] (-1,0) to (-1.35,-0.35);  
\node[red] at (5,0.2) [right] {\small$P_{2,1}$};  \node[blue] at (6.5,0.2) [right] {\small$P_{3,1}$};   \node[blue] at (3.5,0.2) [right] {\small$P_{3,2}$} ;    \node[red] at (2.95,1.7) [right] {\small$P_{2,2}$};    ;\node[red] at (3,-1.6) [right] {\small$P_{2,3}$};   \node[blue] at (2.25,1.7) [right] {\small$P_{3,3}$} ;      \node[blue] at (2.25, -1.75) [right] {\small$P_{3,4}$} ;     \node at (2, 3.5) {\small{$\mathcal{F}_{1,1}$}};    \node at (2, 1.9) {\small{$\mathcal{F}_{3,1}$}};      \node at (2, -1.2) {\small{$\mathcal{F}_{4,1}$}};  \node at (1.65, 0) {\small{$\mathcal{F}_{2,1}$}};  \node at (4.1, -1  ) {\small{$\mathcal{F}_{2,3}$}};  \node at (4.1,  1  ) {\small{$\mathcal{F}_{2,2}$}};      
\node[gray] at (-0.35, 0.7) [right] {\small{$B^s_{2,2}$}};    \node[gray] at (-0.35, -0.7) [right] {\small{$B^u_{2,3}$}};  \node[gray] at (-1.6, -0.8) [left] {\small{$B^s_{1,1}$}};    \node[gray] at (-1.6, 0.7) [left] {\small{$B^u_{1,1}$}}; 
\draw[thick, purple, ->] (-0.7, 0.5) [out=210, in=330] to (-1.3, 0.5);    \draw[thick, purple, <-] (-0.7, -0.5) [out=150, in=30] to (-1.3, -0.5);  \draw[thick, purple, ->]  (-0.5, 0.3) [out=240, in=120] to (-0.5, -0.3);    \draw[thick, purple, ->]  (-1.5, -0.3) [out=60, in=300] to (-1.5, 0.3);
\node[purple] at (-0.55, 0.5) [above] {\footnotesize{$l^{\mathrm{int}}_{2,2}$}};    \node[purple] at (-1.25, -0.45) [below] {\footnotesize{$l^{\mathrm{int}}_{1,1}$}};    \node[purple] at (-0.5, 0.3) [right] {\footnotesize{$l^{\mathrm{ext}}_{2,2}$}};    \node[purple] at (-1.4, -0.3) [left] {\footnotesize{$l^{\mathrm{ext}}_{1,1}$}};
\draw[thick, purple, ->] (1.4, -0.6) [out=140, in=270] to ( 1.15,  0.1);  
\draw[thick, purple ] (1.15,  0) [out=90, in=220] to ( 1.4, 0.6);  
\node[purple] at (1.21,0.15) [left] {\footnotesize{$g_{2,1} $}};

\draw[<-] (5.7,0) [out=270, in=160] to (6, -1);  \node at (6, -1) [right] {\footnotesize{$V^-_{1,1}=V^+_{1,1}$}};    \draw[<-] (4.7,-1.3) [out=340, in=90] to (5.2, -2.2);  \node at (5.2, -2.2) [below] {\footnotesize{$U^+_{1,1}=U^-_{2,3-}$}};    \begin{scope}[yscale=-1] \draw[<-] (4.7,-1.3) [out=340, in=90] to (5.2, -2.2); \end{scope}     \node at (5.2, 2.2) [above] {\footnotesize{$U^- _{1,1}=U^+_{2,2}$}};    \draw[<-] (2.5,-2.35) [out=320, in=120] to (4, -3.3);  \node at (4, -3.3) [right] {\footnotesize{$U^+_{2,3}=U^-_{4,1 }$}};    \begin{scope}[yscale=-1] \draw[<-] (2.5,-2.35) [out=320, in=120] to (4, -3.3); \end{scope}  \node at (4, 3.3) [right] {\footnotesize{$U^-_{2,2}=U^+_{3,1}$}};     \draw[<-] (2.85 ,-1) [out=50, in=310] to (2.75, -0.4);  \node at (2.75, -0.55) [above] {\footnotesize{$U^-_{2,1 }=U^+_{4,1 }$}};    \begin{scope}[yscale=-1]  \draw[<-] (2.85 ,-1) [out=50, in=310] to (2.75, -0.4);  \end{scope}  \node at (2.75, 0.55) [below] {\footnotesize{$U^+_{2,1 }=U^-_{3,1}$}};    \draw[<-] (0.7 ,-1) [out=210, in=20] to (-0.7 , -2.2   );  \node at (-0.7, -2.2) [left] {\footnotesize{$V^-_{2,1}=V^+_{2,3}$}};    \begin{scope}[yscale=-1] \draw[<-] (0.7 ,-1) [out=210, in=20] to (-0.7 , -2.2   ); \end{scope} \node at (-0.7, 2.2) [left] {\footnotesize{$V^+_{2,1}=V^-_{2,2}$}};    \draw[<-] (4.5 ,0) [out=90, in=180] to (5.5 , 1  );  \node at (5.5,1 ) [right] {\footnotesize{$V^+_{2,2}=V^-_{2,3}$}};    \draw[<-] (1.35 ,-1.5) [out=270, in=20] to (0.3 , -3  );  \node at (0.3,-3 ) [left] {\footnotesize{$V^-_{4,1}=V^+_{4,1}$}};    \begin{scope}[yscale=-1]\draw[<-] (1.35 ,-1.5) [out=270, in=20] to (0.3 , -3  ); \end{scope} \node at (0.3,3 ) [left] {\footnotesize{$V^-_{3,1}=V^+_{3,1}$}};
 \draw [fill, red] (-1,0) circle [radius=0.07]; \draw [fill, blue] (0.5,0) circle [radius=0.07]; \draw [fill, blue] (3.5,0) circle [radius=0.07];  \draw [fill, red] (5,0) circle [radius=0.07];  \draw [fill, red] (1 ,1.5) circle [radius=0.07]; \draw [fill, red] (3,1.5) circle [radius=0.07]; \draw [fill, red] (1,-1.5) circle [radius=0.07]; \draw [fill, red] (3,-1.5) circle [radius=0.07]; \draw [fill, blue] (6.5,0) circle [radius=0.07]; \draw [fill, blue] (-2.5,0) circle [radius=0.07]; \draw [fill, blue] (1.7,1.5) circle [radius=0.07]; \draw [fill, blue] (2.3,1.5) circle [radius=0.07]; \draw [fill, blue] (1.7,-1.5) circle [radius=0.07]; \draw [fill, blue] (2.3,-1.5) circle [radius=0.07];
  \end{tikzpicture}
 \end{center}
\caption{This  example  illustrates  a section of  a weakly convex foliation  on $S^3$ with $\ell=3$.   The blue  and red dots represent the binding orbits with Conley-Zehnder indices $3$ and $2$, respectively. The rigid planes and rigid cylinders are represented by bold curves, and  the families of planes are represented by dotted curves. 
}
\label{fig:examplenotation}
 \end{figure}
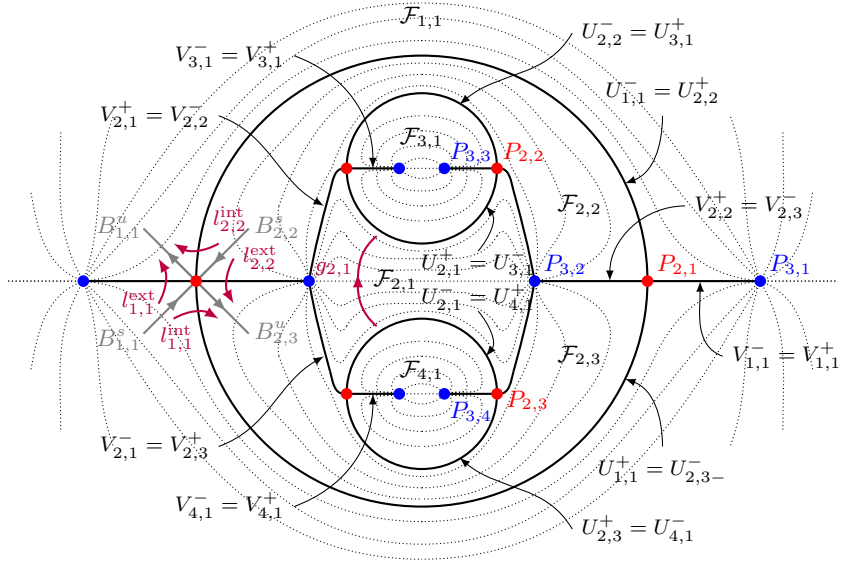

For  each $(j,k) \in \mathcal{C}$, we have  the following transition maps that  preserve the area form induced by $d\lambda$. 

 \begin{itemize}
 \item The  {\bf global transition map}
 $
 g_{j,k} : \F^-_{j,k} \to \F^+_{j,k}
 $
 is defined as   the first intersection point of the forward trajectory  with the plane $\F_{j,k}^{+}.$ This map is well-defined since $P_{3,j}$ has index 3. 

  \item The {\bf   local exterior transition map} 
$
l^{\rm ext}_{j,k}\colon   \F^+_{j,k} \setminus \mathcal{D}^s_{j,k} \to \F^{-}_{j,k+1} \setminus \mathcal{D}^u_{j,k+1}
$
    is defined as the first intersection point of the forward trajectory  with $\F_{j,k+1}^{-}$. Any such a trajectory crosses $V_{j,k}^{+}$ before hitting $\F_{j,k+1}^{-}$. 

\item  The {\bf  local interior transition map} 
$
l^{\rm int}_{j,k}\colon   \mathcal{D}^s_{j,k}\setminus C^s_{j,k} \to \mathcal{D}^u_{j',k'} \setminus C^u_{j',k'}
$
 is defined as the first intersection point of the forward trajectory  with $\F_{j',k'}^{-}$, where $(j', k')$ is such that $U_{j,k}^+ = U_{j',k'}^-$. Any such a trajectory crosses  $U_{j,k}^{+}$ before hitting $\F_{j',k'}^{-}$. 

   \end{itemize}
   
\subsection{Real-analytic models}\label{sec:coordinates} 
 
  Let $P_{2,i}=(x_{2,i},T_{2,i})\in \P_2(\lambda)$. In this section, we show the existence of suitable real-analytic coordinates near $P_{2,i}$ that will be used to model neighborhoods of  $C^s_{j,k} \subset \F^+_{j,k}$ and $C^u_{j,k} \subset \F^-_{j,k}$. 
  
  
  \begin{prop}\label{prop:coordiantes}
  There exist real-analytic coordinates $(t,x,y)\in (\R / T_{2,i} \Z) \times B_{\delta'}(0),$ $\delta'>0$ small, on a small tubular neighborhood $\U_{\delta'}\subset S^3$ of $P_{2,i}$, so that 
     $P_{2,i} \equiv (\R / T_{2,i} \Z )\times \{0\},$ and,
      up to time reparametrization, the trajectories of the Reeb flow of $\lambda$ in $\U_\delta$ coincide with the trajectories of
      \begin{equation}\label{eq_motion}
      \dot t  =1, \quad  \dot x  = -u(xy)x, \quad  \dot y  =u(xy) y,
    \end{equation}
    where   $u(w)=\ln \eta + \eta_1 w + \eta_2 w^2 +\cdots $ ($\eta>1$) is a convergent power series  near $w=0$. In particular, the quantity $xy$ is preserved by the flow.
  \end{prop}
  
  \begin{proof}

 The proof is a direct application of a result due to Moser in \cite{Moser56} which asserts that a real-analytic mapping $\phi$ defined near a hyperbolic fixed point at $0\in \R^2$ admits coordinates $(x,y)$ so that it has the form $\phi(x,y)= \left( x e^{ -u(xy)}, y e^{ u(xy)} \right),$ 
where $u(w)=\ln \eta + \eta_1 w + \eta_2 w^2 +\cdots $ is a convergent power series  near $w=0$. Such coordinates on a cross section of $P_{2,i}$ induce the desired coordinates on a tubular neighborhood of $P_{2,i}$. \end{proof}

 In the coordinates given in Proposition \ref{prop:coordiantes}, the local stable and unstable manifolds of $P_{2,i}$ are
 $
 W^s_{\rm loc}(P_{2,i}) \subset ( \R / T_{2,i} \Z)  \times \R \times \{0\}$ and 
 $W^u_{\rm loc}(P_{2,i})\subset ( \R / T_{2,i} \Z)  \times \{0\} \times \R,
 $
 respectively.

  \subsubsection{The local exterior transition maps}\label{sec:localtranext}

    Fix $(j,k) \in \mathcal{C}$ and set $k'=k+1$.
   Fix also $\tau_-$ and $\tau_+$  close enough to $0$ and $1$, respectively, so that the circles $C^s_{j,k}$ and $C^u_{j, k'}$,  see \eqref{eq:CC}, are contained in the tubular neighborhood $\mathcal{U}_{\delta'}$ of $\P_{j,k}^{+}=\P_{j,k'}^{-}$ as in Proposition \ref{prop:coordiantes}. 
        Choose  sufficiently small annular neighborhoods   $R_{j,k}^s \subset \mathcal{F}_{j,k}^{+} $   of $C^s_{j,k} $  and   $R_{j,k'}^u \subset \mathcal{F}_{j,k'}^{-} $   of $C^u_{j,k'} $  that are modelled in the real-analytic coordinates $(t,x,y)$ near $P_{2,i}$ by 
    \begin{align*}
    R^s   &:= \{ ( t, x, y) \mid x = \tfrac{\delta}{2} , \;  y \in (-\tfrac{\delta }{4}  , \tfrac{ \delta }{4}) \},\\
        R^u   &:= \{ ( t, x, y) \mid    x \in (-\tfrac{\delta }{4}  , \tfrac{ \delta }{4}), \; y = \tfrac{\delta}{2} \},
    \end{align*}
    respectively, where $0<\delta \ll \delta'$.  Set  $ A_{j,k}^{\mathrm{ext},s}:= R_{j,k}^s \setminus \mathcal{D}^s_{j,k}$   and $ A_{j,k'}^{\mathrm{ext},u}:= R_{j,k'}^u \setminus \mathcal{D}^u_{j,k'}$. These annuli are modelled in our real-analytic coordinates by 
    \begin{align}\label{eq:extA}
\begin{split}
    A ^{\mathrm{ext},s} : = \{ ( t, x, y) \mid x = \tfrac{\delta}{2} , \; y \in (0 , \tfrac{ \delta }{4}) \}, \\
    A ^{\mathrm{ext},u}  := \{ ( t, x, y) \mid x \in (0 , \tfrac{ \delta }{4}), \; y= \tfrac{\delta}{2}   \},
\end{split}
\end{align}
respectively. 
     Note that  $A_{j,k}^{\mathrm{ext},s}$ is mapped under $l^{\mathrm{ext}}_{j,k}$ onto $A^{\mathrm{ext},u}_{j,k'}. $

 Consider the real-analytic maps 
    \begin{align}\label{eq:extrealanalyticcoordinates}
    \begin{split}
 F_{j,k}^{\mathrm{ext},s} \colon (\R /T_{2,i} \Z) \times (0, \tfrac{\delta}{4}) \to   A_{j,k}^{\mathrm{ext},s }, \\
 F_{j,k'}^{\mathrm{ext},u} \colon (\R /T_{2,i} \Z) \times (0, \tfrac{\delta}{4}) \to   A_{j,k'}^{\mathrm{ext},u }, 
 \end{split}
 \end{align}
 given in  our coordinates by   $(t, y) \mapsto (t, \tfrac{\delta}{2}, y)$ and  $(t, x) \mapsto (t, x, \tfrac{\delta}{2}),$ respectively.   
In these coordinates, the local exterior transition map   $l^{\mathrm{ext}}_{j,k} \colon    A_{j,k}^{\mathrm{ext},s } \to  A_{j,k'}^{\mathrm{ext},u }$ admits a lift
\begin{equation*}\label{eq:ellext}
\tilde l :\R \times (0, \tfrac{\delta}{4}) \to \R \times (0, \tfrac{\delta}{4}), \quad (t, r) \mapsto (t + \Delta t(r), r),
\end{equation*}
where   
\begin{equation}\label{eq:timevar}
\Delta t(r)  = T_{2,i} \frac{1}{u(  \delta r / 2 )} \ln \frac{ \delta}{2r} = g(r) - h(r)\ln r,
\end{equation}
for real-analytical functions $g(r),h(r)$ defined near $r=0$, with $h(0)>0$. Notice that   $\Delta t(r) \to +\infty$ as $r \to 0^+$. 

\begin{lem}
 \label{lem:infinitetwists} 
 Let $\gamma : [0, 1) \to \R \times [0, \tfrac{\delta}{4})$  be a real-analytic curve  such that $\gamma(s)  \in \R \times \{ 0 \}$ only at $s=0$. Then  $\tilde l \circ \gamma$ is a monotone curve in the $\R$-direction for $s>0$ sufficiently small. Moreover, writing $\tilde l \circ \gamma$ as a graph    $r=\eta(t), t\gg 0,$ we have $\frac{d\eta}{dt} \to 0$ as  $t\to +\infty$
 \end{lem}
 \begin{proof} 
Write  $\gamma(s) = (t(s), r(s))$ such that $t(0)=t_0$ and $r(0)=0$. Since $r(s)$ is real-analytic, we find  $a>0$ and $n \in \N$ such that 
     \begin{equation}\label{eq:rs}
  r(s) = a s^n + O(s^{n+1})
  \end{equation}
    and thus $r(s)$ is strictly increasing for $s\geq 0$ sufficiently small. Moreover, there exists $b,B>0$ such that
   \begin{equation}\label{eq:estima1}
  \frac{B}{s}> \frac{r'(s)}{r(s)}>\frac{b}{s}, \quad \forall s>0 \mbox{ small}.
   \end{equation}
This implies that there exists $c_1>0$ so that
\begin{equation}\label{eq:derivats}    
  \begin{aligned} 
  \frac{d}{ds} \left(t(s) + \Delta t ( r(s)) \right) &= t'(s) +\left(g'(r(s))-h'(r(s))\ln r(s) - \frac{h(r(s))}{r(s)} \right)r'(s)\\
 & < -\frac{c_1}{s},
  \end{aligned}
  \end{equation}
   for every $s>0$  sufficiently small.
  Hence $t(s)  + \Delta t( r(s))$ increases monotonically to $+\infty$ as $s \to 0^+$.
    From the conclusion above, we can write the curve $\tilde l \circ \gamma$  as a graph  $r=\eta(t), t \gg 0$, where $\eta$ is  real-analytic.  It satisfies $\eta(t(s) + \Delta t(r(s))) = r(s), \forall s>0$ small.   From \eqref{eq:estima1} and \eqref{eq:derivats}, we see that 
  $$
  0  > \eta'(t(s)+\Delta t(r(s))) =\frac{r'(s)}{\frac{d}{ds}(t(s) + \Delta t(r(s)))} >-\frac{B}{c_1}r(s) \to 0
    $$
  as $s \to 0^+$. Since    $t(s)  + \Delta t( r(s)) \to +\infty$ as $s \to 0^+$, this completes the proof. \end{proof}

Now assume that every orbit in $\P_2(\lambda)$ has the same action $T_2>0$, so that the disks $\mathcal{D}_{j,k}^s$ and $\mathcal{D}_{j,k}^u$, see \eqref{eq:DD}, have the  same $d\lambda$-area equal to $T_2$ for all $(j,k) \in \mathcal{C}$.

Using an area-preserving argument, we  check below that for any fixed $P_{2, i} \in \P_2(\lambda)$, the branch of its unstable manifold in $\U_j$ must intersect the branch in $\U_j$ of the stable manifold of an orbit $P_{2,i'} \subset \partial \mathcal{U}_j$. Note that such an intersection  corresponds to a heteroclinic/homoclinic trajectory in $\U_j$ connecting $P_{2, i}$ to $P_{2, i'}$. 

Let 
$
G: \mathcal{C} \to \mathcal{C}
$ be defined so that $(J,K) = G(j,k)$ is the family of planes so that  $\mathcal{D}^s_{J,K}\subset \F^+_{J,K}$ is the first disk  intersected by the forward flow of $\mathcal{D}^u_{j,k}$. 
To be more precise in the definition of $G$, fix $(j,k)\in \mathcal{C}$ and set $D^0 = \mathcal{D}^u_{j,k}$. If
$
g_{j,k}(D^0) \cap C^s_{j,k} \neq \emptyset,
$
then we define $G(j,k)=(j,k).$ Otherwise, since all index 2 orbits have the same action, we have 
$
g_{j,k}(D^0) \subset \F^+_{j,k} \setminus \mathcal{D}^s_{j,k}, 
$
and then define 
$
D^1:= l^{\rm ext}_{j,k}({g_{j,k}(D^0))}\subset \F^{-}_{j,k+1}.
$  
Now we repeat the procedure with $D^1$: if
$
g_{j,k+1}(D^1) \cap C^s_{j,k+1} \neq \emptyset,
$
then we define $G(j,k)=(j ,k+1).$ Otherwise, we have
  $
  g_{j,k+1}(D^1) \subset \F^{+}_{j,k+1} \setminus \mathcal{D}^s_{j, k+1} 
  $
  and then   define
$
D^2:= l^{\rm ext}_{j,k+1}({g_{j,k+1}(D^1))}\subset \F^{-}_{j,k+2}.
$ 
Repeating this process we obtain a sequence  $(j, k^{(n)}) \in \mathcal{C}$, where $k^{(n)} \equiv k+n \; ({\rm mod} \; \tilde k_j),$ and disks 
$
D^n \subset \F^{-}_{j,k^{(n)}},   n \in \N. 
$
By construction and uniqueness of solutions, the disks $D^n$ are mutually disjoint. Since the  $d\lambda$-area of $D^n$ is independent of $n$ (indeed, it is equal to $T_2=1$ for all $n$) and the available area in all $\F^\pm_{j,k}$ is finite (it coincides with the action of $P_{3,j}$), the procedure above has to terminate after finitely many steps and we find $(j ,k^{(N)})$ for some $N\geq 0 $ so that 
$
g_{j,k^{(N)}}(D^N)\cap C^s_{j,k^{(N)}}\neq \emptyset.
$
In this case we set $(J,K)=(j ,k^{(N)})$ and define $G(j,k) = (J,K),$ where $J=j$ since along the process we remain inside the same component $\mathcal{U}_j$ via the global and the  local exterior transition maps.
 
We also consider a map 
\begin{equation}\label{eq_psijk}
\Psi_{j,k}\colon  \mathcal{N}(\mathcal{D}^u_{j,k} ) \to \F^{+}_{J,K},  (j,k) \in \mathcal{C},
\end{equation}
where $(J,K) = G(j,k)$ and $\mathcal{N}(\mathcal{D}^u_{j,k})$ denotes a small neighborhood of  $\mathcal{D}^u_{j,k}$ in the plane $\F^-_{j,k}.$ The first hit of $\mathcal{N}(\mathcal{D}^u_{j,k} ) $ into $\F^+_{J,K}$ under the forward flow is precisely $\Psi_{j,k}$.

\begin{dfn}
We say that $(j,k)\in \mathcal{C}$ is {\bf coincident} if 
$
\Psi_{j,k}(\mathcal{D}^u_{j,k}) = \mathcal{D}^s_{J,K},
$
where $(J,K)= G(j,k).$ Otherwise, we say it is {\bf non-coincident}.
\end{dfn}

   
      Pick  $(j,k) \in \mathcal{C}$ and abbreviate $(J,K)=G(j,k)$. By definition of   $G,$ we have
   $
   \Psi_{j, k }(\mathcal{D}^u_{j, k }) \cap C^s_{J, K } \neq \emptyset.
   $
   Since all index-$2$ orbits have the same action and the flow is real-analytic, we find one of the following   scenarios associated with the pair $(j,k)$: 

 \begin{enumerate}
 
 \item[(a)] $\Psi_{j,  k } (\mathcal{D}^u_{j, k }) =\mathcal{D}^s_{J, K}$, i.e.\ $(j ,k)$ is coincident. 

 \item[(b)]   $\Psi_{j, k} (\mathcal{D}^u_{j,k}) \cap (    \mathcal{D}^s_{J, K}   \setminus C^s_{J, K}  )= \emptyset $, and   $\Psi_{j, k} (\mathcal{D}^u_{j,k})$ intersects the circle $C^s_{J, K}$   at finitely many points.  

 \item[(c)] $\Psi_{j, k}(\mathcal{D}^u_{j,k})$  intersects both   $ \mathcal{F}^{+}_{J,K} \setminus \mathcal{D}^s_{J, K}  $  and $\mathcal{D}^s_{J, K}  \setminus  C^s_{J, K}$. In this case the  disk $\Psi_{j, k} (\mathcal{D}^u_{j,k})$ also intersects the circle $C^s_{J, K}$  at finitely many points. 
 \end{enumerate}

Notice that every intersection in   (b) is tangent and an intersection  in  (c) is not necessarily transverse.

   \begin{figure}[ht]
     \centering
    \begin{subfigure}{1\textwidth}
\centering
\begin{tikzpicture}[scale=0.7 ]

 \begin{scope}[xshift=3cm]  \draw[thick ]  (-7.7,-0.6 ) arc (180:360: 1cm and  1cm);    \draw[thick ]   (-7.7,-0.6 ) arc (180:0: 1cm and 1cm);  
  \node[gray] at ( -6.2, 0.9)  [right] {\footnotesize{$C^s_{J,K}$}}; 
\node  at ( -6, 0.9) [left] {\footnotesize{$\Psi_{j, k} (\mathcal{D}^u_{j, k}) =$}}; 
\end{scope}

 \begin{scope}[xshift=7cm, yshift=-0.2cm]
\draw[dashed, lightgray] (-5, -0.5) to (-5,1.5);    \draw[dashed, lightgray] (-4, -1) to (-4,1 );    \draw[dashed, lightgray] (-5, -0.5) to (-4,-1);     \draw[dashed, lightgray] (-5, 1.5) to (-4,1);
\draw[ gray, thick] (-4.5, 0.6) to (-3.6, 0.6);\draw[ gray, thick] (-4.5, -0.2) to (-3.6, -0.2);
  \draw[  gray, thick] (-4.5 ,-0.2) arc (-90:150: 0.2cm and  0.4cm);       \draw[ gray, thick, densely dotted,] (-4.5 ,-0.2) arc (-90:270: 0.2cm and  0.4cm); 
  \draw[  thick] (-4.5 ,-0.2) arc (-90:90: 0.2 cm and  0.4cm);     \draw[ densely dotted, thick] (-4.5 ,-0.2) arc (-90:-270: 0.2cm and  0.4cm); 
\draw[thick] (-4.5, 0.6) to (-5.4, 0.6);     \draw[thick] (-4.5, -0.2) to (-5.4, -0.2);
\node at (-4.9, -1 ) [below] {\footnotesize{$\Psi_{j, k}(\mathcal{D}^u_{j, k})=$}};
\node[gray] at (-3.2 , -1 ) [below] {\footnotesize{$C^s_{J,K}$}};
\node[lightgray] at (-5.05 ,  1.1)  [right] {\tiny{$\mathcal{F}^{+}_{J,K}$}};
\draw[gray] (-3.9, -1.2 ) [out=  110 , in= 330  ] to (-4.3,  0.2);
 \end{scope}
  \end{tikzpicture}
  \caption{}
   \label{fig:scenarioA}
\end{subfigure}

   \begin{subfigure}{0.48\textwidth}
\centering
\begin{tikzpicture}[scale=0.7 ]
      \draw[white]   (5,3.2 ) circle [radius=0.05];

  \begin{scope}[xshift=-3cm] 
   \draw[thick, lightgray] (2,0 ) arc (180:360: 1cm and  1cm);    \draw[thick, lightgray]  (2,0 ) arc (180:0: 1cm and 1cm);  
       \draw[fill]   ( 2.3,0.7) circle [radius=0.05];        
\draw[thick] (2.3, 0.7) [out=30, in= 180] to (2.9,1.2);
\draw[thick] (2.9, 1.2) [out=0, in=150] to (3.35 , 1.1);
 \draw[thick] (3.35, 1.1) [out=330, in=250] to (4, 1.2);
 \draw[thick] (4, 1.2) [out=70, in=290] to (3, 2 );
 \draw[thick] (3, 2 ) [out=110, in=70] to (2.2, 1.6);
 \draw[thick] (2.2, 1.6) [out=250, in=210] to (2.3, 0.7);
   \node[gray] at ( 3.3,-1) [below] {\footnotesize{$C^s_{J,K}$}}; 
\node  at ( 3.8, 1.1) [below] {\footnotesize{$\Psi_{j, k} (\mathcal{D}^u_{j, k}) $}}; 
 \end{scope}
\begin{scope}[xshift=9cm]
\draw[dashed, lightgray] (-5, -0.5) to (-5,1.5);    \draw[dashed, lightgray] (-4, -1) to (-4,1 );    \draw[dashed, lightgray] (-5, -0.5) to (-4,-1);     \draw[dashed, lightgray] (-5, 1.5) to (-4,1);
\node[lightgray] at (-5.05 ,  1.1)  [right] {\tiny{$\mathcal{F}^{+}_{J,K}$}};
\begin{scope}[yshift=-0.2cm]
\draw[ gray, thick] (-4.5, 0.4) to (-3.6, 0.4);\draw[ gray, thick] (-4.5, -0.2) to (-3.6, -0.2);
  \draw[  gray, thick] (-4.5 ,-0.2) arc (-90:90: 0.15cm and  0.3cm);       \draw[ gray, thick] (-4.5 ,-0.2) arc (-90:270: 0.15cm and  0.3cm); 
\node[gray] at (-3.4, -0.6) [below] {\footnotesize{$C^s_{J,K}$}};
\draw[gray] (-3.5, -0.6 ) [out=  110 , in= 330  ] to (-4.35,  0.2);
\end{scope}
\begin{scope}[yshift=0.4cm]
\draw (-4.6  ,0) [ out= 230, in= 60] to (-5.3  , -0.9);
  \draw[  thick] (-4.5 ,-0.2) arc (-90:90: 0.1cm and  0.3cm);     \draw[ densely dotted, thick] (-4.5 ,-0.2) arc (-90:-270: 0.1cm and  0.3cm); 
\draw[thick] (-4.5, 0.4) to (-5.4, 0.4);     \draw[thick] (-4.5, -0.2) to (-5.4, -0.2);
\node at (-5.2, -0.8) [below] {\footnotesize{$\Psi_{j, k}(\mathcal{D}^u_{j, k})$}};
\end{scope}
 \end{scope}
 \end{tikzpicture}
  \caption{}
   \label{fig:scenario2}
\end{subfigure}
    \begin{subfigure}{0.48\textwidth}
\centering
\begin{tikzpicture}[scale=0.7 ]
      \draw[white]   (7,3 ) circle [radius=0.05];
  \begin{scope}[xshift=1cm] 
   \draw[thick, lightgray] ( 6,0 ) arc (180:360: 1cm and  1cm);    \draw[thick, lightgray]  ( 6,0 ) arc (180:0: 1cm and 1cm);  
       \draw[fill]   ( 6.3,0.7) circle [radius=0.05];         \draw[fill]   (8, 0 ) circle [radius=0.05];
\draw[thick] (6.3, 0.7) [out=300, in=240] to (8,0);
\draw[thick] (8,0) [out=60, in= 240] to (8.2, 1 );
\draw[thick] (8.2, 1) [out=60, in= 330] to (7.5, 1.8);
\draw[thick] (7.5, 1.8) [out=150, in=40] to (6.8, 1.5);
\draw[thick] (6.8, 1.5) [out=220, in=120] to (6.3, 0.7);
   \node[gray] at ( 7,-1) [below] {\footnotesize{$C^s_{J,K}$}}; 
\node  at (5.9, 2.2) [right] {\footnotesize{$\Psi_{j, k} (\mathcal{D}^u_{j, k}) $}}; 
\end{scope}
\begin{scope}[xshift=17cm] 
\draw[dashed, lightgray] (-5, -0.5) to (-5,1.5);    \draw[dashed, lightgray] (-4, -1) to (-4,1 );    \draw[dashed, lightgray] (-5, -0.5) to (-4,-1);     \draw[dashed, lightgray] (-5, 1.5) to (-4,1);
\node[lightgray] at (-5.05 ,  1.1)  [right] {\tiny{$\mathcal{F}^{+}_{J,K}$}};
\begin{scope}[yshift=0.2cm, xshift=0.1cm]
\draw[ gray, thick] (-4.5, 0.4) to (-3.6, 0.4);        \draw[ gray, thick, densely dotted] (-4.5, 0.4) to (-4.3, 0.4);   \draw[ gray, thick] (-4.5, -0.2) to (-3.6, -0.2);
  \draw[ gray, thick] (-4.5 ,-0.2) arc (-94:160: 0.1cm and  0.3cm);     \draw[ gray,densely dotted, thick] (-4.5 ,-0.2) arc (-90:-270: 0.1cm and  0.3cm); 
  \node[gray] at (-3.6, -0.8) [below] {\footnotesize{$C^s_{J,K}$}};
  \draw[gray] (-3.7, -0.8) [out=  110 , in= 330  ] to (-4.4,  0.2);
  \end{scope}
    \draw[  thick] (-4.5 ,-0.2) arc (-90:90: 0.1cm and  0.3cm);     \draw[ densely dotted, thick] (-4.5 ,-0.2) arc (-90:-270: 0.1cm and  0.3cm); 
\draw[thick] (-4.5, 0.4) to (-5.4, 0.4);     \draw[thick] (-4.5, -0.2) to (-5.4, -0.2);
\node at (-5.3, -0.5) [below] {\footnotesize{$\Psi_{j, k}(\mathcal{D}^u_{j, k})$}};
\draw (-4.6  ,0.1) [ out= 230, in= 60] to (-5.3 , -0.65);
 \end{scope}
 \end{tikzpicture}
  \caption{}
   \label{fig:scenario4}
\end{subfigure}
 \caption{Possible scenarios for  $\Psi_{j,k}(\mathcal{D}^u_{j,k}) \cap C^s_{J,K}.$}
 \label{fig:scenarios}
 \end{figure}
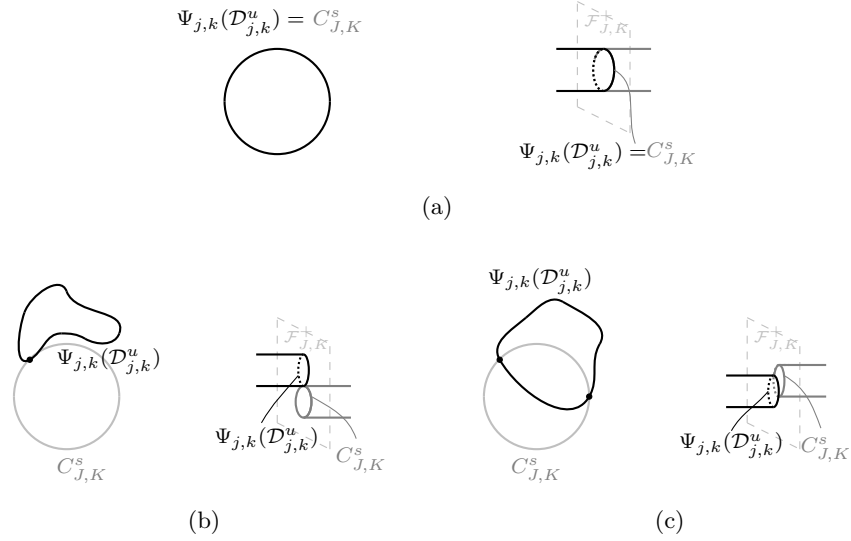

\section{Proof of Theorem \ref{main2}}

\subsection{Proof of Theorem \ref{main2}-(i)}

Fix $j\in \{1,\ldots,l\}$ and assume that for every $k\in \{1,\ldots, \tilde k_j\}$ the branch in $\U_j$ of the unstable manifold of $P_{2,n_k^j}$ coincides with the branch in $\U_j$ of the stable manifold of $P_{2,n_l^j},$ for some $l\in \{1,\ldots,\tilde k_j\}$. This means that all pairs $(j,k)\in \mathcal{C}$ are coincident as in scenario (a).

Fix $(j,k_1)\in \mathcal{C}$ as above,  let $(j,K_1) = G(j,k_1)$, and let $(j,k_2):=(j,K_1+1)$. Our assumptions imply that we can construct an $N$-periodic sequence of distinct $k_1,\ldots,k_N,\ldots,$ so that $G(j,k_i)=(j,K_i),$ and $k_{i+1}=K_i+1, \forall i.$ In particular, the mapping $\Psi_{j,k_i} \colon \mathcal{N}(\mathcal{D}^u_{j,k_i}) \to \mathcal{F}^+_{j,K_i}$ satisfies
$
\Psi_{j,k_i}(\mathcal{D}^u_{j,k_i}) = \mathcal{D}^s_{j,K_i},  \forall i=1,\ldots,N.
$
Moreover, due to the coincidences,
the mapping 
\begin{equation}\label{eq:psi_coincid}\Psi:= \Psi_{j,k_1} \circ l^{\rm ext}_{j,K_N}\circ \ldots   \circ \Psi_{j,k_2} \circ l^{\rm ext}_{j,K_1},
\end{equation}
is well-defined in $\mathcal{V}_1\setminus \mathcal{D}^s_{j,K_1}$, where $\mathcal{V}_1\subset \F^+_{j,K_1}$  is a sufficiently small neighborhood of $\mathcal{D}^s_{j,K_1}$.
For simplicity, assume that $T_{2,i}=1, \forall i.$ Let $\tilde l_i \colon \R \times (0, \delta) \to \R \times (0, \delta)$ be a lift of $l^{\rm ext}_{j,K_i}$ in coordinates $(t,r)$, see Section \ref{sec:localtranext}.

\begin{prop}\label{prop:coincid}
There exist real-analytical functions  $g_i,h_i \colon (-\epsilon,\epsilon) \to \R$,  $h_i(0)>0,$ so that 
$\tilde l_i(t,r) = (t+ g_i(r) - h_i(r) \ln r,r),$ $\forall (t,r).
$

\end{prop}
 
 \begin{proof} 
 Recall that $\tilde l_i(t,r) = (t+\Delta t_i(r),r)$, where $\Delta t_i(r) = \frac{1}{u_i(\tilde \delta r /2)} \ln \frac{\tilde \delta_i}{2r},$  $u_i>0$ is real-analytic and $\tilde \delta_i>0$. This implies the existence of $g_i,h_i$ as in the statement. \end{proof}

Now we describe  the global mappings $\Psi_{j,k_i}, i=1,\ldots, N,$ in coordinates $(t,r)$. Denote by 
$
\tilde \psi_i  \colon  \R \times (-\delta, \delta)  \to \R  \times (-\delta', \delta'),
$
$0<\delta\ll \delta'$, a lift of $\Psi_{j,k_i}$, defined in a small neighborhood of $C^u_{j,k_i}\subset \F^-_{j,k_i}$.  

\begin{prop}{\cite[Lemma 7.1]{dPS2}}\label{prop:coincid_global}
The global mapping $\tilde \psi_i =( X_i, Y_i)$ has the form
$
 X_i(t,r)  =  H_i(t) + r \tilde X_i(t,r)$ and $
Y_i(t,r)  =r \tilde Y_i(t,r),$
where $H_i-{\rm Id} \colon \R \to \R$ and $\tilde X_i,\tilde Y_i \colon \R \times (-\delta,\delta)\to \R$ are real-analytic functions, $1$-periodic in $t$, and satisfy $H_i',\tilde Y_i>0.$ 
\end{prop}

Consider the sequence 
$
(T_i,R_i):=\tilde \psi_{i+1} \circ\tilde l_i \circ \ldots \circ \tilde \psi_2 \circ \tilde l_1 ,
$
 so that  $(T_N,R_N)$ represents the mapping $\Psi$ in \eqref{eq:psi_coincid}.
Then $(T_{i+1},R_{i+1}) = \tilde \psi_{i+2} \circ \tilde l_{i+1}(T_i,R_i),  \forall i,$~implying 
\begin{equation}\label{eq:t_i}
\begin{aligned}
T_{i+1} & = H_{i+2}\left(T_i + g_{i+1}(R_i) - h_{i+1}(R_i) \ln R_i\right)\\ & \quad + R_i\tilde X_{i+2}(T_i + g_{i+1}(R_i) - h_{i+1}(R_i) \ln R_i,R_i),\\
R_{i+1} & = R_i\tilde Y_{i+2}(T_i+g_{i+1}(R_i)-h_{i+1}(R_i)\ln R_i,R_i),
\end{aligned}
\end{equation}
for every $i\geq 1$, see Propositions \ref{prop:coincid} and \ref{prop:coincid_global}.


\begin{prop}\label{prop:coincid_ln} For every $i\geq 1$, there exist $A_i,B_i,C_i>0$ so that
\begin{equation}\label{eq:tjrj}
T_i(t,r) - t > -C_i \ln r  \ \ \ \mbox{ and } \ \ \  A_ir  <R_i(t,r)< B_i r,
\end{equation}
uniformly in $t$, for every $r>0$ sufficiently small. 
\end{prop}
\begin{proof} We prove  by induction. Notice that for $i=1$, we have 
\begin{equation}\label{eq:t1r1}
\begin{aligned}
T_1(t,r)  & = H_2(t+g_1(r) - h_1(r)\ln r) + r\tilde X_2(t+g_1(r)-h_1(r) \ln r,r),\\
& = t+g_1(r) -h_1(r)\ln r + (H_2-{\rm Id})(t+g_1(r) -h_1(r)\ln r) \\
& \ \ \ + r\tilde X_2(t+g_1(r)-h_1(r) \ln r,r),\\
R_1(t,r) & = r \tilde Y_2(t+g_1(r) - h_1(r)\ln r,r).
\end{aligned}
\end{equation}
The claim follows for $i=1$ since $H_2-{\rm Id}$, $\tilde X_2$ and $\tilde Y_2$ are 1-periodic in $t$, and $\tilde Y_2(\cdot ,0)$ and $h_1(0)$ are positive.

Next assume \eqref{eq:tjrj} holds for $i$. We shall prove that \eqref{eq:tjrj} holds for $i+1$. Using that $H_{i+2}={\rm Id} + \tilde H_{i+2}$, for some $1$-periodic function $\tilde H_{i+2}$, we obtain  from \eqref{eq:t_i} 
$$
\begin{aligned}
T_{i+1} & = T_i + g_{i+1}(R_i) - h_{i+1}(R_i) \ln R_i + \tilde H_{i+2}\left(T_i + g_{i+1}(R_i) - h_{i+1}(R_i) \ln R_i\right)\\ & \quad  + R_i\tilde X_{i+2}(T_i + g_{i+1}(R_i) - h_{i+1}(R_i) \ln R_i,R_i),\\
R_{i+1} & = R_i\tilde Y_{i+2}(T_i+g_{i+1}-h_{i+1})\ln R_i,R_i),
\end{aligned}
$$
where $g_{i+1}$ and $h_{i+1}$ are real-analytic near $0$. Since $\tilde Y_{i+2},\tilde X_{i+2}$ are $1$-periodic in $t,$ and  $\tilde Y_{i+2}(\cdot,0), h_{i+1}(0)>0$, it follows from the induction hypothesis that \eqref{eq:tjrj}  holds for $i+1$.
\end{proof}

Fix any component $\U_j$ of      $S^3 \setminus \bigcup_{i=1}^{\ell} \S_i$ so that $(j,k)\in \mathcal{C}$ is coincident for every $1\leq k \leq \tilde k_j$.  Recall that $\U_j$ is homeomorphic to a $3$-sphere with $\tilde k_j$ disjoint closed $3$-balls removed. 
Fix a plane $\mathcal{F}^+_{j,k_1}$ of the genus zero transverse foliation in $\U_j$, bounded by $P_{3,j}$.  
The branches in $\U_j$ of the stable/unstable manifolds of the orbits in $\P_2(\lambda)$  transversely intersect $\F^+_{j,k_1}$  at mutually disjoint circles. Indeed, these circles bound closed disks $B_\alpha, \alpha=1,\ldots,\nu$. Since their symplectic areas coincide, $B_\alpha$ have  mutually disjoint interiors. The trajectories through the interior of each $B_\alpha$ eventually leave $\U_j$ to an adjacent component $\U_i,i\neq j$. The points in $\partial B_\alpha$ converge to some $P_{2,i} \subset \partial \U_j$. 

Consider the connected subset of $\F^+_{j,k_1}$ given by
$
\mathcal{A}:= \mathcal{F}^+_{j,k_1}   \setminus \cup_{\alpha=1}^{\nu} B_\alpha . 
$
Since the Reeb flow is transverse to $\mathcal{F}^+_{j,k_1}$, the successive local and global maps  determine a diffeomorphism $\Psi \colon \mathcal{A} \to \mathcal{A}$ that  preserves the finite area form induced by $d\lambda$. 
Abbreviating by $\mathcal{U}_{\mathcal{A}} \subset \U_j$ an invariant open subset, defined by the Reeb trajectories through  $\mathcal{A}$, we conclude that  $\mathcal{A}$ is a global surface of section for the Reeb flow restricted to $\mathcal{U}_{\mathcal{A}}$. The mapping $\Psi$ is the first return map to $\mathcal{A}$, and thus  periodic orbits of $\Psi$ correspond to periodic orbits of the Reeb flow in $\mathcal{U}_{\mathcal{A}}$.

Notice that the outer boundary component of $\mathcal{A},$ i.e. the binding orbit $P_{3,j},$ is preserved by $\Phi$, and hence  points near $P_{3,j}$ are mapped under $\Psi$ to points near $P_{3,j}.$ However, near the inner boundary components,     $\Psi$ behaves as a permutation. More precisely, given $\alpha \in \{1, \ldots, \nu\}$ there exists $\alpha'$ (possibly $\alpha'=\alpha)$ such that $\Psi$ maps points near $\partial B_\alpha$ to points near $\partial B_{\alpha'}.$ 

Now we proceed as in   \cite[Section 7]{dPS2}.
Define the  equivalence relation $\sim$ on $\mathcal{F}^+_{j,k_1} $ such that   
$
x \sim y  \Leftrightarrow x, y \in B_\alpha$  for some    $\alpha.
$  
 Abbreviate by $\Pi \colon \mathcal{F}^+_{j,k_1}  \to  \widetilde{\mathcal{F}}^+_{j,k_1} := \mathcal{F}^+_{j,k_1}/ \sim$ the natural projection. It induces the quotient topology on $\widetilde{\mathcal{F}}^+_{j,k_1}$, which becomes an open disk. For each $  \alpha$,  put $p _\alpha = \Pi (B_\alpha) \in \widetilde{\mathcal{F}}^+_{j,k_1}$. The restriction   $\Pi |_{  \mathcal{A}} \colon \mathcal{A} \to    \widehat{\mathcal{F}} ^+_{j,k_1} := \widetilde{\mathcal{F}}^+_{j,k_1} \setminus \{ p _\alpha \mid   \alpha = 1, \ldots, \nu \}$ is a bijection. We endow  $\widehat{\mathcal{F}}^+_{j,k_1} $ with  a natural smooth structure by declaring $\Pi|_{\mathcal{A}}$ to be a smooth diffeomorphism onto $\widehat{\mathcal{F}} ^+_{j,k_1} $. In this way, we obtain the finite area form $\omega$ on $  \widehat{\mathcal{F}}^+_{j,k_1} $ which is defined as the push-forward of the area form on $\mathcal{A}$ induced by $d\lambda$.  The area form $\omega$ naturally extends to a finite area form on $\widetilde{\mathcal{F}}^+_{j,k_1}$, still denoted by $\omega$. This enables us to define a  homeomorphism $\Phi \colon \widetilde{\mathcal{F}}^+_{j,k_1} \to \widetilde{\mathcal{F}}^+_{j,k_1}$, which preserves $\omega$, permutes the points $p_1, \ldots, p_\nu$, and satisfies $\Phi(z) = \Pi \circ \Psi \circ \Pi^{-1}(z), \forall   z \in  \widehat{\mathcal{F}} ^j_{k_1+}$.  
  In particular, the points $p_1,\ldots,p_\nu$ are periodic. 

 \medskip
 
 {\bf Case 1.} $  \nu \geq 2$. 

In this case, Brouwer's fixed point theorem \cite{Brouwer} gives a fixed point of $\Phi$, say $p \in \widetilde{\mathcal{F}}^+_{j,k_1}$. The restriction $\Phi |_{\widetilde{\mathcal{F}}^+_{j,k_1} \setminus \{ p \} }$ is an area preserving homeomorphism  of the open annulus $\widetilde{\mathcal{F}}^+_{j,k_1}\setminus \{ p \},$ containing at least one periodic point $p_j\neq p$.
   An application of Franks' Theorem \cite{Franks}   implies that $\Phi$ has infinitely many periodic points. It follows that $\Psi$ admits infinitely many periodic points, and hence  the Reeb flow admits infinitely many periodic orbits in $\mathcal{U}_{\mathcal{A}}$.

 \medskip
 
 {\bf Case 2.} $ \nu=1$. 
 
  In this case, there exists only one such disk $B_\alpha\subset \F^+_{j,k_1}$ which coincides with $\mathcal{D}^s_{j, k_1}$. Its boundary is necessarily the intersection of the unstable manifold of $\P^-_{j,k_1}$ with $\F^+_{j,k_1}$, which coincides with the intersection of the stable manifold of $\P^+_{j,k_1}$ with $\F^+_{j,k_1}$. Moreover, $\F^+_{j,k_1}\setminus B_\alpha$ is mapped under $l^{\rm ext}_{j,k_1}$ to $\F^-_{j,k_1+1}\setminus \mathcal{D}^u_{j,k_1+1}$. In particular, since $\nu =1$, the branch of the 
  unstable manifold of $\P^-_{j,k_1+1}$ 
  inside $\U_j$  coincides with the branch of the stable manifold of $\P^+_{j,k_1+1}$. Now we can consider the global map $g_{j,k_1+1}$, which necessarily maps $C^u_{j,k_1+1} \subset \F^-_{j,k_1+1}$ to  $C^s_{j,k_1+1} \subset \F^+_{j,k_1+1}$. Continuing this process, we conclude that the first return map $\Psi\colon \F^+_{j,k_1}\setminus B_\alpha \to \F^+_{j,k_1} \setminus B_\alpha,$ is given by successive compositions of local and global maps
  \begin{equation}\label{eq_composition}
  \Psi= g_{j,k_1} \circ l^{\rm ext}_{j ,k_1+\tilde k_j-1} \circ \cdots g_{j,k_1+2} \circ l^{\rm ext}_{j ,k_1+1} \circ g_{j,k_1+1} \circ l^{ \rm ext}_{j,k_1},
  \end{equation}
  where $\tilde k_j$ is the number of boundary components of $\U_j$.
  
Recall that, in the special coordinates $(t,r) \in \R \times (0,\epsilon_i)$, $\epsilon_i>0$ small, defined near  $C^s_{j,k_i}\subset \F^+_{j,k_i}$, the local mapping $l^{\rm ext}_{j,k_i}$ has a lift of the form
$\tilde l_i(t,r)= \left( t + g_i(r)-h_i(r)\ln r,r\right),$  
 for real-analytic functions $g_i(r),h_i(r)$ defined near $r=0$ with $h_i(0)>0$, see Proposition \ref{prop:coincid}.

The global mapping $g_{j,k_i}$ has a lift of the form
$
\tilde \psi(t,r) = (t+H_i(t)+r\tilde X_i(t,r),$ $r\tilde Y_i(t,r)),
$
where $\tilde H_i,\tilde X_i, \tilde Y_i$ are real-analytic and $1$-periodic in $t$   and satisfy  $\tilde H_i'>1$ and $\tilde Y_i>0$, see Proposition \ref{prop:coincid_global}.
 A lift $\tilde \Psi$ of $\Psi$ in coordinates $(t,r)\in \R \times (0,1)$ is then given by the composition
 $
 \tilde \Psi = \tilde \psi_{1} \circ \tilde l_{\tilde k_j} \circ \ldots \circ \tilde \psi_2 \circ \tilde l_1,
 $ 
 for every $r>0$ sufficiently small.
 By Proposition \ref{prop:coincid_ln}, $\tilde \Psi$ has the form
 $
 \tilde \Psi(t,r) = (T_{\tilde k_j}(t,r),R_{\tilde k_j}(t,r)),
 $
 where 
 $
 T_{\tilde k_j}(t,r)-t > -C \ln r$ and $Ar<R_{\tilde k_j}(t,r) < Br
  $
 for every $(t,r)\in \R \times (0,1)$ with $r>0$ sufficiently small. Here,  $0<A<B$ and $0<C$ are positive constants.  
 In particular, $\Psi$ has infinite twist near $\R / \Z \times \{0\}$ which implies that $\Psi$ has infinitely many fixed points.  Let  $\tilde{\Psi}_k (x,y):= \tilde{\Psi}(x,y)-(k  , 0), \forall k \in \N^*$. Proposition 7.2 from \cite{dPS2} says that $\tilde{\Psi}_k$ has a fixed point for every $k$ large. The proof is based on Franks' generalization of the Poincar\'e-Birkhoff theorem \cite{Franks88}.   Moreover,  fixed points of $\tilde{\Psi}_{k_1}$ and $\tilde{\Psi}_{k_2}$, with $k_1 \neq k_2$, correspond to distinct fixed points of $\Psi$. 
Therefore,  the existence of infinitely many periodic orbits in $\mathcal{U}_{\mathcal{A}}$ follows, and this finishes the proof of   Theorem \ref{main2}-(i).

  \subsection{Proof of Theorem \ref{main2}-(ii)}\label{sec_tmain2-(ii)}

Fix $j\in \{1,\ldots,l\}$, and  assume that $\Psi_{j,k}(C^u_{j,k}) \neq C^s_{j,K},  \forall k \in \{1,\ldots, \tilde k_j\},$
where the map $\Psi_{j,k}$ is as in \eqref{eq_psijk} and $(j,K)=G(j,k)$. This means that for every $k=1,\ldots, \tilde k_j$, the branch in $\U_j$ of the unstable manifold of $P_{2,n_k^j}$ does not coincide with the branch in $\U_j$ of the stable manifold of any $P_{2,n_l^j}.$ Hence, for every $(j,k)$, with $k=1,\ldots, \tilde k_j$, we are in scenario (b) or (c), see Figure \ref{fig:scenarios}. 

\begin{prop}\label{prop:trans_homoc}
There exists $P_{2,n_k^j}\subset \partial \mathcal{U}_j$ with a transverse homoclinic in $\U_j$.
\end{prop}

To prove Proposition \ref{prop:trans_homoc}, we  find $k\in \{1,\ldots,\tilde k_j\}$ so that $(j,k)\in \mathcal{C}$ is periodic under $G$. To find such $k$, choose any $(j,k_1)\in \mathcal{C}$,  let $(j,K_1) = G(j,k_1)$, and let $(j,k_2):=(j,K_1+1)$. Define the sequences $k_1,k_2,k_3,\ldots,$ and $K_1,K_2,K_3,\ldots,$ accordingly, as $(j,K_i)=G(j,k_i)$ and $k_{i+1}=K_i+1$ for every $i$. The sequence  $k_1,k_2,\ldots,$ is eventually periodic, so we may ignore the first elements and assume that $k_1,k_2,\ldots,k_N,k_1,\ldots$ is periodic with least period $N>0$. 



Next we show that for every $i$ the branch in $\U_j$ of the unstable manifold of $\P_{j,k_i}^-$ intersects transversely the branch in $\U_j$ of the stable manifold of $\P_{j,k_{i+m}}^-, \forall m\geq 2.$ In particular, the orbits $\P^-_{j,k_i}$ admit transverse homoclinics for every $i$.
Take a real-analytic curve 
$      
 \gamma \colon [0, \varepsilon)   \to  \Psi_{j, k_i}(C^u_{j, k_i}),  \varepsilon>0 \mbox{ small,}
$
  such that 
$
  \gamma(0) \in   \Psi_{j, k_i}(C^u_{j, k_i}) \cap  C^s_{j, K_i}$ and $ 
\gamma(t) \notin  \mathcal{D}^s_{j, K_i},   \forall t.
$
Let $\dot \gamma = \gamma \setminus \{ \gamma(0)\}$.
       
Recall that the mapping $\Psi_{j, k_{i+1}}$ is defined on a small neighborhood $\mathcal{N}( \mathcal{D}^u_{j, k_{i+1}})$ of  $\mathcal{D}^u_{j, k_{i+1}}.$ Let $\beta: [0,   \varepsilon) \to \mathcal{N}( \mathcal{D}^u_{j, k_{i+1}})$ 
be a real-analytic arc  intersecting 
  $C^u_{j, k_{i+1}}$ only at $t=0$ and satisfying
$
\beta(s) \in \mathcal{N}( \mathcal{D}^u_{j, k_{i+1}}) \setminus   \mathcal{D}^u_{j, k_{i+1}} ,  \forall s \in (0, \varepsilon),$ and $ 
 \Psi_{j, k_{i+1}} (\beta ) \subset C^s_{j, K_{i+1}}.
$
The following lemma is based on Conley's ideas \cite{C1}.

\begin{prop} \label{prop:extinfinittransverse} The real-analytic arc $l^{\rm ext}_{j, K_i}( \dot \gamma)$ is a spiral around  $C^u_{j,k_{i+1}}\subset \F^-_{j,k_{i+1}}$  intersecting $\beta$ transversely infinitely many times near  $C^u_{j k_{i+1}}.$  In particular, 
the branch in $\U_j$ of the unstable manifold of $\P^-_{j,k_i}$ transversely intersects the branch in $\U_j$ of the stable manifold of $\P^-_{j,k_{i+2}}$. 
\end{prop}

 \begin{proof} 
 As before, let $\tilde l_i \colon \R \times (0, \delta) \to \R \times (0, \delta)$ denote a lift of $l^{\rm ext}_{j,K_i}$. 
  In view of Lemma \ref{lem:infinitetwists},  $l^{\rm ext}_{j,K_i}(\dot \gamma)$ intersects $\beta$ infinitely many times. 
  The curve $\gamma$ is written as 
$
 \gamma  = \{ ( t_{\gamma }(r), r) \mid r \in (0, \varepsilon ) \}
$
for some real analytic function $t_{\gamma}\colon (0,\delta') \to \R \times (0,\varepsilon),$ which continuously extends over $[0,\delta')$.  We then obtain
$
\tilde l(\dot \gamma) = \{ (  \tilde t_\gamma(r)=t_{\gamma }(r) + \Delta t (r), r) \mid r \in ( 0, \varepsilon) \},
$
where $\Delta t(r) = g(r)-h(r)\ln r,$
and $g(r),h(r)$ are real-analytic functions  defined near $r=0$, with $h(0)>0$. See \eqref{eq:timevar}.
Hence
$
\tilde t_\gamma'(r) = t_\gamma'(r) +g'(r) -h'(r) \ln r -\frac{h(r)}{r} 
$
for every $r\in (0,\varepsilon)$. Due to the estimates in the proof of Lemma \ref{lem:infinitetwists}, we find  a constant $c_1>0$ and $\varepsilon_0 \in (0,\varepsilon ) $ such that
\begin{equation}\label{eq:anglesgamma}
\tilde t_\gamma' (r) <- \frac{c_1}{r}, \quad \forall r \in (0, \varepsilon_0 ).
\end{equation}

Now we proceed similarly with  $\beta.$ We may assume that $\beta(0)= (0,0)$. 
Hence there exists a continuous function  $t_\beta: [0, \varepsilon) \to [0,\varepsilon_0)$, which is real-analytic in $(0,\varepsilon)$ such that $t_\beta (0)= 0$ and 
$
  \beta = \{ ( t_\beta(r), r) \mid r \in [0, \varepsilon)\}.
$
Because of  real analyticity of $\beta$,  the  intersection  at $\beta(0)=(0,0)$  is  either transverse to $\{r=0\}$ or has  finite order tangency. Hence, we have either   $t_\beta \equiv 0$ or $t_\beta (r) = r^m g(r),$
 where $g$ is  real-analytic  on $r>0$ and satisfies $g(0) \neq 0$, and $m \in \N^*$ or $\frac{1}{m} \in \N^*$, depending on the way $\beta$ intersects $C^u_{j k_{i+1}}$ at $\beta(0)=(0,0)$. 
In all cases, we can choose $\lambda \in (0,1)$ and $c_2>0$ such that
\begin{equation}\label{eq:anglesbeta}
 t_{\beta }'(r) = m r^{m-1} g + r^m g' > - \frac{c_2}{ r^{1-\lambda}}
\end{equation}
for  every $r>0$ small enough.  
 We conclude   from \eqref{eq:anglesgamma} and  \eqref{eq:anglesbeta} that 
  $t_\gamma'(r) <  t_{\beta }'(r),$ for every $r>0$ sufficiently small. 
  \end{proof}

  Instead of $\dot \gamma$, we can now repeat the above construction using a small arc $\dot \gamma_1 \subset \Psi_{j,k_{i+1}} \circ l^{\rm ext}_{j,K_i}(\dot \gamma)$ in $\F^+_{j,K_{i+1}}\setminus \mathcal{D}^s_{j,K_{i+1}},$ which corresponds to a transverse intersection of the branch in $\U_j$ of the unstable manifold of $\P^-_{j,k_i}$ with the branch in $\U_j$ of the stable manifold of $\P^-_{j,k_{i+2}}.$ Proposition \ref{prop:extinfinittransverse} then provides a transverse intersection between the branch in $\U_j$ of the unstable manifold of $\P^-_{j,k_i}$ and the branch in $\U_j$ of the stable manifold of $\P^-_{j,k_{i+3}}.$ Using the periodicity of the sequence $k_1,k_2,\ldots,k_N,\ldots,$ which satisfies $(j,k_{i+1}-1)= G(j,k_i), \forall i$, we repeat this construction to find, for every $i$, a transverse homoclinic to $\P^-_{j,k_i}$. 
  
  The proof of Proposition \ref{prop:trans_homoc} is complete. We are ready to prove Theorem \ref{main2}-(ii).

  It is well-known that a transverse homoclinic forces positivity of topological entropy. For completeness, we include  the construction of an invariant subset  $\Lambda_j\subset \U_j$ whose dynamics contains the Bernoulli shift as a sub-system. We follow  Moser's book \cite{Moser}, see also \cite{BGRS}. 
  
  Let $(j,k)\in \mathcal{C}$ be such that $\P^-_{j,k}$ admits a transverse homoclinic orbit, as obtained in Proposition \ref{prop:trans_homoc}.  
 Consider a point $q_0\in C^s_{j,k-1}\subset \F^+_{j,k-1}$ which corresponds to a transverse homoclinic to $\P^+_{j,k-1}=\P^-_{j,k}$. This point lies in a small arc   $V_\infty \subset \F^+_{j,k-1}$ which is transverse to $C^s_{j,k-1}$ and is the image of an arc $V_\infty' \subset C^u_{j,k}$ under the forward flow. Denote by 
  $\mathcal{G}\colon \V'\subset \F^-_{j,k} \to \mathcal{V}\subset \F^+_{j,k-1},$ 
  the corresponding diffeomorphism, given by the first forward hitting point, so that $\mathcal{G}(V_\infty')=V_\infty$. Notice that $V'_\infty \subset \{r=0\}$, where $(t,r)$ are the real-analytic canonical coordinates near $C^u_{j,k}\subset \F^-_{j,k}$ as in Section \ref{sec:infinitetwist}.

  There exists a small arc $H_\infty \subset C^s_{j,k-1}$, starting from $q_0$ which, except for $q_0$, does not intersect  $\mathcal{G}(\V' \cap \mathcal{D}^u_{j,k})$. In this way, we find a small topological square $Q_0\subset \V$ near $q_0$, whose boundary is formed by the following arcs.
  \begin{itemize}
     \item[(i)] the two arcs  $H_\infty,V_\infty$ above.
     
     \item[(ii)] an arc $H_0$ starting from $\mathcal{G}(\V'\cap C^u_{j,k})$ which, except for this extreme point, is contained in the exterior of $\mathcal{D}^s_{j,k-1}$ and does not intersect  $\mathcal{G}(\V' \cap \mathcal{D}^u_{j,k})$.
     
     \item[(iii)] an arc $V_0$ starting from $C^s_{j,k-1}$ which, except for this extreme point, is contained in the exterior of $\mathcal{D}^s_{j,k-1}$ and does not intersect  $\mathcal{G}(\V' \cap \mathcal{D}^u_{j,k})$.
     \end{itemize}
     In local coordinates $(t,r)$ defined on a neighborhood of $C^s_{j,k-1}$ 
     we may assume that $H_0$ is a horizontal segment $r= r_0>0$ and $V_0$ is a horizontal shift of $V_\infty$.

      We can always find coordinates $(u,v)$ near $Q_0$ so that 
      $Q_0 \equiv [0,1] \times [0,1],
      V_\infty \equiv \{0\} \times [0,1],
      H_\infty \equiv [0,1] \times \{0\},
      V_0 \equiv \{1\} \times [0,1],
      H_0 \equiv [0,1] \times \{1\}.
      $
            We may assume that $q_0=(0,0)$ and that the mapping $(t,r) \mapsto (u,v)$ has the form
      \begin{equation}\label{eq:uv}
      (u,v) = M(t,r) + O(t^2+r^2),
      \end{equation}
      for some invertible linear mapping $M$.
  
  A vertical strip $V$ in $Q_0$ is a topological closed disk  whose boundary consists of horizontal arcs $h_0\subset H_0, h_\infty\subset H_\infty$ and two regular arcs  $v_1,v_2 \subset Q_0$ that connect $H_0$ and $H_\infty$. We assume that the arcs $v_1,v_2$ are transverse to the horizontals  $[0,1] \times  {\rm const}$. A horizontal strip in $Q_0$ is similarly defined.

Let 
$
\mathcal{P} := \mathcal{G} \circ l^{\rm ext}_{j,k-1}\colon\mathcal{H}_{-1}\subset  Q_0 \to Q_0,
$
where $\mathcal{H}_{-1}\subset Q_0$ is its domain of definition. Abusing the notation, we denote by
$
\mathcal{P}^{-1}:=(l^{\rm ext}_{j,k-1})^{-1} \circ \mathcal{G}^{-1}\colon   \mathcal{P}(\mathcal{H}_{-1})=:\mathcal{V}_1 \to Q_0,
$
 the  first return map under the backward flow. 
 

\begin{lem}\label{lem:strips}If $Q_0$ is (suitably chosen) sufficiently small, then $\mathcal{H}_{-1}$ is formed by countably many horizontal strips $H_n,n\in \N,$ in $Q_0$, monotonically accumulating on $H_\infty$ as $n\to \infty$. Moreover, $\mathcal{V}_1$ is formed by  countably many  vertical strips $V_n,n\in \N,$ in $Q_0$, monotonically accumulating on $V_\infty$ as $n\to \infty$. For every $n$,  $\mathcal{G} \circ l^{\rm ext}_{j,k-1}(H_n)  = V_n,$ and the vertical (horizontal) boundary components of $H_n$ are mapped to the respective vertical (horizontal) boundary components of $V_n$.
\end{lem}

\begin{proof}
Let $\tilde Q_0:=\mathcal{G}^{-1}(Q_0)\subset \F^-_{j,k}$. In coordinates $(t,r)$ on $\mathcal{F}^-_{j,k}$,  $\tilde Q_0$ is a square whose sides are 
$
\tilde H_\infty  := \mathcal{G}^{-1}(V_\infty), 
\tilde H_0  := \mathcal{G}^{-1}(V_0),
\tilde V_\infty  := \mathcal{G}^{-1}(H_\infty),
\tilde V_0  := \mathcal{G}^{-1}(H_0).
$
Notice that $\tilde H_\infty = V_\infty'.$
Recall that in coordinates $(t,r)$ on $\F^+_{j,k-1}$, we have $H_0\subset \{r=r_0\},$ for some $r_0>0$ small, and $V_0 = V_\infty + (t_0,0)$ for some $|t_0|>0$ small.
Taking $r_0$ and $|t_0|$ sufficiently small, we can assume that $\tilde H_0$ is arbitrarily $C^1$-close to $\tilde H_\infty\subset \{r=0\}$ and $\tilde V_0$ is arbitrarily $C^1$ close to $\tilde V_\infty$. Moreover, the image under $\mathcal{G}^{-1}$ of the curves $v_t = V_\infty + (t,0), |t|\leq |t_0|,$ are also $C^1$-close to $\tilde H_0$ and $\tilde H_\infty$.

Now observe that $l^{\rm ext}_{j,k-1}(V_\infty \setminus\{q_0\})$ transversely intersects $\tilde V_\infty$ and is arbitrarily $C^1$-close to the horizontal line $\{r=0\}$. This follows from Propositions \ref{prop:coincid_ln} and \ref{prop:extinfinittransverse}. Hence, if $r_0,|t_0|>0$ are sufficiently small and suitably chosen, then $l^{\rm ext}_{j,k-1}(Q_0 \setminus H_\infty)\cap \tilde Q_0$ consists of countably many horizontal strips $\tilde H_n,n\in \N,$ in $\tilde Q_0$. In particular, the image of such strips under $\mathcal{G}$ are vertical strips $V_n,n\in \N,$ in $Q_0$, ordered from right to left.

Finally, notice that a lift $\tilde l^{-1}$ of $(l^{\rm ext}_{j,k-1})^{-1}$ in coordinates $(t,r)$ has the form $\tilde l^{-1}(t,r) = (t-g(r)+h(r)\ln r,r),   \forall (t,r),$
for suitable real-analytic functions $g(r),h(r),$ defined near $r=0$, with $h(0)>0$. Therefore, $\tilde l^{-1}$ shares the same properties of a lift $\tilde l$ representing $l^{\rm ext}_{j,k-1}$ and we can apply Propositions \ref{prop:coincid_ln} and \ref{prop:extinfinittransverse} to the curve $\tilde V_\infty$ (or $\tilde V_0$) to conclude that $H_n:=(l^{\rm ext}_{j,k-1})^{-1}(\tilde H_n)$ are horizontal strips in $Q_0$ that accumulate on $H_\infty$, and are ordered from top to bottom. Moreover, the horizontal (vertical) boundary components of $\tilde H_n$ are mapped under $\tilde l^{-1}$ to vertical (horizontal) boundary components of $H_n$. The interchanging of vertical (horizontal) strips between $\tilde Q_0$ and  $Q_0$ under the mapping $\mathcal{G}$ finishes the proof of this proposition. 
\end{proof}

Each vertical strip $V_n \subset \mathcal{V}_1$ is regarded as a new square and its image under $\P$ consists of infinitely many vertical strips,  precisely one sub-strip of each strip in $\mathcal{V}_1$. Hence $\mathcal{V}_2:=\P(\mathcal{V}_1)\subset \mathcal{V}_1$   consists of countably many vertical strips, with countably many sub-strips of each strip in $\mathcal{V}_1$. Similarly, $\mathcal{V}_2=\P(\mathcal{H}_{-2})$, where $\mathcal{H}_{-2}\subset \mathcal{H}_{-1}$ consists of countably many horizontal strips,  with countably many sub-strips of each strip in  $\mathcal{H}_{-1}$. 

Repeating indefinitely this construction, we obtain sequences 
$\mathcal{V}_1 \supset \mathcal{V}_2 \supset \mathcal{V}_3  \supset \ldots$
and
$\mathcal{H}_{-1} \supset \mathcal{H}_{-2} \supset \mathcal{H}_{-3}\supset \ldots,$
so that $\mathcal{V}_{n+1}$ consists of countably many vertical strips, with countably many sub-strips of each strip in $\mathcal{V}_n$.  In the same way, $\mathcal{H}_{-n-1}$ consists of countably many horizontal strips, with countably many sub-strips of each strip in $\mathcal{H}_{-n}$. The image of $\mathcal{H}_{-n}$ under $\P$ coincides with $\mathcal{V}_n$.

The non-empty compact subsets of $Q_0$, defined as
$
\bar \Lambda_\mathcal{H} := \cap_{i=1}^{+\infty} {\rm closure}( \mathcal{H}_{-i} )$ and $
\bar \Lambda_\mathcal{V} := \cap_{i=1}^{+\infty}{\rm closure}( \mathcal{V}_i),
$
contain points whose forward and backward trajectories remain in the fixed component $\U_j$ of $S^3 \setminus \cup_{i=1}^l \S_i,$ respectively. 

The non-empty compact subset 
$
\bar \Lambda := \bar \Lambda_\mathcal{H} \cap \bar \Lambda_\mathcal{V} \subset Q_0,
$
 contains points whose entire trajectories remain in $\U_j$. It admits a symbolic dynamics as we outline below. Notice that some points in $\bar \Lambda$ are eventually mapped to $H_\infty$  where $\P$ is not well-defined. Similarly, $\P^{-1}$ is not well-defined on $V_\infty$.
 

Let $\bar \Sigma$ be the set of doubly infinite sequences $a=(a_n)_{n\in \Z}$ of the form $$\ldots,\infty,\infty,a_{l_0},\ldots, a_{-2},a_{-1},a_0,a_1,a_2,\ldots,a_{r_0},\infty,\infty,\ldots,$$ where $a_n$ is a positive integer for every $l_0 \leq n \leq r_0,$ and $a_n = \infty$ for $n<l_0$ and for $n>r_0$. Here, $-\infty \leq l_0 \leq 0 \leq r_0 \leq +\infty.$ The usual shift in $\bar \Sigma$ is given by  $\sigma(a)_n:=a_{n+1}, \forall n,$ and is defined only if $a_0 \neq \infty.$ Similarly, one has the inverse~$\sigma^{-1}.$ 

The conjugation $h\colon (\bar \Lambda,\P) \to (\bar \Sigma, \sigma)$ maps each $x\in \bar \Lambda$ to $(a_n)_{n\in \Z}$ satisfying $\P^n(x) \in V_{a_n},  -\infty \leq l_0\leq n \leq r_0 \leq +\infty.$ If $\P^{r_0}(x) \in H_\infty$ for some $0<r_0<+\infty$, then  $\P^{r_0+1}(x)$ is not well-defined. In this case, we define $h(x)_n := \infty$ for every $n> r_0$. In the same way, if $\P^{l_0}(x)\in V_\infty$ for some $-\infty<l_0<0$, then  $\P^{l_0-1}(x)$ is not well-defined, and we define $h(x)_n:=\infty$ for every $n<l_0$. We also define $h(q_0)=(\ldots,\infty,\infty,\infty,\ldots).$ Hence
$
h \circ \P = \sigma \circ h
$
wherever  defined. By Lemma \ref{lem:strips} and the reasoning just after it, we see that $h(\bar \Lambda)=\bar \Sigma.$ 

The subset $\Sigma\subset \bar \Sigma$ of sequences whose entries are positive integers, that is $l_0=-\infty$ and $r_0=+\infty$, corresponds to an invariant subset $\Lambda \subset \bar \Lambda$,  all whose iterates of $\P$ (positive and negative)  are defined. Our estimates below show that $\Lambda$ is in one-to-one correspondence with $\Sigma.$

 

 Now we study the hyperbolic structure of  the invariant set $\Lambda$. To do that, we first compute the derivative of $\P=\mathcal{G} \circ l^{\rm ext}_{j,k-1}$ in coordinates $(t,r)$. Since the lift $\tilde l$ representing $l^{\rm ext}_{j,k-1}$ has the form
$
\tilde l(t,r) = (t+g(r) - h(r)\ln r,r),  \forall (t,r),
$
for suitable real-analytic functions $g(r)$ and $h(r)$ defined  near $r=0$, with $h(0)>0$, we find
$$
d\tilde l(t,r) \equiv \left(\begin{array}{cc} 1 & g'(r)-h'(r) - \frac{h(r)}{r}\\ 0 & 1 \end{array}\right), \quad \forall (t,r).
$$
Observe that the dominating term in $d\tilde l(t,r)$ is
$
L(r) := g'(r) - h'(r) - h(r)/r \to -\infty$ as  $r \to 0.
$
A lift of  $\mathcal{G}$ is represented by  $\tilde \psi=\tilde \psi(t,r)$.
We may assume that $q_0\equiv (0,0)$ and $\tilde q_0=\mathcal{G}^{-1}(q_0)\equiv (0,0)$. Due to the transversality of the homoclinic trajectory, we may also assume that 
$$
d\tilde \psi(t,r) \equiv \left(\begin{array}{cc}A(t,r) & B(t,r)\\ C(t,r) & D(t,r) \end{array} \right)\to \left(\begin{array}{cc} A_0 & B_0\\ C_0 & D_0 \end{array}\right)=d\tilde \psi(0,0) \ \ \mbox{ as } \ \ (t,r) \to (0,0),
$$
where $A_0D_0-B_0C_0=1$ and $C<0$. Notice that  $(A_0,C_0)^T = d\tilde \psi(0,0)\cdot (1,0)^T$ is tangent to $V_\infty$ at $(0,0)$.
Hence, $d\mathcal{P}  = d\mathcal{G} \cdot  d l^{\rm ext}_{j,k-1}$ is represented by
$$
\begin{aligned}
d\tilde \psi \cdot d\tilde l 
& \equiv \left(\begin{array}{cc}A(t,r) & A(t,r)L(r) +B(t,r)\\ C(t,r) & C(t,r)L(r) + D(t,r)  \end{array}\right),
\end{aligned}
$$
whose eigenvalues are
$
\lambda_{\pm} = \frac{\rm tr}{2} \pm \frac{\sqrt{{\rm tr}^2-4}}{2},
$
where 
$
{\rm tr} := A(t,r) + D(t,r) +C(t,r)L(r) \to +\infty \quad \mbox{as} \quad r \to 0.
$
Hence $\lambda_+ \to +\infty$ and $\lambda_- \to 0^+$ as $r \to 0$. The respective eigenspaces $E_+(t,r)$ and $E_-(t,r)$ converge to 
$
E_+:=\R(A_0,C_0)^T$ and $E_-:= \R(1,0)^T
$  
as $r \to 0$. Notice that $E_+$ is tangent to $V_\infty$ and $E_-$ is tangent to $H_\infty$ at $q_0\equiv (0,0)$. 
In coordinates $(u,v)$, see \eqref{eq:uv}, $E_+$ and $E_-$ converge to $\R(0,1)^T$ and $\R(1,0)^T,$ respectively, as $(u,v) \to (0,0)$.

Fixing $0<\mu<\frac{1}{2}$ and taking $Q_0$ sufficiently small, we conclude that the  cones
$
\zeta_{(u,v)}:=\{|\delta_v| < \mu |\delta_u|\}$  and $\eta_{(u,v)}:=\{|\delta_u| < \mu |\delta_v|\},
$
with $\delta_u \frac{\partial}{\partial u} + \delta_v \frac{\partial}{\partial v} \in T_{(u,v)}Q_0 \equiv \R^2,$ satisfy
$
d\P \cdot \eta_{(u,v)} \subset \eta_{\P(u,v)},$ $ d\P^{-1} \cdot \zeta_{(u,v)} \subset \zeta_{\P^{-1}(u,v)},$
$$
\begin{aligned}
|d\P \cdot \eta| > \mu^{-1}|\eta|, \ \ \ \forall \eta\in \eta_{(u,v)} \ \ \ \mbox{ and } \ \ \ 
|d\P^{-1} \cdot \zeta| > \mu^{-1} |\zeta|, \ \ \ \forall \zeta \in \zeta_{(u,v)}.
\end{aligned}
$$

As proved in \cite[chapter III]{Moser}, the mapping $h \colon  \bar \Lambda \to \bar \Sigma$ is a conjugation between $\P$ and $\sigma$. In particular, the topological entropy of $\P$ is positive. The trajectories through $\Lambda$ form an invariant subset $\Lambda_j \subset \U_j$, where the Reeb flow has positive topological entropy.  This completes the proof of Theorem \ref{main2}-(ii).

  \bibliographystyle{abbrv}
\bibliography{mybibfile}

\begin{thebibliography}{10}

\bibitem{BGRS}
P.~Bernard, C.~Grotta~Ragazzo, and P.~A. Santoro Salom\~{a}o.
\newblock Homoclinic orbits near saddle-center fixed points of {H}amiltonian
  systems with two degrees of freedom.
\newblock {\em Ast\'{e}risque}, 286:xviii--xix, 151--165, 2003.
\newblock Geometric methods in dynamics. I.

\bibitem{BEHWZ03}
F.~Bourgeois, Y.~Eliashberg, H.~Hofer, K.~Wysocki, and E.~Zehnder.
\newblock Compactness results in symplectic field theory.
\newblock {\em Geom. Topol.}, 7:799--888, 2003.

\bibitem{Brouwer}
L.~E.~J. Brouwer.
\newblock Beweis des ebenen {T}ranslationssatzes.
\newblock {\em Math. Ann.}, 72(1):37--54, 1912.

\bibitem{C1}
C.~Conley.
\newblock Twist mappings, linking, analyticity, and periodic solutions which
  pass close to an unstable periodic solution.
\newblock {\em Topological Dynamics}, pages 129--153, 1968.

\bibitem{mechanical}
N.~V. de~Paulo, S.~Kim, P.~A.~S. Salom\~ao, and A.~Schneider.
\newblock Transverse foliations for two-degree-of-freedom mechanical systems.
\newblock {\em In preparation}.

\bibitem{dPS1}
N.~V. de~Paulo and P.~A.~S. Salom\~{a}o.
\newblock Systems of transversal sections near critical energy levels of
  {H}amiltonian systems in {$\mathbb R^4$}.
\newblock {\em Mem. Amer. Math. Soc.}, 252(1202):v+105, 2018.

\bibitem{dPS2}
N.~V. de~Paulo and P.~A.~S. Salom\~{a}o.
\newblock On the multiplicity of periodic orbits and homoclinics near critical
  energy levels of {H}amiltonian systems in {$\mathbb R^4$}.
\newblock {\em Trans. Amer. Math. Soc.}, 372(2):859--887, 2019.

\bibitem{dPS_SPJ}
N.~V. de~Paulo and P.~A.~S. Salom\~ao.
\newblock Reeb flows, pseudo-holomorphic curves and transverse foliations.
\newblock {\em S\~ao Paulo Journal of Mathematical Sciences}, pages 1--26,
  2022.

\bibitem{Drag}
D.~L. Dragnev.
\newblock Fredholm theory and transversality for noncompact pseudoholomorphic
  maps in symplectizations.
\newblock {\em Comm. Pure Appl. Math.}, 57(6):726--763, 2004.

\bibitem{Franks88}
J.~Franks.
\newblock Generalizations of the {P}oincar\'{e}-{B}irkhoff theorem.
\newblock {\em Ann. of Math. (2)}, 128(1):139--151, 1988.

\bibitem{Franks}
J.~Franks.
\newblock Area preserving homeomorphisms of open surfaces of genus zero.
\newblock {\em New York J. Math.}, 2:1--19, electronic, 1996.

\bibitem{FvKbook}
U.~Frauenfelder and O.~van Koert.
\newblock {\em The restricted three-body problem and holomorphic curves}.
\newblock Pathways in Mathematics. Birkh\"{a}user/Springer, Cham, 2018.

\bibitem{Hofer93}
H.~Hofer.
\newblock Pseudoholomorphic curves in symplectizations with applications to the
  {W}einstein conjecture in dimension three.
\newblock {\em Invent. Math.}, 114(3):515--563, 1993.

\bibitem{props2}
H.~Hofer, K.~Wysocki, and E.~Zehnder.
\newblock Properties of pseudo-holomorphic curves in symplectisations. {II}.
  {E}mbedding controls and algebraic invariants.
\newblock {\em Geom. Funct. Anal.}, 5(2):270--328, 1995.

\bibitem{props1}
H.~Hofer, K.~Wysocki, and E.~Zehnder.
\newblock Properties of pseudoholomorphic curves in symplectisations. {I}.
  {A}symptotics.
\newblock {\em Ann. Inst. H. Poincar\'{e} Anal. Non Lin\'{e}aire},
  13(3):337--379, 1996.

\bibitem{convex}
H.~Hofer, K.~Wysocki, and E.~Zehnder.
\newblock The dynamics on three-dimensional strictly convex energy surfaces.
\newblock {\em Ann. of Math. (2)}, 148(1):197--289, 1998.

\bibitem{props3}
H.~Hofer, K.~Wysocki, and E.~Zehnder.
\newblock Properties of pseudoholomorphic curves in symplectizations. {III}.
  {F}redholm theory.
\newblock In {\em Topics in nonlinear analysis}, volume~35 of {\em Progr.
  Nonlinear Differential Equations Appl.}, pages 381--475. Birkh\"{a}user,
  Basel, 1999.

\bibitem{small}
H.~Hofer, K.~Wysocki, and E.~Zehnder.
\newblock Finite energy cylinders of small area.
\newblock {\em Ergodic Theory Dynam. Systems}, 22(5):1451--1486, 2002.

\bibitem{fols}
H.~Hofer, K.~Wysocki, and E.~Zehnder.
\newblock Finite energy foliations of tight three-spheres and {H}amiltonian
  dynamics.
\newblock {\em Ann. of Math. (2)}, 157(1):125--255, 2003.

\bibitem{HS_ICM}
U.~L. Hryniewicz and P.~A.~S. Salomão.
\newblock Global surfaces of section for {R}eeb flows in dimension three and
  beyond.
\newblock In {\em Proceedings of the {I}nternational {C}ongress of
  {M}athematicians---{R}io de {J}aneiro 2018. {V}ol. {II}. {I}nvited lectures},
  pages 941--967. World Sci. Publ., Hackensack, NJ, 2018.

\bibitem{Katokentropy}
A.~Katok.
\newblock Lyapunov exponents, entropy and periodic orbits for diffeomorphisms.
\newblock {\em Inst. Hautes \'{E}tudes Sci. Publ. Math.}, 51:137--173, 1980.

\bibitem{HKentropy}
A.~Katok and B.~Hasselblatt.
\newblock {\em Introduction to the modern theory of dynamical systems},
  volume~54 of {\em Encyclopedia of Mathematics and its Applications}.
\newblock Cambridge University Press, Cambridge, 1995.
\newblock With a supplementary chapter by Katok and Leonardo Mendoza.

\bibitem{lemos}
C.~Lemos~de Oliveira.
\newblock $3-2-1$ foliations for {R}eeb flows on the tight 3-sphere.
\newblock {\em Trans. Amer. Math. Soc.}, 377:3983--4053, 2024.

\bibitem{LMSimo}
J.~Llibre, R.~Mart\'{\i}nez, and C.~Sim\'{o}.
\newblock Tranversality of the invariant manifolds associated to the {L}yapunov
  family of periodic orbits near {$L_2$} in the restricted three-body problem.
\newblock {\em J. Differential Equations}, 58(1):104--156, 1985.

\bibitem{MSbook}
D.~McDuff and D.~Salamon.
\newblock {\em {$J$}-holomorphic curves and symplectic topology}, volume~52 of
  {\em American Mathematical Society Colloquium Publications}.
\newblock American Mathematical Society, Providence, RI, 2004.

\bibitem{Moser56}
J.~Moser.
\newblock The analytic invariants of an area-preserving mapping near a
  hyperbolic fixed point.
\newblock {\em Comm. Pure Appl. Math.}, 9:673--692, 1956.

\bibitem{Moser}
J.~Moser.
\newblock {\em Stable and random motions in dynamical systems}.
\newblock Princeton University Press, Princeton, N. J.; University of Tokyo
  Press, Tokyo, 1973.
\newblock With special emphasis on celestial mechanics, Hermann Weyl Lectures,
  the Institute for Advanced Study, Princeton, N. J, Annals of Mathematics
  Studies, No. 77.

\bibitem{Nel13}
J.~Nelson.
\newblock {\em Applications of automatic transversality in contact homology}.
\newblock ProQuest LLC, Ann Arbor, MI, 2013.
\newblock Thesis (Ph.D.)--The University of Wisconsin - Madison.

\bibitem{Paternainentropy}
G.~P. Paternain.
\newblock {\em Geodesic flows}, volume 180 of {\em Progress in Mathematics}.
\newblock Birkh\"{a}user Boston, Inc., Boston, MA, 1999.

\bibitem{Sa1}
P.~A.~S. Salom\~{a}o.
\newblock Convex energy levels of {H}amiltonian systems.
\newblock {\em Qual. Theory Dyn. Syst.}, 4(2):439--457 (2004), 2003.

\bibitem{Si2}
R.~Siefring.
\newblock Intersection theory of punctured pseudoholomorphic curves.
\newblock {\em Geom. Topol.}, 15(4):2351--2457, 2011.

\bibitem{Si3}
R.~Siefring.
\newblock Finite-energy pseudoholomorphic planes with multiple asymptotic
  limits.
\newblock {\em Math. Ann.}, 368(1-2):367--390, 2017.

\bibitem{Wendl1}
C.~Wendl.
\newblock Automatic transversality and orbifolds of punctured holomorphic
  curves in dimension four.
\newblock {\em Comment. Math. Helv.}, 85(2):347--407, 2010.

\bibitem{WendlSFT}
C.~Wendl.
\newblock Lectures on symplectic field theory.
\newblock {\em arXiv preprint}, 2016.

\end{thebibliography}

\end{document}